\documentclass{amsart}
\usepackage{amsmath}
\usepackage{amssymb}
\usepackage{amsfonts}
\usepackage{graphicx}
\usepackage{comment}
\usepackage{subfigure}
\usepackage{epsfig}
\usepackage{epstopdf}
\usepackage{hyperref}

\newtheorem{theorem}{Theorem}[section]
\newtheorem{corollary}[theorem]{Corollary}

\newtheorem{lemma}[theorem]{Lemma}
\newtheorem{proposition}[theorem]{Proposition}
\newtheorem{definition}[theorem]{Definition}
\newtheorem{example}[theorem]{Example}
\newtheorem{notation}[theorem]{Notation}
\newtheorem{remark}[theorem]{Remark}
\newcommand{\BbR}{\mathbb{R}}

\begin{document}
\title[Local dimensions]{Local dimensions of measures of finite type on the
Torus.}
\author[K.~E.~Hare, K.~G.~Hare, K.~R. Matthews]{Kathryn E. Hare, Kevin G.
Hare, Kevin R. Matthews}
\thanks{Research of K. G. Hare was supported by NSERC Grant RGPIN-2014-03154}
\thanks{Research of K. E. Hare and K. R. Matthews was supported by NSERC
Grant 44597-2011}
\address{Dept. of Pure Mathematics, University of Waterloo, Waterloo, Ont.,
N2L 3G1, Canada}
\email{kehare@uwaterloo.ca}
\address{Dept. of Pure Mathematics, University of Waterloo, Waterloo, Ont.,
N2L 3G1, Canada}
\email{kghare@uwaterloo.ca}
\address{Dept. of Pure Mathematics, University of Waterloo, Waterloo, Ont.,
N2L 3G1, Canada}
\email{kevinmatthews12@hotmail.com}

\begin{abstract}
The structure of the set of local dimensions of a self-similar measure has
been studied by numerous mathematicians, initially for measures that satisfy
the open set condition and, more recently, for measures on $\mathbb{R}$ that
are of finite type.

In this paper, our focus is on finite type measures defined on the torus,
the quotient space $\mathbb{R}\backslash \mathbb{Z}$. We give criteria which
ensures that the set of local dimensions of the measure taken over points in
special classes generates an interval. We construct a non-trivial example of
a measure on the torus that admits an isolated point in its set of local
dimensions. We prove that the set of local dimensions for a finite type
measure that is the quotient of a self-similar measure satisfying the strict
separation condition is an interval. We show that sufficiently many
convolutions of Cantor-like measures on the torus never admit an isolated
point in their set of local dimensions, in stark contrast to such measures
on $\mathbb{R}$. Further, we give a family of Cantor-like measures on the
torus where the set of local dimensions is a strict subset of the set of
local dimensions, excluding the isolated point, of the corresponding
measures on $\mathbb{R}$. 
\end{abstract}

\maketitle

\section{Introduction}

\label{sec:intro} The local dimension of a probability measure $\mu $
defined on a metric space, at a point $x$ in the support of $\mu ,$ is the
number%
\begin{equation*}
\dim _{loc}\mu (x)=\lim_{r\rightarrow 0}\frac{\log \mu (B(x,r))}{\log r}.
\end{equation*}%
It is of interest to determine the local dimensions of a given measure as
these numbers quantify the local concentration of the measure. For
self-similar measures that satisfy the open set condition (OSC), it is well
known that the set of local dimensions is a closed interval whose endpoints
are given by a simple formula.

When the OSC condition fails, the situation is more complicated, less
understood and can be quite different. Indeed, in \cite{HL}, Hu and Lau
discovered that when $\mu $ is the three-fold convolution of the classic
middle-third Cantor measure on $\mathbb{R}$, then there is an isolated point
in the set of local dimensions, namely at $x=0,3$. This fact
was later established for other `overlapping' Cantor-like measures in \cite%
{BHM, FL, FHJ,Sh}.

These Cantor-like measures are special examples of equicontractive
self-similar measures of finite type on $\mathbb{R}$, a notion introduced by
Ngai and Wang \cite{NW} that is weaker than the OSC, but stronger than the
weak separation condition \cite{Ng}. Such measures have a very structured
geometry, which makes them more tractable than arbitrary self-similar
measures. In \cite{F3,F1,F2}, Feng conducted a detailed study of
equicontractive, self-similar measures of finite type on $\mathbb{R}$,
focussing primarily on Bernoulli convolutions with contraction factors the
inverse of a simple Pisot number. In \cite{HHM, HHN}, the authors, together
with Ng, showed that if $\mu $ is a measure of finite type on $\mathbb{R}$
that arises from regular probabilities and has full interval support, then
the set of local dimensions of $\mu $ at the points in any positive loop
class is a closed interval and the local dimensions at periodic points in
the loop class are dense within this interval. See Section \ref{basicdef}
for a definition of regular probabilities. Moreover, there is
always a distinguished positive loop class, known as the essential class,
which has full measure and is often all but the endpoints of the support of
the measure. This is the situation for the Bernoulli convolutions with
contraction a simple Pisot inverse and the 3-fold convolution of the
middle-third Cantor measure, for instance.
If the self-similar measure does not
arise from regular probabilities, it is still true that the set of local
dimensions at points in the interior of the essential class is a closed
interval.

In this paper, we introduce the notion of finite type for measures on the
torus $\mathbb{T}$ (the quotient space $\mathbb{R}\backslash \mathbb{Z}$)
that are quotients of equicontractive, self-similar measures on $\mathbb{R}.$
Examples of such measures include convolutions on the torus of Cantor
measures or Bernoulli convolutions with contraction factor the inverse of a
Pisot number. These measures are the quotients of the convolutions on $%
\mathbb{R}$ of the (initial) Cantor measures or Bernoulli convolutions.

We develop a general method for calculating the local dimensions of finite
type measures on $\mathbb{T}$ and obtain a simple formula for the local
dimensions at the periodic points. With these tools, the same techniques as
used for finite type measures on $\mathbb{R}$ show that if the self-similar
measure is associated with regular probabilities and the quotient measure
has support $\mathbb{T}$, then the set of local dimensions of the quotient
measure at the points in any positive loop class is an interval. If we do
not assume regular probabilities, under a mild technical assumption it is
still true that the set of local dimensions at points in the interior of any
positive essential class is an interval. As with finite type measures on $%
\mathbb{R}$, in either case this interval is the closure of the local
dimensions at the periodic points in the class. 
However, in contrast to the case for finite type measures on 
$\mathbb{R}$, the essential class for measures of finite type on $\mathbb{T}$
need not be unique or of positive type.

We use these results to prove that if a self-similar measure of finite type
on the torus is associated with an IFS that satisfies the strong separation
condition, then the set of local dimensions of the quotient measure is not
only an interval, but coincides with the set of local dimensions of the
original measure. We also give the first (as far as we are aware)
non-trivial example of a quotient measure on $\mathbb{T}$ whose set of local
dimensions admits an isolated point.

In \cite{BHM}, it was shown that the sets of local dimensions for quotients
of $k$-fold convolutions of Cantor measures with contraction factor $1/d$
are intervals whenever $k\geq d$. Although these quotient measures do not
have an essential class of positive type, we are able to modify our general
approach to give a new proof of this fact. Moreover, we extend this result
to what we call complete quotient Cantor-like measures and also prove that
the set of local dimensions is the closure of the set of local dimensions of
periodic points.

In \cite{FHJ} it was explicitly shown that set of local dimensions of the $3$%
-fold convolution of the Cantor measure with contraction factor $1/3$ on the
torus is a strict subset of the set of local dimensions at points in the
essential class, $(0,3),$ of the corresponding measure on $%
\mathbb{R}$. The authors also comment that a
similar proof can be used for all $d$-fold convolutions of the Cantor
measure with contraction factor $1/d$ for $d\geq 3$. In the last section we
extend this result to show that for all $k\geq 0$ and all $d$ sufficiently
large the $(k+d)$-fold convolution of the Cantor measure with contraction
factor $1/d$ shares this property.

\section{Finite type quotient measures on $\mathbb{T}$}

\subsection{Basic definitions and notation\label{basicdef}}

Assume 
\begin{equation*}
\{S_{j}(x)=\varrho x+d_{j}:j=0,\dots ,k\}
\end{equation*}%
is an iterated function system (IFS) of equicontractive similarities on $%
\mathbb{R}$ and let $p_{j}$ for $j\in \mathcal{A}=\{0,...,k\}$ denote
probabilities, meaning $p_{j}>0$ and $\sum p_{j}=1$. By the associated
self-similar measure $\mu $, we mean the unique measure satisfying the
identity 
\begin{equation}
\mu =\sum_{j=0}^{k}p_{j}\mu \circ S_{j}^{-1}.  \label{ssmeas}
\end{equation}%
Its support is the associated self-similar set $K\subseteq \mathbb{R}$, the
unique non-empty compact set satisfying $K=\bigcup\limits_{j=0}^{k}S_{j}(K).$

There is no loss of generality in assuming $d_{0}<d_{1}<\cdot \cdot \cdot
<d_{k}$. If $p_{0}=p_{k}=\min p_{j}$, then the probabilities are referred to
as \textit{regular}.

We first recall what it means to say such an equicontractive IFS or
self-similar measure on $\mathbb{R}$ is of finite type.

\begin{definition}
The iterated function system, $\{S_{j}(x)=\varrho x+d_{j}:j=0,\dots ,k\}$,
is said to be of \textbf{finite type} if there is a finite set $F\subseteq 
\mathbb{R}$ such that for each positive integer $n$ and any two sets of
indices $\sigma ,\tau $ $\in \mathcal{A}^{n}$, either 
\begin{equation*}
\varrho ^{-n}\left\vert S_{\sigma }(0)-S_{\tau }(0)\right\vert >\delta \text{
or }\varrho ^{-n}(S_{\sigma }(0)-S_{\tau }(0))\in F,
\end{equation*}%
where $\delta =(1-\varrho )^{-1}(\max d_{j}-\min d_{j})$ is the diameter of $%
K$. If $\{S_{j}\}$ is of finite type and $\mu $ is an associated
self-similar measure, we also say that $\mu $ is of finite type.
\end{definition}

Hereafter, we will refer to this notion as `finite type on $\mathbb{R}$' to
distinguish it from the notion of `finite type on the torus', which will be
the focus of this paper and will be defined shortly.

Measures which satisfy the open set condition are of finite type on $\mathbb{%
R}$ and measures of finite type on $\mathbb{R}$ satisfy the weak open set
condition \cite{Ng}. The structure of finite type measures and aspects of
their multi-fractal analysis is explained in detail in \cite{F3,F1,F2, HHM,
HHN}.

\begin{example}
The self-similar measures associated with the IFS $\{S_{0}(x)=\varrho x$, $%
S_{1}(x)=\varrho x+(1-\varrho )\}$ and probabilities $p_{0}=p_{1}=1/2$ are
known as (uniform) Bernoulli convolutions\textbf{\ }when $\varrho >1/2$ and
Cantor measures when $\varrho <1/2$. Cantor measures satisfy the OSC.
Bernoulli convolutions do not, but when $\varrho $ is the inverse of a Pisot
number\footnote{%
A Pisot number is an algebraic integer greater than one, all of whose Galois
conjugates are strictly less than one in modulus.}the measures are of finite
type.

Given two measures $\mu ,$ $\nu $ on $\mathbb{R}$, by their convolution, $%
\mu \ast \nu ,$ we mean the measure on $\mathbb{R}$ defined by%
\begin{equation*}
\mu \ast \nu (E)=\int_{\mathbb{R}}\mu (E-x)\text{ }d\nu (x)\text{ for all
Borel sets }E\subseteq \mathbb{R}\text{.}
\end{equation*}%
The $k$-fold convolutions of the Bernoulli convolutions or Cantor measures
are (also) the self-similar measures associated with the IFS $%
\{S_{j}(x)=\varrho x$ $+j(1-\varrho ):$ $j=0,1,...,k\}$ and probabilities $%
p_{j}=\binom{k}{j}2^{-k}$. These measures do not satisfy the OSC when $k\geq
1/\varrho $, but are of finite type on $\mathbb{R}$ whenever $\varrho $ is
the inverse of a Pisot number, such as a positive integer \cite{NW}.
\end{example}

By the torus, $\mathbb{T}$, we mean the quotient group, $\mathbb{R}%
\backslash \mathbb{Z}$. We let $\pi $ denote the canonical quotient map and
denote the usual metric on $\mathbb{T}$ by $d_{\mathbb{T}}$. Sometimes it is
convenient to identify the torus as the group $[0,1)$ under addition mod $1$.

Given a measure $\mu $ on $\mathbb{R}$, we let $\mu _{\pi }$ be the quotient
measure defined by $\mu _{\pi }(E)=\mu (\pi ^{-1}(E))$ for any Borel set $%
E\subseteq \mathbb{T}$. Of course, if $\mu $ has support $K$, then $\mu
_{\pi }$ has support $\pi (K)$. If $\mu $ has support contained in $[0,1]$,
the only difference between the local dimensions of $\mu $ and $\mu _{\pi }$
is that (identifying $\mathbb{T}$ with $[0,1)$) $\dim _{loc}\mu _{\pi
}(0)=\min (\dim _{loc}\mu (0),\dim _{loc}\mu (1))$, so this situation is
trivial.

\begin{definition}
\label{defn:finite type T} The iterated function system, $\{S_{j}(x)=\varrho
x+d_{j}:j=0,\dots ,k\},$ defined on $\mathbb{R}$, is said to be of \textbf{%
finite type on the torus $\mathbb{T}$} if there is a finite set $\Lambda
\subseteq \mathbb{R}$ such that for each positive integer $n$ and any two
sets of indices $\sigma ,\tau \in \mathcal{A}^{n}$, either 
\begin{equation*}
d_{\mathbb{T}}(\pi (S_{\sigma }(0)),\pi (S_{\tau }(0)))>\varrho ^{n}\delta 
\text{ or }d_{\mathbb{T}}(\pi (S_{\sigma }(0)),\pi (S_{\tau }(0)))\in \text{ 
}\varrho ^{n}\Lambda ,
\end{equation*}%
where $\delta =(1-\varrho )^{-1}(\max d_{j}-\min d_{j})$ is the diameter of
the self-similar set $K\subseteq $ $\mathbb{R}$. If the IFS is of finite
type on the torus and $\mu _{\pi }$ is the quotient measure of a
self-similar measure $\mu $ associated with the IFS, we also say that $\mu
_{\pi }$ is of finite type on\textbf{\ }$\mathbb{T}$.
\end{definition}

Later in this section we will show that quotients of the $k$-fold
convolutions of Bernoulli convolutions or Cantor measures with contraction
factors the inverse of a Pisot number are of finite type on $\mathbb{T}$. We
will also give an example of a measure which is of finite type on $\mathbb{R}
$, but whose quotient is not of finite type on $\mathbb{T}$.

A translation argument shows there is no loss of generality in assuming $%
d_{0}=0$. We will also suppose that the diameter of the self-similar set $K$
is an integer $\delta $. We will leave it to the reader to consider what
modifications to our arguments need to be made when the diameter is not an
integer.

First, we introduce the important notions of quotient net intervals,
neighbours and characteristic vectors. These are motivated by the analogous
ideas for measures of finite type on $\mathbb{R}$.

For each integer $n\geq 1$, let $0=$ $h_{1}<h_{2}<\dots <h_{s_{n}-1}<1$ be
the collection of elements $\{\pi (S_{\sigma }(0)),\pi (S_{\sigma }(\delta
)):\sigma \in \mathcal{A}^{n}\}\mathbb{\ }$in the torus which here it is
convenient to identify with $[0,1).$ Put $h_{s_{n}}=1$ and let 
\begin{equation*}
\mathcal{F}_{n}^{(\pi )}=\{[h_{j},h_{j+1}]:1\leq j\leq s_{n}-1,\pi
^{-1}(h_{j},h_{j+1})\bigcap K\neq \emptyset \}.
\end{equation*}%
Put $\mathcal{F}_{0}^{(\pi )}=\{[0,1]\}$. Elements in $\mathcal{F}_{n}^{(\pi
)}$ are called the \textit{quotient net intervals of level }$n$. In what
follows, we will often omit the adjective `quotient' if this is clear from
the context. We note that these are intervals in $\mathbb{R}$ that are
contained in $[0,1]$. Each quotient net interval $\Delta ,$ of level $n\geq
1,$ is contained in a unique quotient net interval $\widehat{\Delta }$ of
level $n-1,$ called its \textit{parent}.

There is a similar notion of net intervals for measures on $\mathbb{R}$. It
is worth observing that an interval in $\mathcal{F}_{n}^{(\pi )}$ may not
correspond directly to a net interval of level $n$ in $\mathbb{R}$ because $%
\pi (S_{\sigma }(\epsilon _{1}))$ and $\pi (S_{\tau }(\epsilon _{2}))$ may
be adjacent in $\mathbb{T}$ without being adjacent in $\mathbb{R}$.

Let $\Delta =[a,b]\in \mathcal{F}_{n}^{(\pi )}$, $n\geq 1$ and for $%
l=0,...,\delta -1$ let $\Delta ^{(l)}=\Delta +l$. As $\pi ^{-1}(a,b)\cap K$
is not empty there must be some $l$ such that $int(\Delta ^{(l)})\bigcap
K=(a+l,b+l)\bigcap K$ is not empty. Put 
\begin{equation*}
\{a_{1},\dots ,a_{m}\}=\{\varrho ^{-n}(a-S_{\sigma }(0)+l):\sigma \in 
\mathcal{A}^{n}\ ,l\in \mathbb{N}\text{, }int(\Delta ^{(l)})\bigcap
S_{\sigma }(K)\neq \emptyset \},
\end{equation*}%
where we assume the real numbers $\{a_{i}\}$ satisfy $a_{1}<a_{2}<\cdot
\cdot \cdot <a_{m}$. By the \textit{quotient neighbour set} of $\Delta $ we
mean the tuple 
\begin{equation*}
V_{n}^{(\pi )}(\Delta )=(a_{1},\dots ,a_{m}).
\end{equation*}%
The normalized length of $\Delta =[a,b]\in \mathcal{F}_{n}^{(\pi )}$ is
denoted%
\begin{equation*}
\ell _{n}(\Delta )=\varrho ^{-n}|b-a|.
\end{equation*}%
As $\left\vert S_{\sigma }(0)-S_{\sigma }(\delta )\right\vert =\rho
^{n}\delta $ when $\sigma \in \mathcal{A}^{n}$, it follows that for large
enough $n$, $|b-a|=d_{\mathbb{T}}(a,b)$.

Suppose $\widehat{\Delta }$ is the parent of $\Delta \in \mathcal{F}%
_{n}^{(\pi )}$ and $\Delta _{1},...,\Delta _{j}$ are the quotient net
intervals of level $n$ that are also children of $\widehat{\Delta }$, listed
from left to right, with the same normalized length and quotient neighbour
set as $\Delta $. Define $r_{n}(\Delta )$ to be the integer $r$ such that $%
\Delta =\Delta _{r}$. By the \textit{quotient characteristic vector} of $%
\Delta \in \mathcal{F}_{n}^{(\pi )}$ for $n\geq 1$ we mean the triple%
\begin{equation*}
\mathcal{C}_{n}^{(\pi )}(\Delta )=(\ell _{n}(\Delta ),V_{n}^{(\pi )}(\Delta
),r_{n}(\Delta )).
\end{equation*}%
We call $(\ell _{n}(\Delta ),V_{n}^{(\pi )}(\Delta ))$ the \textit{reduced
quotient characteristic vector} of $\Delta $. For $n=0$, we define 
\begin{equation*}
\mathcal{C}_{0}^{(\pi )}([0,1])=(1,(0,1,2,\dots ,\delta -1),1).
\end{equation*}%
The reason for this choice will be clear later.

We remark that this structure depends only on the similarities $%
\{S_{j}\}_{j=0}^{k}$ and not on the measure itself.

Here is a very important fact about measures of finite type on $\mathbb{T}$.
The same statement (with the appropriate definitions) is known to be true
for measures of finite type on $\mathbb{R}$ \cite{F3}.

\begin{lemma}
If the IFS is of finite type on $\mathbb{T}$, then there are only finitely
many quotient characteristic vectors.
\end{lemma}

\begin{proof}
First, we will check there are only finitely many normalized lengths of
quotient net intervals. It is certainly enough to verify this for net
intervals of level $n$ for large enough $n$, thus we can assume that if $%
\Delta =[a,b]\in \mathcal{F}_{n}^{(\pi )}$, then $\rho ^{n}\delta <1/2$.
Since $\left\vert S_{\sigma }(0)-S_{\sigma }(\delta )\right\vert \leq \rho
^{n}\delta $ when $\sigma \in \mathcal{A}^{n}$, it follows that $\left\vert
b-a\right\vert =d_{\mathbb{T}}(\pi (a),\pi (b))\leq \rho ^{n}\delta $.

Suppose $a=\pi (S_{\sigma }(0))$ and $b=\pi (S_{\tau }(0))$ for some $\sigma
,\tau \in \mathcal{A}^{n}$ (viewing $\mathbb{T}$ as $[0,1)$.) Then 
\begin{equation*}
\rho ^{-n}\left\vert b-a\right\vert =\rho ^{-n}d_{\mathbb{T}}(\pi (S_{\sigma
}(0)),\pi (S_{\tau }(0)))\leq \rho ^{-n}\delta
\end{equation*}%
and thus by definition $\rho ^{-n}\left\vert b-a\right\vert $ belongs to the
finite set $\Lambda $. The argument is similar if $a,b$ are both images of $%
\delta $.

Now suppose $a=\pi (S_{\sigma }(\delta ))$ and $b=\pi (S_{\tau }(0))$. Since 
$\pi ^{-1}(a,b)\cap K$ is not empty, the definition of a net interval
ensures there is some $\alpha \in \mathcal{A}^{n}$ such that $\pi
^{-1}[a,b]\subseteq \lbrack S_{\alpha }(0),S_{\alpha }(\delta )]$. Of
course, $\rho ^{-n}d_{\mathbb{T}}(\pi (S_{\alpha }(0)),\pi (S_{\alpha
}(\delta )))=\delta $. Since 
\begin{eqnarray*}
d_{\mathbb{T}}(\pi (S_{\sigma }(\delta )),\pi (S_{\tau }(0))) &=&d_{\mathbb{T%
}}(\pi (S_{\sigma }(\delta )),\pi (S_{\alpha }(\delta )))+d_{\mathbb{T}}(\pi
(S_{\tau }(0)),\pi (S_{\alpha }(0))) \\
&&-d_{\mathbb{T}}(\pi (S_{\alpha }(\delta )),\pi (S_{\alpha }(0))),
\end{eqnarray*}%
we see that $\rho ^{-n}\left\vert b-a\right\vert =\rho ^{-n}d_{\mathbb{T}%
}(\pi (S_{\sigma }(\delta )),\pi (S_{\tau }(0)))$ belongs to the finite set $%
\Lambda \pm \Lambda \pm \Lambda $ (where we assume, without loss of
generality, that $\delta \in \Lambda $).

Similar, but easier, arguments apply if $a=\pi (S_{\sigma }(0)),$ $b=\pi
(S_{\tau }(\delta ))$. Consequently, there are only finitely many normalized
lengths.

This fact guarantees that each quotient net interval has a bounded number of
children. In particular, there can only be finitely many choices for the 3rd
component of the characteristic vectors.

Lastly, we need to show there are only finitely many neighbours. So suppose $%
a_{i}=\rho ^{-n}(a-S_{\sigma }(0)$ $+l)$ where $\Delta =[a,b]$ and $(\Delta
+l)\cap S_{\sigma }(0,\delta )$ is not empty for $\sigma \in \mathcal{A}^{n}$
and integer $l$. This guarantees that $\left\vert S_{\sigma
}(0)-(a+l)\right\vert \leq \rho ^{n}\delta $ and therefore%
\begin{equation*}
\rho ^{-n}\left\vert S_{\sigma }(0)-(a+l)\right\vert =\rho ^{-n}d_{\mathbb{T}%
}(\pi (S_{\sigma }(0)),\pi (a+l))\in \Lambda \pm \Lambda \pm \Lambda ,
\end{equation*}%
completing the proof.
\end{proof}

\subsection{Examples and Counterexamples}

\label{sec:Pisot}We begin this subsection by exhibiting a family of examples
of quotient measures of finite type on the torus.

\begin{proposition}
\label{ft}Let $\beta $ be a Pisot number and $S_{j}(x)=\beta ^{-1}x+d_{j}$
for $d_{j}\in \mathbb{Q}[\beta ]$ and $j=0,1,2,...,k$. Assume the
self-similar set has convex hull the interval $[0,\delta ]$ with $\delta $
an integer. Then the quotient of any associated self-similar measure is of
finite type on the torus.
\end{proposition}

\begin{lemma}
\label{Galois}Let $S\subset \mathbb{Q}[\beta ]$ be a finite set and $\Lambda
^{S}(\beta )=\left\{ \sum_{i=0}^{n}a_{i}\beta ^{i}|a_{i}\in S,n\in \mathbb{N}%
\right\} $. Then there exists a constant $c>0$ such that if $y,z\in \Lambda
^{S}(\beta ),$ then either $y=z$ or $|y-z|>c$.
\end{lemma}

\begin{proof}
This is essentially done in \cite{Garsia62}, but we include it here for
completeness.

We first observe that we can assume $S\subset \mathbb{Z}[\beta ]$. To see
this we multiply by the least common multiple of the denominators of the $%
a_{i}$. This scales the constant $c$, but does not alter its existence.

Thus we have that $\Lambda ^{S}(\beta )\in \mathbb{Z}[\beta ]$. Let $\sigma
_{1},\sigma _{2},\dots ,\sigma _{n}$ be the Galois automorphisms on $\mathbb{%
Q}[\beta ]$. Either $y=z$ or $y-z\neq 0$. In the latter case, we have $%
y-z\in \mathbb{Z}[\beta ]$, and hence the algebraic norm, $N(y-z)$, is a
non-zero integer. This implies that 
\begin{equation*}
|N(y-z)|=\left\vert \prod_{i}\sigma _{i}(y-z)\right\vert \geq 1
\end{equation*}%
and hence if $y=\sum a_{j}\beta ^{j}$ and $z=\sum a_{j}^{\prime }\beta ^{j}$%
, with $a_{j},a_{j}^{\prime }\in S$, then 
\begin{equation*}
|y-z|\geq \frac{1}{\left\vert \prod_{\sigma _{i}\neq \mathrm{id}}\sigma
_{i}(y-z)\right\vert }\geq \frac{1}{\prod_{\sigma _{i}\neq \mathrm{id}%
}\left\vert \sum_{j}\sigma _{i}(a_{j}-a_{j}^{\prime })\sigma _{i}(\beta
^{j})\right\vert }
\end{equation*}%
where $\mathrm{id}$ is the identity automorphism. Let $c_{\sigma
}=\max_{a,b\in S}|\sigma (a-b)|>0$. We have that 
\begin{equation*}
|y-z|\geq \frac{1}{\prod_{\sigma _{i}\neq \mathrm{id}}\sum_{j=0}^{\infty
}c_{\sigma_i}\left\vert \sigma_{i}(\beta )\right\vert ^{j}}.
\end{equation*}%
As $\beta $ is Pisot, $|\sigma _{i}(\beta )|<1$ for all Galois actions, and
hence the right hand side is bounded below, giving the result.
\end{proof}

\begin{proof}[Proof of Proposition \protect\ref{ft}]
Let 
\begin{equation*}
F=\left\{ \delta ,d_{j}-\ell :j=0,...,k;\ell =0,...,\delta \right\}
\subseteq \mathbb{Q}[\beta ]
\end{equation*}%
and consider the choices of $\sigma ,\tau $ $\in \mathcal{A}^{n}$ such that $%
\rho ^{-n}d_{\mathbb{T}}(\pi (S_{\sigma }(0)),\pi (S_{\tau }(0)))\leq \delta 
$. These normalized distances are equal to $\beta ^{n}\left\vert S_{\sigma
}(0)-S_{\tau }(0)+\ell \right\vert $ for suitable integers $\ell \in
\{-\delta ,...,\delta \}$ and hence $\beta ^{n}d_{\mathbb{T}}(\pi (S_{\sigma
}(0)),\pi (S_{\tau }(0)))\in \Lambda ^{F-F}(\beta )$. According to Lemma \ref%
{Galois}, the absolute values of the non-zero elements of $\Lambda
^{F-F}(\beta )$ are bounded away from zero, say $\geq \varepsilon $. Thus
there can be at most $(2\delta +1)/\varepsilon $ elements in $\Lambda
^{F-F}(\beta )\cap \lbrack -\delta ,\delta ]$. This proves the finite type
property.
\end{proof}

\begin{remark}
\label{Rconv}If $\mu _{\pi },\nu _{\pi }$ are two measures on $\mathbb{T}$,
then their convolution (on $\mathbb{T)}$ is defined by 
\begin{equation*}
\mu _{\pi }\ast \nu _{\pi }(E)=\int_{\mathbb{T}}\mu _{\pi }(E-x)\text{ }d\nu
_{\pi }(x)
\end{equation*}%
for all Borel sets $E\subseteq \mathbb{T}$. (Here the group operation is
understood on the torus.) If $\mu _{\pi }$ and $\nu _{\pi }$ are the
quotients of measures $\mu $ and $\nu $ on $\mathbb{R}$ respectively, then
the convolution $\mu _{\pi }\ast \nu _{\pi }$ is equal to the quotient of
the convolution (on $\mathbb{R}$) $\mu \ast \nu $. In other words, $(\mu
\ast \nu )_{\pi }=\mu _{\pi }\ast \nu _{\pi }$.
\end{remark}

From Proposition \ref{ft} and Remark \ref{Rconv} we immediately deduce the
following:

\begin{corollary}
Any $k$-fold convolution (taken on $\mathbb{T}$) of a Cantor measure or
Bernoulli convolution, where the contraction factor is the inverse of a
Pisot number, is of finite type on $\mathbb{T}$.
\end{corollary}

\label{sec:finite type}

However, there are also measures of finite type on $\mathbb{R}$ whose
quotients are not of finite type on $\mathbb{T}$. Here is an example.

\begin{example}
\label{NotFT} Let $\varrho $ be a solution to 
\begin{equation*}
2(1-\varrho )\sum_{i=1}^{\infty }\varrho ^{i^{2}}=1/2
\end{equation*}%
(approximately $0.384$) and consider the IFS with similarities $%
S_{0}(x)=\varrho x$ and $S_{1}(x)=\varrho x+2(1-\varrho )$. The convex hull
of the self-similar set is $[0,2]$. Let $\mu $ be the associated
self-similar measure. As $\varrho <1/2$, the IFS satisfies the OSC and even
the strong separation condition, and thus $\mu $ is of finite type on $%
\mathbb{R}$. Let $\sigma ^{(n)}=(\sigma _{0},\sigma _{1},...,\sigma
_{n^{2}}) $ and $\tau ^{(n)}=(\tau _{0},\tau _{1},...,\tau _{n^{2}})$ $\in 
\mathcal{A}^{n^{2}+1}$ be defined by $\sigma _{i^{2}}=1$ if $i\neq 0$, $%
\sigma _{i}=0$ otherwise, and $\tau _{i}=1-\sigma _{i}$. The points $%
a_{n}=S_{\sigma^{(n)} }(0)$ and $b_{n}=S_{\tau^{(n)} }(0)$ are easily seen
to be symmetric about $1$ with $a_{n}<1/2$ and $b_{n}>3/2$. Consequently, 
\begin{equation*}
d_{\mathbb{T}}(\pi (a_{n}),\pi (b_{n}))=2\left\vert a-1/2\right\vert
=4(1-\varrho )\sum_{i=n+1}^{\infty }\varrho ^{i^{2}}\approx 4(1-\varrho
)\varrho ^{(n+1)^{2}}.
\end{equation*}%
Hence there exists a net interval in $\mathcal{F}_{n^{2}+1}^{(\pi )}$ with
normalized length at most 
\begin{equation*}
\approx 4\varrho ^{-(n^{2}+1)}\varrho ^{(n+1)^{2}}(1-\varrho )=4(1-\varrho
)\varrho ^{2n}.
\end{equation*}%
These normalized lengths are not bounded below and hence there cannot be a
finite number of characteristic vectors. Therefore $\mu _{\pi }$ is not of
finite type.
\end{example}

\subsection{Symbolic representations and the essential class(es)}

\label{sec:multi}

By an \textit{admissible path} we mean a finite tuple $\eta =(\gamma _{j})$
where each $\gamma _{j}$ is the quotient characteristic vector of $\Delta
_{j}$ and $\Delta _{j}$ is the parent of $\Delta _{j+1}$. Each $\Delta \in 
\mathcal{F}_{n}^{(\pi )}$ can be identified with a unique admissible path $%
\eta =\eta (\Delta )=(\gamma _{j})_{j=0}^{n}$ where $\gamma _{0}=\mathcal{C}%
_{0}[0,1]$ and $\gamma _{n}=\mathcal{C}_{n}(\Delta )$. We call this the 
\textit{symbolic representation} of $\Delta $.

We will often write $\Delta _{n}(x)$ for a net interval of level $n$
containing $x$. By the \textit{quotient} \textit{symbolic representation} of 
$x\in K$ we mean the sequence 
\begin{equation*}
\lbrack x]=(\mathcal{C}_{0}(\Delta _{0}(x)),\dots ,\mathcal{C}_{n}(\Delta
_{n}(x)),...)
\end{equation*}%
where $\Delta _{n}(x)\subseteq \Delta _{n-1}(x)$. If $x$ is an endpoint of $%
\Delta _{n}(x)$ for some $n,$ then $x$ is called a \textit{boundary point}
and it can have two symbolic representations. Otherwise, the symbolic
representation of $x$ is unique.

We can also define the notion of quotient loop classes and essential classes
in the same manner as was done for measures of finite type on $\mathbb{R}$.
A non-empty subset $\Omega ^{\prime }$ of quotient characteristic vectors
will be called a\textbf{\ }\textit{quotient loop class} if whenever $\alpha
,\beta \in \Omega ^{\prime }$, then there are quotient characteristic
vectors $\gamma _{j}$, $j=1,\dots ,n$, such that $\gamma _{1}=\alpha $, $%
\gamma _{n}=\beta $ and $(\gamma _{1},\dots ,\gamma _{n})$ is an admissible
path with all $\gamma _{j}\in \Omega ^{\prime }$. A loop class $\Omega
^{\prime }$ is called a \textit{quotient essential class} if, in addition,
whenever $\alpha \in \Omega ^{\prime }$ and $\beta \in \Omega $ is a child
of $\alpha $, then $\beta \in \Omega ^{\prime }$. We say that an element $x$
with symbolic representation $(\gamma _{0},\gamma _{1},.,,,)$ is \textit{in
a loop class} (such as \textit{in the essential class}) if there is some $J$
such that $\gamma _{j}\in $ loop class for all $j\geq J$.

The finite type property ensures that every element in $\pi (K)$ is
contained in a quotient loop class if the associated IFS is of finite type
on $\mathbb{T}$.

Feng, in \cite[Lemma 6.4]{F2}, proved that if $\mu $ is of finite type on $%
\mathbb{R}$, then there is exactly one essential class. Surprisingly, this
is not true for self-similar measures of finite type on the torus, as the
example below demonstrates.

\begin{example}
\label{EssnotUnique}Consider the IFS with maps $S_{j}(x)=x/4+d_{j}/5$ for $%
j=0,...,4$, $d_{j}=3j$ for $j=0,...,3$ and $d_{4}=15$ and any probabilities $%
p_{j}$. According to Proposition \ref{ft} this IFS is of finite type on $%
\mathbb{T}$. The convex hull of $K$ is $[0,4]$. Using the computer, we
determined that any corresponding quotient measure has 10 reduced quotient
characteristic vectors. From the reduced transition diagram, Figure \ref%
{fig:2EC}, one can see that there are two different essential classes, the
first from the reduced characteristic vector labelled $7$ and the second
consisting of the three reduced characteristic vectors $5,9,10$.

\begin{figure}[tbp]
\includegraphics[scale=0.5]{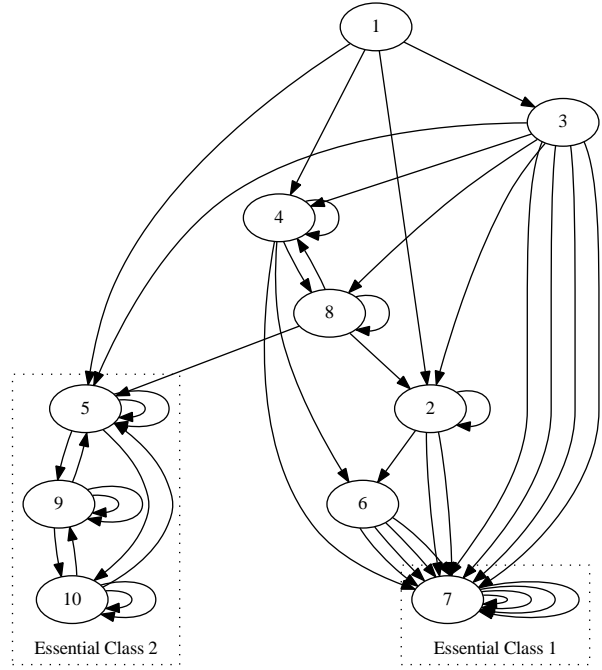}
\caption{Transition diagram for Example \protect\ref{EssnotUnique}}
\label{fig:2EC}
\end{figure}
\end{example}

\begin{remark}
In \cite[Proposition 3.6]{HHN} it was shown that, for self-similar measures
of finite type on $\mathbb{R}$, the set of points in the essential class has
full $\mu $-measure, and full $H^{s}$-measure. Here $H^{s}$ is the Hausdorff
measure associated to the support of $\mu $. Despite not necessarily having
a unique essential class, a similar proof will show that for $\mu _{\pi }$ a
self-similar measure of finite type on the torus, the set of points in the
union of all essential classes of $\mu _{\pi }$ will have full $\mu _{\pi }$%
-measure and full $H^{s}$-measure where $H_{s}$ is the Hausdorf measure
associated with the support of $\mu _{\pi }$ on the torus.
\end{remark}

\section{Computing local dimensions}

\label{sec:comp}

The\textit{\ upper }and\textit{\ lower local dimensions} at a point $x$ in
the support of a probability measure $\mu $ defined on a metric space are
defined as%
\begin{eqnarray*}
\overline{\dim }_{loc}\mu (x) &=&\limsup_{r\rightarrow 0}\frac{\log \mu
(B(x,r))}{\log r} \\
\underline{\dim }_{loc}\mu (x) &=&\liminf_{r\rightarrow 0}\frac{\log \mu
(B(x,r))}{\log r}.
\end{eqnarray*}%
When the two coincide the number is known as the \textit{local dimension} of 
$\mu $ at $x$.

As was the case for measures of finite type on $\mathbb{R}$, there is an
iterative strategy for computing local dimensions for measures of finite
type on $\mathbb{T}$, based upon suitable transition matrices.

\begin{notation}
Suppose $\Delta =[a,b]\in \mathcal{F}_{n}^{(\pi )}$ has parent $\widehat{%
\Delta }=[c,d]\in \mathcal{F}_{n-1}^{(\pi )}$ and assume their quotient
neighbour sets are $V_{n}^{(\pi )}(\Delta )=(a_{1},\dots ,a_{J})$ and $%
V_{n-1}^{(\pi )}(\widehat{\Delta })=(c_{1},\dots ,c_{I})$, respectively. The 
\textbf{quotient} \textbf{primitive transition matrix,} 
\begin{equation*}
T(\mathcal{C}_{n-1}(\widehat{\Delta }),\mathcal{C}_{n}(\Delta ))=(T_{ji})
\end{equation*}%
is the $J\times I$ matrix defined by the rule that $T_{ji}=p_{s}$ if there
is some $\sigma \in \mathcal{A}^{n-1}$ and integer $l$ with $S_{\sigma
}(0)=c-\varrho ^{n-1}c_{j}+l$ and $S_{\sigma s}(0)=a-\varrho ^{n}a_{i}+l$.
We set $T_{ji}=0$ otherwise.
\end{notation}

Given the net interval $\Delta $ with symbolic representation $\eta =(\gamma
_{j})_{j=0}^{n}$ , we write 
\begin{equation*}
T(\eta )=T(\gamma _{0},\gamma _{1})\cdot \cdot \cdot T(\gamma _{n-1},\gamma
_{n}).
\end{equation*}%
Any such product of primitive transition matrices will be called a \textit{%
quotient transition matrix. }By an essential primitive transition matrix, we
mean a transition matrix $T(\gamma _{j-1},\gamma _{j})$ where $\gamma
_{j-1},\gamma _{j}$ belongs to the essential class.

The definition ensures that each column of a primitive transition matrix
contains a non-zero entry. If $\pi (K)=\mathbb{T}$, then the same statement
holds for each row.

We note that as $S_{\sigma s}(0)=\varrho ^{n-1}d_{s}+S_{\sigma }(0)$, if $%
S_{\sigma }(0)=c-\varrho ^{n-1}c_{j}+l$ and $S_{\sigma s}(0)=a-\varrho
^{n}a_{i}+l$ for integer $l$, then 
\begin{equation*}
a-\varrho ^{n}a_{i}+l=\varrho ^{n-1}d_{s}+c-\varrho ^{n-1}c_{j}+l,
\end{equation*}%
so $s$ is uniquely determined (even if $\sigma $ and $l$ are not).

\begin{proposition}
Let $\mu $ be a self-similar measure satisfying (\ref{ssmeas}). Assume the
convex hull of the self-similar set is $[0,\delta ]$ and that the quotient
measure $\mu _{\pi }$ is of finite type on $\mathbb{T}$. Let $\Delta =[a,b]$
be a quotient net interval in $\mathcal{F}_{n}^{(\pi )}$ with reduced
quotient characteristic vector $(\ell _{n}(\Delta ),$ $(a_{1},...,a_{m}))$.
Then 
\begin{equation*}
{\mu _{\pi }}(\Delta )=\sum_{i=1}^{m}{\mu }\left( [a_{i},a_{i}+\ell
_{n}(\Delta )]\right) \sum_{l=0}^{\delta -1}\sum_{\substack{ \sigma \in 
\mathcal{A}_{n}  \\ a_{i}=(S_{\sigma }(0)-a-l)\varrho ^{-n}}}p_{\sigma }.
\end{equation*}
\end{proposition}

\begin{proof}
Assume $\pi ^{-1}(\Delta )\cap \lbrack 0,\delta ]=\bigcup_{l=0}^{\delta
-1}\Delta ^{(l)}$ where $\Delta ^{(l)}=\Delta +l=[a+l,b+l]$. Then 
\begin{equation*}
{\mu _{\pi }}(\Delta )=\mu (\pi ^{-1}(\Delta ))=\sum_{l=0}^{\delta -1}\mu
(\Delta ^{(l)})=\sum_{l}\sum_{\sigma \in \mathcal{A}_{n}}p_{\sigma }\mu
(S_{\sigma }^{-1}(\Delta ^{(l)})).
\end{equation*}%
As $\mu $ is non-atomic this equals 
\begin{equation*}
{\mu _{\pi }}(\Delta )=\sum_{l=0}^{\delta -1}\sum_{\substack{ \sigma \in 
\mathcal{A}_{n}  \\ S_{\sigma }(K)\cap int\Delta ^{(l)}\neq \emptyset }}%
p_{\sigma }\mu (S_{\sigma }^{-1}(\Delta ^{(l)})).
\end{equation*}%
Note that if $S_{\sigma }(K)\cap int\Delta ^{(l)}\neq \emptyset ,$ then by
definition $(S_{\sigma }(0)-a-l)\varrho ^{-n}$ $=a_{i}$ belongs to the
quotient neighbour set of $\Delta $. Hence 
\begin{equation*}
{\mu _{\pi }}(\Delta )=\sum_{l=0}^{\delta -1}\sum_{i=1}^{m}\sum_{\substack{ %
\sigma \in \mathcal{A}_{n}  \\ a_{i}=(S_{\sigma }(0)-a-l)\varrho ^{-n}}}%
p_{\sigma }\mu (S_{\sigma }^{-1}(\Delta ^{(l)})).
\end{equation*}%
As $S_{\sigma }(x)=a-a_{i}\varrho ^{n}+l+\varrho ^{n}x,$ we see that $%
S_{\sigma }([a_{i},a_{i}+\ell (\Delta )])=[a+l,b+l]$. Hence 
\begin{align*}
{\mu _{\pi }}(\Delta )& =\sum_{l=0}^{\delta -1}\sum_{i=1}^{m}\sum_{\substack{
\sigma \in \mathcal{A}_{n}  \\ a_{i}=(S_{\sigma }(0)-a-l)\varrho ^{-n}}}%
p_{\sigma }\mu ([a_{i},a_{i}+\ell (\Delta )]) \\
& =\sum_{i=1}^{m}{\mu }([a_{i},a_{i}+\ell _{n}(\Delta )])\sum_{l=0}^{\delta
-1}\sum_{\substack{ \sigma \in \mathcal{A}_{n}  \\ a_{i}=(S_{\sigma
}(0)-a-l)\varrho ^{-n}}}p_{\sigma }
\end{align*}%
as claimed.
\end{proof}

\begin{notation}
For $n\geq 1$, let 
\begin{equation*}
P_{n}^{(i)}=P_{n}^{(i)}(\Delta )=\sum_{l=0}^{\delta -1}\sum_{\substack{ %
\sigma \in \mathcal{A}_{n}  \\ a_{i}=(S_{\sigma }(0)-a-l)\varrho ^{-n}}}%
p_{\sigma },
\end{equation*}%
\begin{equation*}
Q_{n}(\Delta )=(P_{n}^{(1)},...,P_{n}^{(m)}),
\end{equation*}%
and 
\begin{equation*}
Q_{0}([0,1])=(1,...,1)\in \mathbb{R}^{\delta }.
\end{equation*}
\end{notation}

We remark that as $[a_{i},a_{i}+\ell _{n}(\Delta )]=S_{\sigma
}^{-1}([a+l,b+l])$ and $K\bigcap S_{\sigma }^{-1}(a+l,b+l)$ is not empty, we
have ${\mu }([a_{i},a_{i}+\ell _{n}(\Delta )])>0$. Consequently, 
\begin{equation*}
{\mu _{\pi }}(\Delta )\sim \sum_{i}P_{n}^{(i)}(\Delta )=\left\Vert
Q_{n}(\Delta )\right\Vert ,
\end{equation*}%
with the constants of comparability independent of $\Delta $ or $n$. (Here
the norm of the vector is the sum of the absolute values of the entries.)

\begin{proposition}
For all $n\geq 1$, we have%
\begin{equation*}
Q_{n}(\Delta )=Q_{n-1}(\widehat{\Delta })T(\mathcal{C}_{n-1}(\widehat{\Delta 
}),\mathcal{C}_{n}(\Delta )).
\end{equation*}
\end{proposition}

\begin{proof}
First, assume $n\geq 2$. We want to show that for each $i$,%
\begin{equation}
P_{n}^{(i)}=\sum_{j}P_{n-1}^{(j)}t_{ji}  \label{TransitionMatrix}
\end{equation}%
where $T(\mathcal{C}_{n-1}(\widehat{\Delta }),\mathcal{C}_{n}(\Delta
))=(t_{ji})$

Consider a typical summand in the formula for $P_{n}^{(i)}$, say $p_{\sigma
} $ where $\sigma \in \mathcal{A}^{n}$ and $\varrho ^{-n}(a-S_{\sigma
}(0)+l)=a_{i}$ for suitable integer $l$. Put $\sigma =\widehat{\sigma }s$
where $\widehat{\sigma }\in \mathcal{A}^{n-1}$ and $s\in \mathcal{A}$. As $%
S_{\widehat{\sigma }s}(0)=a+l-$ $\varrho ^{n}a_{i}$, by construction there
is some index $j$ such that $S_{\widehat{\sigma }}(0)=c+l-$ $\varrho
^{n-1}c_{j}$. But then $p_{\sigma }=p_{\widehat{\sigma }}p_{s}=p_{\widehat{%
\sigma }}t_{ji}$ and $p_{\widehat{\sigma }}$ is a summand of $P_{n-1}^{(j)}$.

On the other hand, assume $t_{ji}\neq 0$ and $p_{\widehat{\sigma }}$ is a
summand of $P_{n-1}^{(j)}$. Then there must be some $l$ such that $S_{%
\widehat{\sigma }}(0)=c+l-$ $\varrho ^{n-1}c_{j}$. Furthermore, there is
some integer $k$ and $\tau \in \mathcal{A}^{n}$ such that $S_{\tau }(0)=a+k-$
$\varrho ^{n}a_{i}$ and $S_{\widehat{\tau }}(0)=c+k-$ $\varrho ^{n-1}c_{j}$.
Assume $\tau =\widehat{\tau }s$. Then $S_{\widehat{\tau }s}(0)=$ $S_{%
\widehat{\tau }}(0)+\varrho ^{n-1}d_{s}=S_{\widehat{\sigma }}(0)+k-l+\varrho
^{n-1}d_{s}=S_{\widehat{\sigma }s}(0)+k-l,$ hence $S_{\widehat{\sigma }%
s}(0)=a+l-\varrho ^{n}a_{i}$. This observation shows that $p_{\sigma }=p_{%
\widehat{\sigma }}p_{s}$ is a summand of $P_{n}^{(i)}$ and $t_{ji}=p_{s}$.
Together, these observations prove (\ref{TransitionMatrix}), as required.

Now suppose $n=1$. In this case, we need to show that $P_{1}^{(i)}=%
\sum_{j}t_{ji}$ where $(t_{ji})=T(\mathcal{C}_{0}[0,1]),\mathcal{C}%
_{1}[a,b]) $. By definition $t_{ji}=p_{s}$ if $S_{s}(0)=a+l-\varrho a_{i}$
and $0=l-c_{j}$. Thus the definition of the neighbour set of $[0,1]$ as $%
\{0,1,...,\delta -1\}$ ensures we have $P_{1}^{(i)}=\sum_{j}t_{ji}$ for each 
$i$.
\end{proof}

By the matrix norm of matrix $M=(M_{jk})$ we mean $\left\Vert M\right\Vert
=\sum_{jk}\left\vert M_{jk}\right\vert $. In terms of this notation the
previous results combine to yield

\begin{corollary}
\label{measurenetint}There are positive constants $c_{1},c_{2}$ such that 
\begin{equation*}
c_{1}\left\Vert T(\eta )\right\Vert \leq {\mu _{\pi }}(\Delta )\leq
c_{2}\left\Vert T(\eta )\right\Vert
\end{equation*}%
whenever $\Delta \in \mathcal{F}_{n}^{(\pi )}$ has symbolic representation $%
\eta $.
\end{corollary}

Given $\Delta _{n}\in \mathcal{F}_{n}^{(\pi )}$, we let $\Delta _{n}{}^{+}$
and $\Delta _{n}^{-}$ be the quotient net intervals of level $n$ sharing
endpoints with $\Delta _{n},$ (in the torus sense). If $x$ belongs to the
interior of $\Delta _{n}(x)$ we put 
\begin{equation*}
M_{n}(x)=\mu _{\pi }(\Delta _{n}(x))+\mu _{\pi }(\Delta _{n}^{+}(x))+\mu
_{\pi }(\Delta _{n}^{-}(x)),
\end{equation*}%
while if $x$ is a boundary point of $\Delta _{n}(x)$ we put $M_{n}(x)=\mu
_{\pi }(\Delta _{n}(x))$ $+\mu _{\pi }(\Delta _{n}^{\prime }(x))$ where $%
\Delta _{n}^{\prime }(x)$ is the other net interval of level $n$ with $x$ as
endpoint. Since all quotient net intervals of level $n$ have lengths
comparable to $\rho ^{n}$ one can see in the same manner as in \cite[Thm. 2.6%
]{HHN} that 
\begin{equation*}
\dim _{loc}\mu _{\pi }(x)=\lim_{n\rightarrow \infty }\frac{\log M_{n}(x)}{%
n\log \rho },
\end{equation*}%
with a similar formula for upper and lower local dimensions. In the special
case that the probabilities defining the self-similar measure are regular
and $\pi (K)=\mathbb{T}$, the same arguments as given in \cite[Lemma 3.5 -
Cor. 3.7]{HHM} show that if $[x]=(\gamma _{0},\gamma _{1},...),$ then we
have the simpler formula,%
\begin{equation}
\dim _{loc}\mu _{\pi }(x)=\lim_{n\rightarrow \infty }\frac{\log \mu _{\pi
}(\Delta _{n}(x))}{n\log \rho }=\lim_{n\rightarrow \infty }\frac{\log
\left\Vert T(\gamma _{0},..,\gamma _{n})\right\Vert }{n\log \rho }.
\label{reg}
\end{equation}

In the study of finite type measures on $\mathbb{R}$, periodic points played
an important role. In a similar fashion, we will say that a point $x\in \pi
(K)$ is \textit{periodic} if it has quotient symbolic representation $%
[x]=(\gamma _{0},..,\gamma _{J},\theta ^{-},\theta ^{-},...)$ where $\theta $
is an admissible cycle (a path whose first and last letters coincide) and $%
\theta ^{-}$ is the path with the last letter of $\theta $ deleted. Any
point which is a boundary point of some quotient net interval is a periodic
point.

As was the case for finite type measures in $\mathbb{R}$, we have the
following formula for local dimensions of periodic points. We write $%
L(\theta ^{-})$ for the length of the path $\theta ^{-}$.

\begin{proposition}
\label{periodic}If $\mu _{\pi }$ is a measure of finite type on $\mathbb{T}$
and $x$ is a periodic point with period $\theta $, then the local dimension
exists and is given by 
\begin{equation*}
\dim _{loc}\mu _{\pi }(x)=\frac{\log sp(T(\theta ))}{L(\theta ^{-})\log
\varrho },
\end{equation*}%
where if $x$ is a boundary point of a quotient net interval with two
different symbolic representations given by periods $\theta $ and $\phi $ of
the same length, then $\theta $ is chosen with $sp(T(\theta ))\geq sp(T(\phi
))$.
\end{proposition}

The proof is the same as in \cite[Proposition 2.7]{HHN} as that argument
only required the formula we have developed for $\mu _{\pi }(\Delta )$ 
and the fact that the primitive transition matrices have a
non-zero entry in each column.

We will call a quotient loop class \textit{positive} if there is some path $%
\eta $ in the loop class for which $T(\eta )$ has strictly positive entries.
With the preliminary results we have established, the same
arguments as in \cite[Section 5]{HHM} prove the following important result.
The details are left to the reader.

\begin{theorem}
\label{regprob}Suppose $\mu _{\pi }$ is a quotient measure of finite type on 
$\mathbb{T}$ associated with regular probabilities and with supp$\mu _{\pi }=%
\mathbb{T}$. Then the set of local dimensions at points in any positive
quotient loop class is a closed interval and the local dimensions at the
periodic points in the loop class are dense in that interval.
\end{theorem}

\begin{remark}
It was shown in \cite{HHM} that for measures of finite type on $\mathbb{R}$
the essential class is always a positive loop class. In the next section we
will see that the quotient essential class (even if unique) need not be
positive.
\end{remark}

This theorem can be used to see that, as with measures of finite type on $%
\mathbb{R}$, the sets of local dimensions of measures of finite type on $%
\mathbb{T}$ may or may not admit isolated points. Here are some examples.
The details of these examples can be found in \cite{HArXiv}.

\begin{example}
\label{ex:golden} Consider $\nu ,$ the convolution square (on $\mathbb{R})$
of the uniform Bernoulli convolution with contraction factor $\rho $ the
inverse of the golden mean. This is the self-similar measure associated with
the IFS with similarities $\rho x,$ $\rho x+1-\rho $ and $\rho x+2(1-\rho )$
and probabilities $1/4,1/2$ and 1/4 respectively. In \cite[Sect. 8.2]{HHM}
it was shown that the set of local dimensions of $\nu $ admits an isolated
point. The quotient measure, $\nu _{\pi },$ has one essential class which is
positive and hence generates an interval of local dimensions. The quotient
essential primitive transition matrices and admissible paths precisely
coincide with those of $\nu $. It follows that the set of local dimensions
of $\nu _{\pi }$ at points in the quotient essential class coincides with
the interval of local dimensions of $\nu $ at points in its essential class.

There are also two additional maximal quotient loop classes, both of which
are simple loops. These generate the same three (periodic boundary) points,
namely, $0,\rho ,1-\rho $. The three points have the same local dimension
and it can be shown that this value is contained in the interval of local
dimensions generated by the quotient essential class. It follows that the
set of local dimensions of the quotient measure has no isolated point.
\end{example}

\begin{example}
\label{Ex:isolatedPt}Consider the IFS $\{S_{j}(x)=x/4+d_{j}/8:j=0,...,7\}$
where $d_{j}=j$ for $j=0,...,4$, $d_{5}=7,d_{6}=9$ and $d_{7}=12$. This IFS
generates the self-similar set $[0,2]$ and is of finite type on $\mathbb{T}$%
. There are 10 reduced quotient characteristic vectors and one quotient
essential class which consists of one reduced quotient characteristic vector
(specifically, 4 non-reduced quotient characteristic vectors). Let $\mu $ be
the self-similar measure arising from the regular probabilities $%
p_{0}=p_{7}=1/2402$, $p_{1}=p_{2}=1000/2402$, $p_{3}=p_{4}=$ $%
p_{5}=p_{6}=100/2402.$ We have verified that the quotient essential class is
of positive type when these probabilities are used. Computational work also
shows that 
\begin{equation*}
\lbrack .6283,1.885]\subseteq \{\dim \mu _{\pi }(x):x\text{ in essential
class}\}\subseteq \lbrack .614,2.053].
\end{equation*}

There are three other maximal loop classes, all of which are simple loops.
Two of these loops correspond to the single point $7/8$ whose local
dimension is $\frac{\log (100/2402)}{\log 4}\sim 2.293$. One of these two
loops corresponds to the path approaching $7/8$ from the left, and the other
to the path approaching $7/8$ from the right. There are countably many
points associated with the other maximal loop class, all having local
dimension $\sim 2.286$ (the spectral radius of the normalized transition
matrix of the length one cycle is the root of the polynomial $%
x^{2}-100x-100, $ approximately 100.99). Consequently, the set of local
dimensions of $\mu _{\pi }$ consists of a closed interval and two isolated
points. It is again the case that the interval in the set of local
dimensions of $\mu _{\pi }$ coincides with the interval component of the set
of local dimensions of $\mu $. The two isolated points are also both
isolated points in the set of local dimensions of $\mu $.
\end{example}

A weaker result can be proven if we do not assume that the probabilities are
regular or that the support of the quotient measure is the torus. Note that
by an essential quotient transition matrix, we mean a transition matrix $%
T(\eta )$ where the path $\eta $ belongs to the essential class. Recall that
if supp$\mu _{\pi }=\mathbb{T}$, then there is a non-zero entry on each row
of each primitive transition matrix.

\begin{theorem}
\label{notregprob}Suppose $\mu _{\pi }$ is a quotient measure of finite type
on $\mathbb{T}$. Assume the essential class $E$ is positive and that each
essential primitive transition matrix has a non-zero entry in each row. Then
the set of local dimensions of $\mu _{\pi }$ at points in the relative
interior of $E$ is a closed interval and the local dimensions at periodic
points from $E$ are dense in this interval.
\end{theorem}

This follows by similar arguments to those given in \cite[Sect 3.3]{HHN}.
(The proofs given in Section \ref{sec:4.2} are also similar.)

Next, we will apply this theorem to show that finite-type quotients of
self-similar measures on $\mathbb{R}$ satisfying the strong separation
condition have the same local dimension as the original measure. Recall that
an IFS $\{S_{j}:j=1,..,m\}$ with self-similar set $K$ is said to satisfy the
strong separation condition if the sets $S_{j}(K)$ are pairwise disjoint.

\begin{theorem}
\label{strictsep}Assume the equicontractive IFS $\{S_{j}=\varrho
x+d_{j}:j=0,...,k\}$ with $k\geq 1$ satisfies the strong separation
condition and is of finite type on $\mathbb{T}$. Let $\{p_{j}\}_{j=0}^{k}$
be probabilities and assume $\mu $ is the associated self-similar measure
and $\mu _{\pi }$ is its quotient measure. Then%
\begin{equation*}
\{\dim _{loc}\mu (x):x\in K\}=\left[ \frac{\log (\max p_{i})}{\log {\varrho }%
},\frac{\log (\min p_{i})}{\log {\varrho }}\right] :=I
\end{equation*}%
\begin{equation*}
=\{\dim _{loc}\mu _{\pi }(x):x\in \pi (K)\}.
\end{equation*}
\end{theorem}

\begin{proof}
The first equality is well known (c.f. \cite[ch. 11]{Fa}), so we only need
to prove the second.

As usual, denote by $[0,\delta ]$ the convex hull of the self-similar set $K$%
. Since the IFS satisfies the strong separation condition, the Lebesgue
measure of $K$ is zero. As $\pi (K)$ can be identified with $\bigcup_{n\in 
\mathbb{Z}}(K_{n}-n)$ where $K_{n}=K\cap \lbrack n,n+1)$, it follows that $%
\pi (K)$ also has Lebesgue measure (on $\mathbb{T}$) equal to zero.
Consequently, for $n$ sufficiently large, the Lebesgue measure of $\pi
\left( \bigcup_{|\sigma |=n}S_{\sigma }([0,\delta ])\right) <1$ and hence
there must be intervals $(a_{i},a_{i+1})\subseteq (0,1)$ such that both $%
a_{i}$ and $a_{i+1}$ are in $\pi \left( \bigcup_{|\sigma |=n}S_{\sigma
}([0,\delta ])\right) $, while $(a_{i},a_{i+1})\cap $ $\pi \left(
\bigcup_{|\sigma |=n}S_{\sigma }([0,\delta ])\right) $ is empty.

As $\pi (K)$ is a perfect set, $a_{i+1}$ is a limit point and thus there is
a quotient net interval of level $n,$ say $\Delta _{n}=[a_{i+1},a_{i+2}],$
adjacent (in the torus sense) to the empty interval. Moreover, since $%
(a_{i},a_{i+1})\cap \pi (S_{\sigma }([0,\delta ]))=\emptyset $ for all $%
|\sigma |=n,$ we see that if $[a_{i+1},a_{i+2}]\subset \pi (S_{\sigma
}([0,\delta ])),$ then $\pi (S_{\sigma }(0))=a_{i+1}$. This implies that the
neighbour of $\Delta _{n}$ is simply the singleton $(0)$.

Let $\Delta _{n+1}$ be the left-most descendent of $\Delta _{n}$. The same
reasoning shows the only neighbour of $\Delta _{n+1}$ is the singleton $(0)$%
. The same holds more generally for the left-most descendent of each
generation.

Choose $\ell $ such that $\delta \varrho ^{n+\ell }<|\Delta _{n}|$ and let $%
\Delta _{n+\ell }$ be the left-most descendent of $\Delta _{n}$ at level $%
n+\ell $. As $\Delta _{n+\ell }$ is the intersection of a set of size $%
\delta \varrho ^{n+\ell }$ with $\Delta _{n}$, we see that $\Delta _{n+\ell
} $ will have normalized length $\delta $. This proves that $\mathcal{\gamma 
}=(\delta ,(0))$ is a reduced quotient characteristic vector.

Assume $\pi (S_{\sigma }([0,\delta ]))=\Delta _{n+\ell }$. As the IFS
satisfies the strong separation condition, we see that $\pi (S_{\sigma
i}([0,\delta ]))$ are all disjoint scaled copies of $\Delta _{n+\ell },$
with precisely the same geometry. This shows that the $k+1$ children of $%
\gamma $ have again the same reduced characteristic vector $\gamma $.
Moreover the probabilities associated to these children, in order from left
to right, are $p_{0},p_{1},\dots ,p_{k}$, thus the corresponding transition
matrices are these $1\times 1$ positive matrices.

To this point, we have proven that $\{\gamma \}$ is \emph{one} (reduced)
essential class of $\mu _{\pi }$. Next, we will show that it is \emph{the}
essential class. We prove this by showing $\gamma $ is a descendent of any
quotient characteristic vector.

Let $\Delta $ be a net interval in $\mathcal{F}_{n}^{(\pi )}$. By an
argument similar to above we see that there exists a $j$ such that $\Delta $
will contain a quotient net interval $\Delta _{0}=[a_{i+1},a_{i+2}]\subseteq 
\mathcal{F}_{n+j}^{(\pi )}$, that is adjacent to a set $(a_{i},a_{i+1})$
which is disjoint from all $\pi \left( S_{\sigma }([0,\delta ])\right) $ for 
$|\sigma |=n+j$. We then proceed as before, taking the left-most descendents
of $\Delta _{0}$ to find a quotient net subinterval that has characteristic
vector $\gamma $. Hence $\{\gamma \}$ is \emph{the} (reduced) essential
class and the essential class consists of the $k+1$ vectors $\gamma
_{j}=(\delta ,(0),j)$, $j=0,...,k$.

From the previous theorem it follows that the set of local dimensions at
points in the relative interior of the essential class is a closed interval.
Assume $p_{\alpha }=\min p_{i}$ and $p_{\beta }=\max p_{i}$. It is easy to
see that $[x]=(\chi _{0},\chi _{1},...,\chi _{J},\gamma _{0},\gamma
_{k},\gamma _{\alpha },\gamma _{\alpha },...)$ and $[y]=(\chi _{0},\chi
_{1}^{\prime },...,\chi _{J}^{\prime },\gamma _{0},\gamma _{k},\gamma
_{\beta },\gamma _{\beta },...)$ belong to the interior of the essential
class for suitable admissible paths $\chi _{0},\chi _{1},...,\chi _{J}$ and $%
\chi _{0},\chi _{1}^{\prime },...,\chi _{j}^{\prime }$ and that%
\begin{equation*}
\dim _{loc}\mu _{\pi }(x)=\frac{\log (\min p_{i})}{\log {\varrho }},\dim
_{loc}\mu _{\pi }(y)=\frac{\log (\max p_{i})}{\log {\varrho }}.
\end{equation*}%
Consequently, the set of local dimensions of $\mu _{\pi }$ contains the
interval $I$.

The definition of the quotient measure implies that 
\begin{equation*}
\min_{\ell }\left( \underline{\dim }_{loc}\mu (x+\ell )\right) \leq \dim
_{loc}\mu _{\pi }(x)\leq \min_{\ell }\left( \overline{\dim }_{loc}\mu
(x+\ell )\right)
\end{equation*}%
where the minimum is taken over all integers $\ell $ such that $x+\ell \in $%
supp$\mu $. Since the upper and lower local dimensions of $\mu $ at any
point in its support lie in the interval $I$, it follows that the same is
true for $\dim _{loc}\mu _{\pi }(x)$ at any $x\in \pi (K)$. This completes
the proof.
\end{proof}

\begin{remark}
We remark that Example \ref{NotFT} demonstrates that satisfying the strong
separation condition and having the convex hull of the support being $[0,2]$
is not enough to guarantee that the quotient measure is of finite type.
\end{remark}

\section{Quotients of Cantor-like measures of finite type}

\label{sec:Cantor}

\subsection{Finite-type structure of quotients of Cantor-like measures}

In this section, we will investigate the finite type structure and local
dimensions of quotients of the Cantor-like measures associated with the IFS 
\begin{equation}
\left\{ S_{j}(x)=\frac{1}{d}x+\frac{j(d-1)}{d}\text{ for }j\in \Lambda
\right\} \text{ and probabilities }p_{j}>0,  \label{Cantor-like}
\end{equation}%
where $\Lambda \subseteq \{0,1,...,k\}$ and $d\geq 3$. We will assume the
convex hull of $K$ is the interval $[0,k]$, equivalently, $0,k\in \Lambda $.
The strong separation condition is satisfied if $k<d-1$, so this case is
handled by Theorem \ref{strictsep} giving the following.

\begin{corollary}
Let $\mu _{\pi }$ be the quotient of the self-similar measure $\mu $
associated with the IFS (\ref{Cantor-like}) with $k<d-1$ and $%
\{0,k\}\subseteq \Lambda $. Then 
\begin{eqnarray*}
\{\dim _{loc}\mu _{\pi }(x) &:&x\in \text{supp}\mu _{\pi }\}=\left[ \frac{%
-\log (\max p_{i})}{\log {d}},\frac{-\log (\min p_{i})}{\log {d}}\right] \\
&=&\{\dim _{loc}\mu (x):x\in \text{supp}\mu \}.
\end{eqnarray*}
\end{corollary}

Thus we will assume $k\geq d-1$. If $\Lambda =\{0,1,..,k\}$, the associated
self-similar measures have support equal to the (full) interval $[0,k]$,
while the OSC fails if $k>d-1$.

In the special case that $\Lambda =\{0,1,..,k\}$ and the probabilities
satisfy $p_{j}=\binom{k}{j}2^{-k}$, then the associated self-similar measure
on $\mathbb{R}$ is the $k$-fold convolution of the uniform Cantor measure
with contraction ratio $1/d$. In this case, the corresponding quotient
measure is also equal to the $k$-fold convolution (taken on the torus) of
the uniform Cantor measure.

Cantor-like measures are of finite type on $\mathbb{R}$ (c.f. \cite{HHM, HHN}%
). In \cite{HHN} this fact was used to study their local dimensions,
extending earlier work of \cite{BHM, HL, Sh}. For instance, it was shown
that if $\Lambda =\{0,1,..,k\}$, $k\geq d$ and $p_{0}<p_{j}$ for all $j\neq
0,k$, then there is an isolated point in the set of local dimensions.

Of course, the corresponding quotient measures are also of finite type on $%
\mathbb{T}$ according to Proposition \ref{ft}. In this section we will prove
that for many of these quotient measures the set of local dimensions is a
closed interval. Typically, the essential class is the full torus (and hence
is unique), however it is not in general of positive type, so we cannot
appeal to either Theorem \ref{regprob} or \ref{notregprob}. Rather, we will
modify the previous approach.

We begin our study of these Cantor-like measures by investigating their
finite-type structures. Observe that for all these IFS the endpoints of
quotient net intervals of level $n$ belong to the set $\{jd^{-n}:j=0,...,k\}$
and their lengths are at most $kd^{-n}$. Neighbours will always belong to $%
\{0,...,k\}$.

When $\Lambda =\{0,1,...,k\}$ more can be said.

\begin{lemma}
Consider the IFS (\ref{Cantor-like}) where $\Lambda =\{0,1,...,k\}$ and $%
k\geq d-1$.

\begin{enumerate}
\item The quotient net intervals of level $n$ are the sets $%
[jd^{-n},(j+1)d^{-n}]$ for $j=0,...,d^{n}-1$.

\item There is only one reduced quotient characteristic vector, namely $%
(1,(0,1,...,k-1))$. It has $d$ (identical) children.
\end{enumerate}
\end{lemma}

\begin{proof}
The iterates of $0$ at level one are the points $j(d-1)/d$ for $j=0,...,k$.
Taking these mod $1$ gives the points $j/d$ for $j=0,...,d-1$. The iterates
of $k$ are also in $\mathbb{Z}/d$ and so taken mod $1$ give no additional
terms.

At step $n$ the iterates of $0$ are the points $%
\sum_{i=1}^{n}j_{i}d^{-i}(d-1)$, where $0\leq j_{i}$ $\leq k$ and as $k\geq
d-1,$ taking these mod $1$ again this gives all $jd^{-n}$. The iterates of $%
k $ add no new terms. Thus the net intervals at each level are as claimed.

To determine the neighbours, consider the net interval $[j/d,$ $(j+1)/d]$ of
level one. The neighbours are the integers of the form $d\left(
j/d+l-S_{J}(0)\right) $ where $l$ is an integer, $J\in \{0,1,...,k\}$ and%
\begin{equation}
\left[ \frac{j}{d}+l,\frac{j+1}{d}+l\right] \subseteq \lbrack
S_{J}(0),S_{J}(k)]=\left[ \frac{J(d-1)}{d},\frac{J(d-1)+k}{d}\right] .
\label{incl}
\end{equation}%
The inequalities implied by (\ref{incl}) ensure that such integers are
contained in $\{0,1,...,k-1\}$. Easy computations show that if $i\in
\{0,...,j\}$, $J=d-(j-i)$ and $l=d-1-(j-i)$, then $i=d\left(
j/d+l-S_{J}(0)\right) $ and all the requirements are satisfied. Similarly,
for $i\in \{j+1,...,k-1\}$, we see that $d\left( j/d+l-S_{J}(0)\right) =i$
when $J=i-j=l$. Further, all the additional requirements are met. This
proves the neighbours are precisely the set $\{0,1,...,k-1\}$. In
particular, there is are $d$ children at level one of the parent interval $%
[0,1]$, but only one reduced characteristic vector, $(1,(0,1,...,k-1))$.

The same statement holds for the higher levels because of self-similarity.
\end{proof}

\begin{remark}
We remark that even when $\Lambda $ is a proper subset of $\{0,1,...,k\}$,
the conclusions of this lemma are often true. For instance, one can check
that this is the case if $d=4$, $k=7$ and $\Lambda $ consists of all but one
integer $j$ chosen from $\{1,...,6\}$.
\end{remark}

From here on we will restrict our attention to IFS and quotients of their
self-similar measures for which the conclusion of the lemma holds. We will
refer to these as\textbf{\ }\textit{complete quotient Cantor-like}\textbf{\ }%
\textit{measures}. Note that the support of such quotient measures is the
full torus and they have $d$ primitive transition matrices that we label as $%
T(\ell )$ for $\ell =0,...,d-1,$ where $\ell $ denotes the $\ell $'th child
from left to right. These matrices are computed in the next lemma.

\begin{lemma}
\label{lem:T} The primitive transition matrices $T(\ell )$ for $\ell \in
\{0,\dots ,d-1\}$ for a complete quotient Cantor-like measure satisfying (%
\ref{Cantor-like}) are the $k\times k$ matrices with $j,i$ entry equal to 
\begin{equation*}
T(\ell )_{ji}=p_{s}\text{ if }\frac{\ell -(i-1)+(j-1)d}{d-1}=s\text{ with }%
s\in \Lambda
\end{equation*}%
and $0$ otherwise.
\end{lemma}

\begin{example}
Consider the case when $d=4$, $k=7$ and $\Lambda =\{0,1,...,7\}$. The unique
reduced characteristic vector is $(1,(0,1,...,6))$ and there are four
transition matrices

\begin{align*}
T(0)& =%
\begin{bmatrix}
p_{0} & 0 & 0 & 0 & 0 & 0 & 0 \\ 
0 & p_{1} & 0 & 0 & p_{0} & 0 & 0 \\ 
0 & 0 & p_{2} & 0 & 0 & p_{1} & 0 \\ 
p_{4} & 0 & 0 & p_{3} & 0 & 0 & p_{2} \\ 
0 & p_{5} & 0 & 0 & p_{4} & 0 & 0 \\ 
0 & 0 & p_{6} & 0 & 0 & p_{5} & 0 \\ 
0 & 0 & 0 & p_{7} & 0 & 0 & p_{6}%
\end{bmatrix}
& T(1)& =%
\begin{bmatrix}
0 & p_{0} & 0 & 0 & 0 & 0 & 0 \\ 
0 & 0 & p_{1} & 0 & 0 & p_{0} & 0 \\ 
p_{3} & 0 & 0 & p_{2} & 0 & 0 & p_{1} \\ 
0 & p_{4} & 0 & 0 & p_{3} & 0 & 0 \\ 
0 & 0 & p_{5} & 0 & 0 & p_{4} & 0 \\ 
p_{7} & 0 & 0 & p_{6} & 0 & 0 & p_{5} \\ 
0 & 0 & 0 & 0 & p_{7} & 0 & 0%
\end{bmatrix}
\\
T(2)& =%
\begin{bmatrix}
0 & 0 & p_{0} & 0 & 0 & 0 & 0 \\ 
p_{2} & 0 & 0 & p_{1} & 0 & 0 & p_{0} \\ 
0 & p_{3} & 0 & 0 & p_{2} & 0 & 0 \\ 
0 & 0 & p_{4} & 0 & 0 & p_{3} & 0 \\ 
p_{6} & 0 & 0 & p_{5} & 0 & 0 & p_{4} \\ 
0 & p_{7} & 0 & 0 & p_{6} & 0 & 0 \\ 
0 & 0 & 0 & 0 & 0 & p_{7} & 0%
\end{bmatrix}
& T(3)& =%
\begin{bmatrix}
p_{1} & 0 & 0 & p_{0} & 0 & 0 & 0 \\ 
0 & p_{2} & 0 & 0 & p_{1} & 0 & 0 \\ 
0 & 0 & p_{3} & 0 & 0 & p_{2} & 0 \\ 
p_{5} & 0 & 0 & p_{4} & 0 & 0 & p_{3} \\ 
0 & p_{6} & 0 & 0 & p_{5} & 0 & 0 \\ 
0 & 0 & p_{7} & 0 & 0 & p_{6} & 0 \\ 
0 & 0 & 0 & 0 & 0 & 0 & p_{7}%
\end{bmatrix}%
.
\end{align*}%
We remark that these are also the primitive transition matrices in the case
that $\Lambda $ omits only one integer $j$ other than $0$ or $7$, with the
understanding that the entries denoted $p_{j}$ have value zero.
\end{example}

\begin{proof}[Proof of Lemma \protect\ref{lem:T}]
The $\ell $'th child of parent $[0,1]$ is the net interval $\left[ \frac{%
\ell }{d},\frac{\ell +1}{d}\right] $. The $j$'th neighbour of $[0,1]$ is $%
j-1 $ and the $i$'th neighbour of $\left[ \frac{\ell }{d},\frac{\ell +1}{d}%
\right] $ is $i-1$. The entry $T(\ell )_{ji}$ will be $p_{s}$ where $s\in
\{0,...,k\}$ is such that 
\begin{equation*}
0-(j-1)+\frac{s(d-1)}{d}=\frac{l}{d}-\frac{(i-1)}{d}.
\end{equation*}%
Solving for $s$ yields the desired result.
\end{proof}

The transition matrices of complete quotient Cantor-like measures have not
only a non-zero entry in each column, but also a non-zero entry in each row
as the support of the quotient measure is full.

We next permute the rows and columns of these transition matrices to produce
matrices ${{\tilde{T}}(\ell )}$ with a natural block structure, where block $%
(i,j)$ for $i,j\in \{1,...,d-1\}$ has size $\left( \left\lfloor \frac{k-i}{%
d-1}\right\rfloor +1\right) \times \left( \left\lfloor \frac{k-j}{d-1}%
\right\rfloor +1\right) $. The $(i^{\prime },j^{\prime })$ entry of block $%
(i,j)$ of ${{\tilde{T}}(\ell )}$ will have as its entry the $(i+i^{\prime
}(d-1),j+j^{\prime }(d-1))$ entry of $T({\ell )}$, $i^{\prime }$ and $%
j^{\prime }$ being indexed beginning at $0$. Doing this permutation is not
necessary, but it makes the algebraic manipulations simpler in what follows.

Given such a block matrix $B,$ we will write $B(i,j)$ for the $(i,j)$ block.
We will say that block matrix $B$ (with this structure) is of \textit{type} $%
r$ if $B(i,j)\neq 0$ only if $j-i\equiv r \mod{(d-1)}$.

It is easy to see that the permuted transition matrix $\widetilde{T}(\ell )$
is type $\ell $ (mod $(d-1)$). Furthermore, $\widetilde{T}(\ell )$ has the
special property that if $j-i\equiv \ell \mod{(d-1)}$, then the matrix $%
\widetilde{T}(\ell )(i,j)$ has at least one non-zero entry in each row and
column.

\begin{example}
Consider a complete quotient Cantor-like measure with $d=4$ and $k=7$.
Permuting the rows and columns yields: 
\begin{align*}
\tilde{T}(0)& =\left[ 
\begin{array}{lll|ll|ll}
p_{0} & 0 & 0 & 0 & 0 & 0 & 0 \\ 
p_{4} & p_{3} & p_{2} & 0 & 0 & 0 & 0 \\ 
0 & p_{7} & p_{6} & 0 & 0 & 0 & 0 \\ \hline
0 & 0 & 0 & p_{1} & p_{0} & 0 & 0 \\ 
0 & 0 & 0 & p_{5} & p_{4} & 0 & 0 \\ \hline
0 & 0 & 0 & 0 & 0 & p_{2} & p_{1} \\ 
0 & 0 & 0 & 0 & 0 & p_{6} & p_{5}%
\end{array}%
\right] & \tilde{T}(1)& =\left[ 
\begin{array}{lll|ll|ll}
0 & 0 & 0 & p_{0} & 0 & 0 & 0 \\ 
0 & 0 & 0 & p_{4} & p_{3} & 0 & 0 \\ 
0 & 0 & 0 & 0 & p_{7} & 0 & 0 \\ \hline
0 & 0 & 0 & 0 & 0 & p_{1} & p_{0} \\ 
0 & 0 & 0 & 0 & 0 & p_{5} & p_{4} \\ \hline
p_{3} & p_{2} & p_{1} & 0 & 0 & 0 & 0 \\ 
p_{7} & p_{6} & p_{5} & 0 & 0 & 0 & 0%
\end{array}%
\right] \\
\tilde{T}(2)& =\left[ 
\begin{array}{lll|ll|ll}
0 & 0 & 0 & 0 & 0 & p_{0} & 0 \\ 
0 & 0 & 0 & 0 & 0 & p_{4} & p_{3} \\ 
0 & 0 & 0 & 0 & 0 & 0 & p_{7} \\ \hline
p_{2} & p_{1} & p_{0} & 0 & 0 & 0 & 0 \\ 
p_{6} & p_{5} & p_{4} & 0 & 0 & 0 & 0 \\ \hline
0 & 0 & 0 & p_{3} & p_{2} & 0 & 0 \\ 
0 & 0 & 0 & p_{7} & p_{6} & 0 & 0%
\end{array}%
\right] & \tilde{T}(3)& =\left[ 
\begin{array}{lll|ll|ll}
p_{1} & p_{0} & 0 & 0 & 0 & 0 & 0 \\ 
p_{5} & p_{4} & p_{3} & 0 & 0 & 0 & 0 \\ 
0 & 0 & p_{7} & 0 & 0 & 0 & 0 \\ \hline
0 & 0 & 0 & p_{2} & p_{1} & 0 & 0 \\ 
0 & 0 & 0 & p_{6} & p_{5} & 0 & 0 \\ \hline
0 & 0 & 0 & 0 & 0 & p_{3} & p_{2} \\ 
0 & 0 & 0 & 0 & 0 & p_{7} & p_{6}%
\end{array}%
\right].
\end{align*}
\end{example}

Hereafter, when we speak of a \textit{block matrix} we will mean a matrix $B$
with this block structure of type $\ell $ for some $\ell $ and having the
property that $B(i,j)$ has a non-zero entry in each row and column if $%
j-i\equiv \ell \mod{(d-1)}$. The block matrices of type $0$ will be called 
\textit{block diagonal}. By a\textit{\ block positive} matrix we will mean a
block matrix of type $\ell $ with all entries of $B(i,j)$ strictly positive
for $j-i\equiv \ell \mod{(d-1)}$. It is worth noting that a block matrix of
type $\ell $ is also of type $\ell ^{\prime }$ for all $\ell ^{\prime
}\equiv \ell \mod{(d-1)}$.

We will say that a periodic point with period $\theta $ is a \textit{block
diagonal (positive) periodic point} if the block matrix $\widetilde{T}%
(\theta )$ is block diagonal (and block positive).

We record here some routine facts about block matrices (with this
definition) that follow from block multiplication and the fact that each
non-zero block has a non-zero entry in each row and column.

\begin{lemma}
\label{prodbd}Suppose $A,B,C$ are block matrices.

(i) If $A,B$ are types $a,b$ respectively, then $AB$ is a block matrix of
type $(a+b)$ mod$(d-1)$.

(ii) If $B$ is block positive, then $AB$ and $BA$ are block positive.

(iii) There are positive constants $c_{1},c_{1}^{\prime },$ depending on $B,$
such that 
\begin{equation*}
c_{1}\left\Vert A\right\Vert \leq \min (\left\Vert AB\right\Vert ,\left\Vert
BA\right\Vert )\leq \max (\left\Vert AB\right\Vert ,\left\Vert BA\right\Vert
)\leq c_{1}^{\prime }\left\Vert A\right\Vert .
\end{equation*}

(iv) If $B$ is block positive, then there is a positive constant $c_{0},$
depending on $B$, such that 
\begin{equation*}
\left\Vert ABC\right\Vert \geq c_{0}\left\Vert A\right\Vert \left\Vert
C\right\Vert .
\end{equation*}
\end{lemma}

Being a product of block matrices, every permuted transition matrix of
complete quotient Cantor-like measures \textit{is} a block matrix.

Block positive matrices have further good properties.

\begin{lemma}
\label{positivesp}Suppose $B$ is a block positive matrix.

(i) If $AB$ is block diagonal, then there is a constant $c=c(B)$ such that 
\begin{equation*}
sp(AB)\leq \left\Vert AB\right\Vert \leq c\cdot sp(AB).
\end{equation*}

(ii) If $B$ is block diagonal, then there is a constant $c_{2}=c_{2}(B)$
such that 
\begin{equation*}
\left\Vert B^{n}\right\Vert \leq c_{2}sp(B^{n})\text{ for all }n.
\end{equation*}%
Further, there exists an index $j$ such that $sp(B^{n})=sp(B(j,j))^{n}$ for
all $n$.
\end{lemma}

\begin{proof}
(i) Assume $A$ is type $a$, so $B$ is type $-a$. Then $B(a+j,j)$ is a
positive matrix for each $j$ (here $a+j$ should be understood mod$(d-1)$)
and hence by \cite[Lemma 3.15(2)]{HHM} there are constants $C_{j}$ such that 
\begin{eqnarray*}
\left\Vert A(j,a+j)B(a+j,j)\right\Vert &\leq &C_{j}sp(A(j,a+j)B(a+j,j)) \\
&\leq &C_{j}sp(AB(j,j)).
\end{eqnarray*}%
Since $AB$ is diagonal, this is dominated by $C_{j}sp(AB)$. But%
\begin{equation*}
\left\Vert AB\right\Vert =\sum_{j=1}^{d-1}\left\Vert AB(j,j)\right\Vert
=\sum_{j}\left\Vert A(j,a+j)B(a+j,j)\right\Vert ,
\end{equation*}%
so we can take $c=\max_{j}C_{j}(d-1).$

(ii) As each square matrix $B(j,j)$ is positive, \cite[Lemma 3.15(3)]{HHM}
implies that for each $j$ there is a constant $C_{j}^{\prime }$ such that $%
\left\Vert B(j,j)^{n}\right\Vert \leq C_{j}^{\prime }sp(B(j,j))^{n}\leq
C_{j}^{\prime }sp(B^{n})$ for each $n$. Hence 
\begin{equation*}
\left\Vert B^{n}\right\Vert =\sum_{j}\left\Vert B(j,j)^{n}\right\Vert \leq
\max_{j}C_{j}^{\prime }(d-1)sp(B^{n})\leq c_{2}sp(B^{n}).
\end{equation*}%
The final comment follows because $B$ is block diagonal.
\end{proof}

\begin{lemma}
Suppose $\Lambda =\{0,...,k\}$. The permuted transition matrix $(\widetilde{T%
}(0)$ $\widetilde{T}(d-1))^{k}$ is block diagonal and positive.
\end{lemma}

\begin{proof}
It is easiest to see this if we look at $(T(0)T(d-1))^{k}$. We must show
that if $j\equiv i \mod{(d-1)}$, then $\left( (T(0)T(d-1))^{k}\right)
_{ij}>0 $. Let $\ell _{0}=i$ and $\ell _{2k}=j$. Note that 
\begin{align*}
\left( (T(0)T(d-1))^{k}\right) _{ij}& =\sum_{\ell _{1}=1}^{k}\dots
\sum_{\ell _{2k-1}=1}^{k}T(0)_{i,\ell _{1}}T(d-1)_{\ell _{1},\ell
_{2}}\cdots T(0)_{\ell _{2k-2},\ell _{2k-1}}T(d-1)_{\ell _{2k-1},j} \\
& =\sum_{\ell _{1}=1}^{k}\dots \sum_{\ell
_{2k-1}=1}^{k}\prod_{r=1}^{2k}p_{\ell _{r-1}-1-\frac{\ell _{r}-\ell _{r-1}}{%
d-1}+\tau (r)}
\end{align*}%
where $\tau (r)=0$ if $r$ is odd and $\tau (r)=1$ if $r$ is even. It
suffices to show that there exists $\ell _{1},\ell _{2},\dots ,\ell _{2k-1}$
such that 
\begin{equation}
\prod_{r=1}^{2k}p_{\ell _{r-1}-1-\frac{\ell _{r}-\ell _{r-1}}{d-1}+\tau
(r)}>0.  \label{eq:prod}
\end{equation}%
We will define the $\ell _{r}$ inductively. We set $\ell _{0}=i$. We define 
\begin{equation*}
\ell _{r}=\left\{ 
\begin{matrix}
\ell _{r-1}+(d-1) & \text{if }\ell _{r-1}+(d-1)\leq j\text{ and }\ell
_{r-1}+\tau (r)\geq 2 \\ 
\ell _{r-1}-(d-1) & \text{if }\ell _{r-1}-(d-1)\geq j\text{ and }\ell
_{r-1}+\tau (r)\leq k \\ 
\ell _{r-1} & \text{otherwise}%
\end{matrix}%
\right. .
\end{equation*}

Recall that $1\leq i,j\leq k$.

\begin{itemize}
\item If $\ell _{r}=\ell _{r-1}+(d-1)$, then $\ell _{r-1}+\tau (r)\geq 2$
and $\ell _{r-1}<j$ so $\ell _{r-1}+\tau (r)\leq j\leq k$. Thus $\ell
_{r-1}-1-\frac{\ell _{r}-\ell _{r-1}}{d-1}+\tau (r)=\ell _{r-1}-1-1+\tau
(r)\in \{0,1,\cdots ,k-1\},$ hence the associated probability is non-zero.

\item If $\ell _{r}=\ell _{r-1}-(d-1),$ then $\ell _{r-1}+\tau (r)\leq k$
and $\ell _{r-1}>j,$ so $\ell _{r-1}-1-\frac{\ell _{r}-\ell _{r-1}}{d-1}%
+\tau (r)=\ell _{r-1}-1+1+\tau (r)\in \{1,\cdots ,k\}$ and hence the
associated probability is non-zero.

\item If $\ell _{r}=\ell _{r-1}$ then $\ell _{r-1}-1-\frac{\ell _{r}-\ell
_{r-1}}{d-1}+\tau (r)=\ell _{r-1}-1+\tau (r)\in \{0,1,\cdots ,k\}$ and again
the associated probability is non-zero.
\end{itemize}

We note that for $r$ even we always have $\ell _{r-1}+\tau (r)\geq 2$ and
for $r$ odd we always have $\ell _{r-1}+\tau (r)\leq k$. If $i=$ $\ell
_{0}<j,$ say $j=i+m(d-1)$ for $m\geq 1$, then $\ell _{t}=j$ for $t\geq 2m$
and therefore $\ell _{2k}=j$. A similar statement holds if $i=\ell _{0}>j$
and, of course, if $i=j$, then $\ell _{r}=i=j$ for all $r$.

Hence equation \eqref{eq:prod} is satisfied which proves the result.
\end{proof}

When $\Lambda $ is a proper subset of $\{0,...,k\}$, it is still possible
for this transition matrix to be positive. For example, this can easily be
seen to be true in the case that $d=4,k=7$, $p_{j}=0$ for one of $j=1,...,6$.

Another simple, but useful, fact is the following. By an \textit{interior
periodic point}, we mean a periodic point that is not a boundary point.

\begin{lemma}
\label{finiteset}Assume the complete quotient Cantor-like measure admits a
permuted transition matrix that is block positive. Then there is a finite
set $\mathcal{F}$ of block (permuted) transition matrices such that for each
block matrix $A$ there is some $B\in \mathcal{F}$ so that $AB$ and $BA$ are
block positive and diagonal and, furthermore, any periodic point with period 
$\theta $ satisfying $\widetilde{T}(\theta )=AB$ or $BA$ is an interior
periodic point.
\end{lemma}

\begin{proof}
Let $F$ denote any block diagonal and positive permuted transition matrix.
This is guaranteed to exist since the product of a block positive matrix
with any block matrix is still block positive. For each $r=1,...,d-1$, let $%
B_{r}$ be a block matrix of type $r$. Put $\mathcal{F}%
=\{FB_{1}B_{r-1}:r=0,1,...,d-2\}$. The matrices $FB_{1}B_{r-1}$ are block
positive and type $r$. If $A$ is any block matrix, it has type $-r$ for some 
$r=1,...,d-1$ and hence $FB_{1}B_{r-1}A$ and $AFB_{1}B_{r-1}$ are block
diagonal, positive matrices. The choice of $B_{1}B_{r-1}$ ensures that they
give rise to interior periodic points.
\end{proof}

\subsection{Local dimensions for complete quotient Cantor-like measures}

\label{sec:4.2}

In this section we prove that the set of local dimensions for all complete
quotient Cantor-like measures which admit a block positive, transition
matrix is an interval. We will also prove that the local dimensions at the
interior, positive periodic points are dense.

We begin by proving that the local dimensions at periodic points are dense
in the set of all local dimensions.

\begin{theorem}
Assume $\mu _{\pi }$ is a complete quotient Cantor-like measure which admits
a block positive, transition matrix. Then the set of local dimensions at
block diagonal, positive, interior periodic points is dense in the set of
all local (upper or lower) dimensions of $\mu _{\pi }$.
\end{theorem}

\begin{proof}
Fix $x\in \mathbb{T}$. We will see how to approximate the lower local
dimension of $\mu _{\pi }$ at $x$ by the local dimensions of block diagonal,
positive, interior periodic points. The other cases are similar.

Fix a subsequence $(n_{\ell })$ such that%
\begin{equation*}
\underline{\dim }_{loc}\mu _{\pi }(x)=\lim_{\ell }\frac{\log M_{n_{\ell }}(x)%
}{n_{\ell }\log (1/d)}.
\end{equation*}%
For each $n_{\ell },$ let $\Delta _{n_{\ell }}^{\prime }$ denote a choice of 
$\Delta _{n_{\ell }}(x)$, $\Delta _{n_{\ell }}^{+}(x)$, $\Delta _{n_{\ell
}}^{-}(x)$ such that $\mu _{\pi }(\Delta _{n_{\ell }}^{\prime })\leq
M_{n_{\ell }}(x)\leq 3\mu _{\pi }(\Delta _{n_{\ell }}^{\prime })$ and let $%
(\gamma _{0}^{(\ell )},...,\gamma _{n_{\ell }}^{(\ell )})$ be the symbolic
representation of $\Delta _{n_{\ell }}^{\prime }$. Then 
\begin{equation*}
\underline{\dim }_{loc}\mu (x)=\lim_{\ell }\frac{\log \left\Vert \widetilde{T%
}(\gamma _{0}^{(\ell )},...,\gamma _{n_{\ell }}^{(\ell )})\right\Vert }{%
n_{\ell }\log (1/d)}.
\end{equation*}%
For notational ease, let $A_{\ell }=\widetilde{T}(\gamma _{0}^{(\ell
)},...,\gamma _{n_{\ell }}^{(\ell )})$ and choose a block positive matrix $%
B_{\ell }$ from the finite set $\mathcal{F}$ of Lemma \ref{finiteset} so
that $A_{\ell }B_{\ell }$ is block diagonal and if the periodic point $%
y_{\ell }$ with period $\theta _{\ell }$ satisfies $\widetilde{T}(\theta
_{\ell })=A_{\ell }B_{\ell }$, then $y_{\ell }$ is a block diagonal,
positive, interior periodic point.

By Lemma \ref{positivesp} there is a constant $c,$ which depends only on the
finite set $\mathcal{F}$, such that 
\begin{equation*}
sp(A_{\ell }B_{\ell })\leq \left\Vert A_{\ell }B_{\ell }\right\Vert \leq
c\cdot sp(A_{\ell }B_{\ell })\text{.}
\end{equation*}%
Thus for each $\ell $ there is a constant $c_{\ell }$ so that $c_{\ell
}\left\Vert A_{\ell }B_{\ell }\right\Vert =sp(A_{\ell }B_{\ell }),$ where $%
c_{\ell }$ is bounded above and below from zero. Similarly, by Lemma \ref%
{prodbd} there is a constant $c_{\ell }^{\prime }$, bounded above and below,
such that $\left\Vert A_{\ell }B_{\ell }\right\Vert =c_{\ell }^{\prime
}\left\Vert A_{\ell }\right\Vert $. Moreover, 
\begin{equation*}
\dim \mu _{\pi }(y_{\ell })=\frac{\log sp(A_{\ell }B_{\ell })}{L(\theta
_{\ell }^{-})\log (1/d)}=\frac{\log (c_{\ell }c_{\ell }^{\prime }\left\Vert
A_{\ell }\right\Vert )}{L(\theta _{\ell }^{-})\log (1/d)}.
\end{equation*}%
But $L(\theta _{\ell }^{-})=n_{\ell }+L(B_{\ell })$ (where by $L(B_{\ell })$
we mean the number of permuted primitive transition matrices whose product
is $B_{\ell }$) and $L(B_{\ell })$ is bounded because $B_{\ell }$ is chosen
from a finite set. Thus an easy calculation shows that 
\begin{equation*}
\left\vert \dim \mu _{\pi }(y_{\ell })-\frac{\log \left\Vert A_{\ell
}\right\Vert )}{n_{\ell }\log (1/d)}\right\vert \rightarrow 0\text{ as }\ell
\rightarrow \infty
\end{equation*}%
and therefore $\dim \mu _{\pi }(y_{\ell })\rightarrow \underline{\dim }%
_{loc}\mu _{\pi }(x)$.
\end{proof}

\begin{theorem}
Assume $\mu _{\pi }$ is a complete quotient Cantor-like measure which admits
a block positive transition matrix. Suppose $y,z$ are block diagonal,
positive, interior periodic points. Given any real number $R,$ with $\dim
_{loc}\mu _{\pi }(y)<R<\dim _{loc}\mu _{\pi }(z),$ and $\varepsilon >0$
there is a periodic point $x$ with 
\begin{equation*}
\left\vert R-\dim _{loc}\mu _{\pi }(x)\right\vert <\varepsilon \text{.}
\end{equation*}
\end{theorem}

\begin{proof}
Assume $y$ has period $\phi $ and $z$ period $\theta $ where $A=\widetilde{T}%
(\phi )$ and $B=\widetilde{T}(\theta )$ are block diagonal and positive.
Given $R$ as above, choose $0<t<1$ such that $R=t\dim _{loc}\mu _{\pi
}(y)+(1-t)\dim _{loc}\mu _{\pi }(z)$. Choose integers $n_{\ell },m_{\ell
}\rightarrow \infty $ such that 
\begin{equation*}
\frac{L(\phi ^{-})n_{\ell }}{L(\phi ^{-})n_{\ell }+L(\theta ^{-})m_{\ell }}%
\rightarrow t\text{.}
\end{equation*}%
Let $\alpha =sp(A)$ and $\beta =sp(B)$ and assume $\alpha =sp(A(j,j))$ and $%
\beta =sp(B(i,i))$. With this notation, 
\begin{equation*}
R=\frac{t\log \alpha }{L(\phi ^{-})\log (1/d)}+\frac{(1-t)\log \beta }{%
L(\theta ^{-})\log (1/d)}.
\end{equation*}

Choose block positive matrices $M,N$ from the finite set $\mathcal{F}$ of
Lemma \ref{finiteset} so that $M$ is type $j-i$ and $N$ is type $i-j$.

By Lemmas \ref{prodbd} and \ref{positivesp} and the block structure, there
are positive constants $c,c^{\prime },$ independent of $\ell $, which may
vary from one occurrence to another, such that 
\begin{eqnarray*}
sp(B^{m_{\ell }}MA^{n_{\ell }}N) &\geq &c\left\Vert B^{m_{\ell }}MA^{n_{\ell
}}\right\Vert \geq c\left\Vert B^{m_{\ell }}MA^{n_{\ell }}(i,j)\right\Vert \\
&=&c\left\Vert B^{m_{\ell }}(i,i)M(i,j)A^{n_{\ell }}(j,j)\right\Vert \\
&\geq &c\left\Vert B^{m_{\ell }}(i,i)\right\Vert \left\Vert A^{n_{\ell
}}(j,j)\right\Vert \geq c\beta ^{m_{\ell }}\alpha ^{n_{\ell }}.
\end{eqnarray*}%
On the other hand, as $A,B$ are both block diagonal and positive, we also
have%
\begin{equation*}
sp(B^{m_{\ell }}MA^{n_{\ell }}N)\leq c^{\prime }\left\Vert B^{m_{\ell
}}\right\Vert \left\Vert A^{n_{\ell }}\right\Vert \leq c^{\prime }\beta
^{m_{\ell }}\alpha ^{n_{\ell }}\text{.}
\end{equation*}%
It follows that if we let $x_{\ell }$ be the interior periodic point with
period $B^{m_{\ell }}MA^{n_{\ell }}N$, then%
\begin{eqnarray*}
\lim_{\ell }\dim _{loc}\mu _{\pi }(x_{\ell }) &=&\lim_{\ell }\frac{\log
sp(B^{m_{\ell }}MA^{n_{\ell }}N)}{(L(\theta ^{-})m_{\ell }+L(\phi
^{-})n_{\ell }+L(M)+L(N))\log (1/d)} \\
&=&\lim_{\ell }\frac{m_{\ell }\log \beta +n_{\ell }\log \alpha }{(L(\theta
^{-})m_{\ell }+L(\phi ^{-})n_{\ell })\log (1/d)}=R.
\end{eqnarray*}
\end{proof}

\begin{theorem}
Assume $\mu _{\pi }$ is a complete quotient Cantor-like measure which admits
a block positive transition matrix. Suppose $x_{n}$ are block diagonal,
positive, interior periodic points. Then there is some $x\in \mathbb{T}$
such that 
\begin{eqnarray*}
\overline{\dim }_{loc}\mu _{\pi }(x) &=&\limsup_{n}\dim _{loc}\mu _{\pi
}(x_{n}) \\
\underline{\dim }_{loc}\mu _{\pi }(x) &=&\liminf_{n}\dim _{loc}\mu _{\pi
}(x_{n}).
\end{eqnarray*}
\end{theorem}

\begin{proof}
Assume $x_{n}$ has period $\theta _{n}$ with $\widetilde{T}(\theta _{n})$
block diagonal and positive. Let 
\begin{equation*}
S=\overline{\lim }\frac{\log (sp\widetilde{T}(\theta _{n}))}{L(\theta
_{n}^{-})\log (1/d)}\text{ and }I=\underline{\lim }\frac{\log (sp\widetilde{T%
}(\theta _{n}))}{L(\theta _{n}^{-})\log (1/d)}.
\end{equation*}%
By passing to a subsequence, we can assume 
\begin{equation*}
\dim _{loc}\mu _{\pi }(x_{2n})=\frac{\log (sp\widetilde{T}(\theta _{2n}))}{%
L(\theta _{2n}^{-})\log (1/d)}\rightarrow S
\end{equation*}%
and%
\begin{equation*}
\dim _{loc}\mu _{\pi }(x_{2n+1})=\frac{\log (sp\widetilde{T}(\theta _{2n+1}))%
}{L(\theta _{2n+1}^{-})\log (1/d)}\rightarrow I.
\end{equation*}

We will define a rapidly increasing sequence $(i_{n})$ and then put $x$ to
be 
\begin{equation*}
x=(\gamma _{0},\underbrace{\theta _{1},...,\theta _{1}}_{i_{1}},\underbrace{%
\theta _{2},...,\theta _{2}}_{i_{2}},\underbrace{\theta _{3},...,\theta _{3}}%
_{i_{3}},...).
\end{equation*}%
The sequence $(i_{n})$ will be inductively defined, with $i_{n}$ depending
on $i_{j},\theta _{j}$ for $j=1,...,n-1$, $\theta _{n}$ and $\theta _{n+1}$.
The choice will be clear from the arguments that follow.

The proof that $x$ has the required properties is similar to that of \cite[%
Theorem 5.5]{HHM} and \cite[Theorem 3.13]{HHN}. We only sketch the main
ideas here.

Put $N_{n}=1+\sum_{j=1}^{n}i_{j}L(\theta _{j}^{-})$ so that $\Delta
_{N_{n}}(x)=(\gamma _{0},\underbrace{\theta _{1},...,\theta _{1}}_{i_{1}},%
\underbrace{\theta _{2},...,\theta _{2}}_{i_{2}},...,\underbrace{\theta
_{n},...,\theta _{n}}_{i_{n}})$.

Since $x_{n}$ is an interior point, $\Delta _{N_{n}}(x),$ $\Delta
_{N_{n}}^{+}(x),$ $\Delta _{N_{n}}^{-}(x)$ have a common ancestor $(\gamma
_{0},\underbrace{\theta _{1},...,\theta _{1}}_{i_{1}},\underbrace{\theta
_{2},...,\theta _{2}}_{i_{2}},...,\underbrace{\theta _{n},...,\theta _{n}}%
_{i_{n}-1})=\Delta _{N_{n}-L(\theta _{n}^{-})}(x)$. Applying Corollary \ref%
{measurenetint} and Lemma \ref{prodbd} there are constants $a_{n},b_{n}$, $%
c_{n}$ depending on $i_{j},\theta _{j},$ $j=1,...,n-1$ and $\theta _{n}$, so
that 
\begin{eqnarray*}
M_{N_{n}}(x) &=&a_{n}\mu _{\pi }(\Delta _{N_{n}-L(\theta _{n}^{-})}(x)) \\
&=&b_{n}\left\Vert \widetilde{T}(\gamma _{0},\underbrace{\theta
_{1},...,\theta _{1}}_{i_{1}},...,\underbrace{\theta _{n},...,\theta _{n}}%
_{i_{n}-1})\right\Vert =c_{n}\left\Vert \widetilde{T}(\theta
_{n})^{i_{n}-1}\right\Vert .
\end{eqnarray*}%
By choosing $i_{n}$ large enough we can ensure that 
\begin{equation*}
\left\vert \frac{\log \left\Vert \widetilde{T}(\theta
_{n})^{i_{n}-1}\right\Vert }{L(\theta _{n}^{-})(i_{n}-1)}-\frac{\log (sp%
\widetilde{T}(\theta _{n}))}{L(\theta _{n}^{-})}\right\vert \rightarrow 0,
\end{equation*}%
\begin{equation*}
\frac{\left\vert a_{n}\right\vert +\left\vert b_{n}\right\vert +\left\vert
c_{n}\right\vert }{i_{n}}\rightarrow 0
\end{equation*}%
and 
\begin{equation*}
\frac{L(\theta _{n}^{-})(i_{n}-1)}{N_{n}}\rightarrow 1.
\end{equation*}%
Consequently, 
\begin{equation*}
\left\vert \frac{\log M_{N_{n}}(x)}{N_{n}\log (1/d)}-\frac{\log (sp%
\widetilde{T}(\theta _{n}))}{L(\theta _{n}^{-})\log (1/d)}\right\vert
\rightarrow 0
\end{equation*}%
as $n\rightarrow \infty $. It follows that $\frac{\log M_{N_{n}}(x)}{%
N_{n}\log (1/d)}$ tends to $S$ and $I$ along the even and odd subsequences.
This shows that $\overline{\dim }_{loc}\mu _{\pi }(x)\geq S$ and $\underline{%
\dim }_{loc}\mu _{\pi }(x)\leq I$.

Now consider arbitrary 
\begin{equation*}
N=1+\sum_{j=1}^{n}i_{j}L(\theta _{j}^{-})+q_{n+1}L(\theta _{n+1}^{-})+r
\end{equation*}%
where $0\leq q_{n+1}<i_{n+1}$, $0\leq r<L(\theta _{n+1}^{-})$ and $r>0$ if $%
q_{n+1}=0$. Then $\Delta _{N}(x),$ $\Delta _{N}^{+}(x),$ $\Delta _{N}^{-}(x)$
have common ancestor $\Delta _{J_{N}}(x)$ where 
\begin{equation*}
J_{N}=1+\sum_{j=1}^{n}i_{j}L(\theta _{j}^{-})+(q_{n+1}-1)L(\theta _{n+1}^{-})
\end{equation*}%
if $q_{n+1}>1$ and $J_{N}=1+\sum_{j=1}^{n-1}i_{j}L(\theta
_{j}^{-})+(i_{n}-1)L(\theta _{n}^{-})$ otherwise. Similar arguments to above
show that we should study the limiting behaviour of%
\begin{equation*}
\frac{\log \mu _{\pi }(\Delta _{J_{N}}(x))}{N\log (1/d)},
\end{equation*}%
hence it suffices to study the limiting behaviour of 
\begin{equation*}
E_{n}=\frac{\log \left\Vert \widetilde{T}(\gamma _{0},\underbrace{\theta
_{1},...,\theta _{1}}_{i_{1}},...,\underbrace{\theta _{n},...,\theta _{n}}%
_{i_{n}},\underbrace{\theta _{n+1},...,\theta _{n+1}}_{q_{n+1}-1})\right%
\Vert }{N\log (1/d)},
\end{equation*}%
(with suitable modifications if $q_{n+1}=0,1$). Applying Lemma \ref{prodbd}%
(iv), with the positive matrix $\widetilde{T}(\theta _{n})$ as the matrix $%
B, $ there is a constant $c_{0}(n),$ depending on $i_{j},\theta _{j}$, $%
j=1,...,n-1$ and $\theta _{n}$ such that $E_{n}$ dominates%
\begin{eqnarray*}
&&\frac{\log c_{0}(n)+\log \left\Vert \widetilde{T}(\theta
_{n})^{i_{n}-1}\right\Vert +\log \left\Vert \widetilde{T}(\theta
_{n+1})^{q_{n+1}-1}\right\Vert }{N\log (1/d)} \\
&\geq &\frac{\log c_{0}(n)}{N\log (1/d)}+\frac{\log sp(\widetilde{T}(\theta
_{n}))^{i_{n}-1}}{N\log (1/d)}+\frac{\log sp(\widetilde{T}(\theta
_{n+1}))^{q_{n+1}-1}}{N\log (1/d)}.
\end{eqnarray*}%
Choose a sequence $(\varepsilon _{n})$ tending to $0.$ For $i_{n}$
sufficiently large, $\left\vert \frac{\log c_{0}(n)}{N\log (1/d)}\right\vert
<\varepsilon _{n}$. Furthermore, with possibly larger $i_{n}$ we have 
\begin{eqnarray*}
\frac{\log sp(\widetilde{T}(\theta _{n+1}))^{q_{n+1}-1}}{N\log (1/d)} &=&%
\frac{\log sp(\widetilde{T}(\theta _{n+1}))^{q_{n+1}-1}}{(q_{n+1}-1)L(\theta
_{n+1}^{-})\log (1/d)}\frac{(q_{n+1}-1)L(\theta _{n+1}^{-})}{N} \\
&\geq &\frac{\log sp(\widetilde{T}(\theta _{n+1}))}{L(\theta _{n+1}^{-})\log
(1/d)}t_{n}-\varepsilon _{n}
\end{eqnarray*}%
for 
\begin{equation*}
t_{n}=\frac{q_{n+1}L(\theta _{n+1}^{-})+r}{N}.
\end{equation*}%
As $1-t_{n}=1+\sum_{j=1}^{n}i_{j}L(\theta _{j}^{-})$ , we similarly have 
\begin{equation*}
\frac{\log sp(\widetilde{T}(\theta _{n}))^{i_{n}-1}}{N\log (1/d)}\geq \frac{%
\log sp(\widetilde{T}(\theta _{n}))}{L(\theta _{n}^{-})\log (1/d)}%
(1-t_{n})-\varepsilon _{n}.
\end{equation*}%
Thus 
\begin{equation*}
E_{n}\geq \frac{\log sp(\widetilde{T}(\theta _{n}))}{L(\theta _{n}^{-})\log
(1/d)}(1-t_{n})+\frac{\log sp(\widetilde{T}(\theta _{n+1}))}{L(\theta
_{n+1}^{-})\log (1/d)}t_{n}-3\varepsilon _{n}.
\end{equation*}

Appealing to Lemma \ref{positivesp}(ii), we similarly deduce that for large
enough $i_{n}$ 
\begin{eqnarray*}
E_{n} &\leq &\frac{\log \left\Vert \widetilde{T}(\theta
_{n})^{i_{n}}\right\Vert +\log \left\Vert \widetilde{T}(\theta
_{n+1})^{q_{n+1}-1}\right\Vert }{N\log (1/d)}+\varepsilon _{n} \\
&\leq &\frac{\log sp(\widetilde{T}(\theta _{n}))^{i_{n}}}{N\log (1/d)}+\frac{%
\log sp(\widetilde{T}(\theta _{n+1}))^{q_{n+1}-1}}{N\log (1/d)}+2\varepsilon
_{n} \\
&\leq &\frac{\log sp(\widetilde{T}(\theta _{n}))}{L(\theta _{n}^{-})\log
(1/d)}(1-t_{n})+\frac{\log sp(\widetilde{T}(\theta _{n+1}))}{L(\theta
_{n+1}^{-})\log (1/d)}t_{n}+3\varepsilon _{n}.
\end{eqnarray*}%
Together these estimates show that $E_{n}$ lies within $3\varepsilon _{n}$
of the same convex combinations of $\frac{\log sp(\widetilde{T}(\theta _{n}))%
}{L(\theta _{n}^{-})\log (1/d)}$ and $\frac{\log sp(\widetilde{T}(\theta
_{n+1}))}{L(\theta _{n+1}^{-})\log (1/d)}$. This proves that $\lim \inf
E_{n} $ and $\lim \sup E_{n}$ lie in the interval $[I,S]$ and hence the same
is true for $\underline{\dim }_{loc}\mu _{\pi }(x)$ and $\overline{\dim }%
_{loc}\mu _{\pi }(x).$
\end{proof}

Combining these results we immediately deduce the following.

\begin{corollary}
The set of (upper, lower) local dimensions of any complete quotient
Cantor-like measure that admits a block positive transition matrix is the
closed interval which is the closure of the set of local dimensions at block
diagonal, positive, interior periodic points. In particular, if $\mu _{\pi }$
is the quotient of the self-similar measure $\mu $ associated with the IFS (%
\ref{Cantor-like}) with $k\geq d-1$ and $\Lambda =\{0,1,...,k\}$, then the
set of local dimensions is a closed interval.
\end{corollary}

\subsection{Bounds on the local dimensions of Cantor-like measures}

In \cite{FHJ} it was shown that the set of local dimensions of the quotient
of the $3$-fold convolution of the middle third Cantor measure $\nu $, is a
proper subset of the interval component of the set of local dimensions of $%
\nu $. In this subsection, we will see that this is true, more generally,
for quotients of $(d+k)$-fold convolutions of uniform Cantor measures with
contraction factor $1/d$, provided $d$ is sufficiently large.

For matrix $M,$ let $\left\Vert M\right\Vert _{\min
}=\min_{j}\sum_{i}\left\vert M_{ij}\right\vert $ denote the minimum column
sum. Of course, $\left\Vert M\right\Vert \geq \left\Vert M\right\Vert _{\min
}$ and in \cite{HHN} it is observed that $\left\Vert M_{1}M_{2}\right\Vert
_{\min }\geq \left\Vert M_{1}\right\Vert _{\min }\left\Vert M_{2}\right\Vert
_{\min }$.

There is a refinement of this for block matrices. Let $T$ be a block matrix
with $D$ non-zero blocks and recall that by \ $T(i,j)$ we mean the block $%
(i,j)$ submatrix of $T$. The arithmetic/geometric mean inequality implies
that 
\begin{align*}
\left\Vert T\right\Vert & \geq \sum_{\substack{ (i,j)  \\ \text{non-zero
blocks}}}\left\Vert T(i,j)\right\Vert _{\min } \\
& \geq D\left( \prod_{\substack{ (i,j)  \\ \text{non-zero blocks}}}%
\left\Vert T(i,j)\right\Vert _{\min }\right) ^{1/D} \\
& \geq \left( \prod_{\substack{ (i,j)  \\ \text{non-zero blocks}}}\left\Vert
T(i,j)\right\Vert _{\min }\right) ^{1/D}. \\
&
\end{align*}%
More generally, for block matrices $T_{k},$ 
\begin{eqnarray*}
\left\Vert T_{1}T_{2}\cdot \cdot \cdot T_{t}\right\Vert &\geq &\left( \prod 
_{\substack{ (i,j)  \\ \text{non-zero blocks of }T_{1}T_{2}\dots T_{t}}}%
\left\Vert T_{1}T_{2}\cdot \cdot \cdot T_{t}(i,j)\right\Vert _{\min }\right)
^{1/D} \\
&=&\left( \prod_{\substack{ (i,j)  \\ \text{non-zero blocks of }%
T_{1}T_{2}\dots T_{t}}}\prod_{k=1}^{t}\left\Vert
T_{k}(i_{k},j_{k})\right\Vert _{\min }\right) ^{1/D} \\
&=&\left( \prod_{k=1}^{t}\prod_{\substack{ (i,j)  \\ \text{non-zero blocks
of }T_{k}}}\left\Vert T_{k}(i,j)\right\Vert _{\min }\right) ^{1/D}
\end{eqnarray*}%
where $T_{1}T_{2}\cdot \cdot \cdot
T_{t}(i,j)=\prod_{k=1}^{t}T_{k}(i_{k},j_{k})$. 

Combined with (\ref{reg}), this observation directly yields the following
upper bound on local dimensions.

\begin{proposition}
\label{upperbd}Suppose $\mu _{\pi }$ is a complete quotient Cantor-like
measure with contraction factor $1/d$ and regular probabilities. If 
\begin{equation*}
\left( \prod_{\text{ }(i,j)\text{ non-zero blocks}}\left\Vert \widetilde{T}%
(\ell )(i,j)\right\Vert _{\min }\right) ^{1/(d-1)}\geq \theta
\end{equation*}%
for each primitve transition matrix $\widetilde{T}(\ell )$, then 
\begin{equation*}
\sup_{x}\dim _{loc}\mu _{\pi }(x)\leq \frac{\log \theta }{\log 1/d}.
\end{equation*}
\end{proposition}

\begin{proposition}
\label{prop:shrink} Let $d,k$ be non-negative integers with $d\geq 3$ and
let $\nu =\nu (d,k)$ be the $(d+k)$-fold convolution of the uniform Cantor
measure with contraction factor $1/d$ and $\nu_\pi$ the associated measure
on the torus. For any fixed $k,$%
\begin{equation*}
\{\dim _{loc}\nu _{\pi }(x):x\in \text{supp }\nu _{\pi }\}\subsetneq \{\dim
_{loc}\nu (x):x\in \text{supp }\nu ,x\neq 0,d+k\}
\end{equation*}%
provided $d$ is sufficiently large$.$
\end{proposition}

\begin{proof}
It is easy to see that 
\begin{equation*}
\{\dim _{loc}\nu _{\pi }(x):x\in \text{supp }\nu _{\pi }\}\subseteq \{\dim
_{loc}\nu (x):x\in \text{supp }\nu ,x\neq 0,d+k\}
\end{equation*}%
by noting that 
\begin{equation*}
\min_{\ell }\left( \underline{\dim }_{loc}\mu (x+\ell )\right) \leq \dim
_{loc}\mu _{\pi }(x)\leq \min_{\ell }\left( \overline{\dim }_{loc}\mu
(x+\ell )\right) .
\end{equation*}%
Hence we need only show the strict inclusion. In fact, we will show that $%
\sup_{x}\dim _{loc}\nu _{\pi }(x)<\sup_{x\neq 0,d+k}\dim _{loc}\nu (x)$.

There is no loss in assuming $k\leq d-2$. In this case, \cite[Thm. 6.1]{BHM}
gives the formula%
\begin{equation*}
\sup_{x\neq 0,d+k}\dim _{loc}\nu (x)=\frac{\log \beta }{\log 1/d}
\end{equation*}%
for 
\begin{equation*}
\beta =\frac{p_{r+d+1}+p_{r}+\sqrt{(p_{r+d+1}-p_{r})^{2}+4p_{r+1}p_{r+d}}}{2}%
,
\end{equation*}%
where $r=\left[ \frac{k}{2}\right] $ and 
\begin{equation*}
p_{j}=\binom{d+k}{j}2^{-(d+k)}.
\end{equation*}%
Using the bound $\binom{d+k}{j}\leq (d+k)^{j}$ for $j\leq (d+k)/2$ one can
easily verify that $\beta \leq C_{0}(d+k)^{k/2+1}2^{-(d+k)}$ for a constant $%
C_{0}$ independent of $d.$ Thus%
\begin{equation}
\sup_{x\neq 0,d+k}\dim _{loc}\nu (x)\geq \frac{\log 2^{d+k}-\log
C_{0}(d+k)^{k/2+1}}{\log d}.  \label{nontorus}
\end{equation}

Next, we will apply Proposition \ref{upperbd} to find an upper bound on $%
\dim _{loc}\nu _{\pi }(x)$. For this, we will need to obtain lower bounds on 
$\left\Vert \widetilde{T}(\ell )(i,j)\right\Vert _{\min }$ for the primitive
transition matrices $\widetilde{T}(\ell )$, $\ell =0,...,d-1$. Each of these
block matrices has $d-1$ non-zero blocks $(i,j)$ where $j-i\equiv \ell $ $%
\mod(d-1)$. We recall that the $(i,j)$ block is of size $\left(\left[ \frac{%
d+k-i}{d-1}\right] +1\right)\times \left(\left[ \frac{d+k-j}{d-1}\right]
+1\right)$ and so has either one or two rows, the latter if and only if $%
i\leq k+1,$ and similarly for the columns. A calculation shows that the $%
(i^{\prime },j^{\prime })$ entry of block $(i,j)$ is equal to $%
p_{i-1-j^{\prime }+i^{\prime }d}$ if $j\geq i$ and $p_{i-j^{\prime
}+i^{\prime }d}$ if $j<i$.

First, suppose block $(i,j)$ has one row (hence $i>k+1$). Then the $%
j^{\prime }$ - column sum is either $p_{i-1-j^{\prime }}$ (when $j\geq i$)
or $p_{i-j^{\prime }}$ (when $j<i$)$.$ If there is also only one column
(equivalently, $j>k+1$), then 
\begin{equation*}
\left\Vert \widetilde{T}(\ell )(i,j)\right\Vert _{\min }\geq \left\{ 
\begin{array}{cc}
p_{i-1} & \text{ if }k+1<i\leq (d+k)/2 \\ 
p_{i} & \text{if }i>(d+k)/2%
\end{array}%
\right. .
\end{equation*}

If, instead, the block has two columns (when $j\leq k+1$ and therefore $j<i$%
), we obtain the same conclusion for the minimum column sum.

Otherwise, block $(i,j)$ has two rows, (i.e., $i\leq k+1\leq (d+k)/2).$
Similar reasoning shows that then $\left\Vert \widetilde{T}(\ell
)(i,j)\right\Vert _{\min }\geq \min (p_{i-2},$ $2^{-(d+k)})$. Thus%
\begin{eqnarray*}
\prod_{(i,j)\text{ non-zero blocks}}\left\Vert \widetilde{T}(\ell
)(i,j)\right\Vert _{\min } &\geq &\prod_{i=2}^{k+1}p_{i-2}\prod_{i=k+2}^{%
\left[ \frac{d+k}{2}\right] }p_{i-1}\prod_{i>\left[ \frac{d+k}{2}\right]
}^{d-1}p_{i} \\
&\geq &\frac{1}{p_{k}}\prod_{i=0}^{\left[ \frac{d+k}{2}\right]
-1}p_{i}\prod_{i>\left[ \frac{d+k}{2}\right] }^{d-1}p_{i}.
\end{eqnarray*}

Let $t=k+2.$ For large enough $d,$ $p_{k}\leq p_{k+1}$ and $\binom{d+k}{t}%
\geq (d+k)^{t-1}$. Of course, $\binom{d+k}{s}\geq \binom{d+k}{t}$ for $t\leq
s\leq \left[ \frac{d+k}{2}\right] $ and for $s>\left[ \frac{d+k}{2}\right] $%
, $p_{s}=p_{d+k-s}$. Applying these facts, it follows that 
\begin{eqnarray*}
\prod_{(i,j)}\left\Vert \widetilde{T}(\ell )(i,j)\right\Vert _{\min } &\geq
&2^{-(d+k)(d-1)}\prod_{i=k+3}^{\left[ \frac{d+k}{2}\right] -1}\binom{d+k}{i}%
\prod_{i>\left[ \frac{d+k}{2}\right] }^{d-3}\binom{d+k}{i} \\
&\geq &2^{-(d+k)(d-1)}(d+k)^{t\left( d-6-k\right) }.
\end{eqnarray*}%
Thus for sufficiently large $d,$%
\begin{equation*}
\left( \prod_{(i,j)}\left\Vert \widetilde{T}(\ell )(i,j)\right\Vert _{\min
}\right) ^{1/(d-1)}\geq 2^{-(d+k)}(d+k)^{3(k+2)/4}.
\end{equation*}

It follows from Proposition \ref{upperbd} and (\ref{nontorus}) that for
large enough $d,$ 
\begin{eqnarray*}
\sup_{x}\dim _{loc}\nu _{\pi }(x) &\leq &\frac{\log
2^{-(d+k)}(d+k)^{3(k+2)/4}}{\log 1/d} \\
&=&\frac{\log 2^{d+k}-\log (d+k)^{3(k+2)/4}}{\log d} \\
&<&\sup_{x\neq 0,d+k}\dim _{loc}\nu (x).
\end{eqnarray*}
\end{proof}

\begin{remark}
\label{remark:conj} We conjecture that when $3\leq d\leq m$, the set of
local dimensions of the quotient of the $m$-fold convolution of a uniform
Cantor measure with contraction factor $1/d$ is a proper subset of the set
of local dimensions at essential points of the pre-quotient measure. We have
checked this numerically for all $3\leq d\leq m\leq 10.$ See Table \ref%
{tab:conj} of \cite{HArXiv}. Numerical evidence suggests this is also true
for $m=d-1$ for $d\geq 4$. It is known by Theorem \ref{strictsep} that the
two sets of local dimensions will be equal for $m<d-1$.
\end{remark}

\newpage
\clearpage

\allowdisplaybreaks
\begin{center}
{\textbf{\large LOCAL DIMENSIONS OF MEASURES OF FINITE TYPE ON THE TORUS -- SUPPLEMENTARY INFORMATION}}

\vspace{1 in}

KATHRYN E. HARE, KEVIN G. HARE, AND KEVIN R. MATTHEWS
\end{center}

\thispagestyle{empty}
\renewcommand{\rightmark}{SUPPLEMENTARY INFO}
\setcounter{equation}{0}
\setcounter{section}{0}
\setcounter{figure}{0}
\setcounter{table}{0}
\setcounter{page}{1}
\makeatletter
\renewcommand{\thefigure}{S\arabic{figure}}
\renewcommand{\thetable}{S\arabic{table}}
\renewcommand{\thesection}{S\arabic{section}}

\section{Details of Example \ref{EssnotUnique}}
 
Consider the measure $\mu$ given by the maps $S_i(x) = x/4  + d_i$ with $d_{0} = 0$, $d_{1} = 3/5$, $d_{2} = 6/5$, $d_{3} = 9/5$, and $d_{4} = 3$.
The convex hull of this IFS is $[0,4]$.
Let $\mu_\pi$ be the quotient measure.
The reduced transition diagram has 10 reduced characteristic vectors.
The reduced characteristic vectors are:
\begin{itemize}
\item Reduced characteristic vector 1: $(1, (0, 1, 2, 3))$ 
\item Reduced characteristic vector 2: $(4/5, (0, 4/5, 8/5, 16/5))$ 
\item Reduced characteristic vector 3: $(8/5, (0, 4/5, 8/5, 12/5))$ 
\item Reduced characteristic vector 4: $(4/5, (0, 8/5, 12/5, 16/5))$ 
\item Reduced characteristic vector 5: $(4/5, (0, 4/5, 12/5, 16/5))$ 
\item Reduced characteristic vector 6: $(4/5, (0, 4/5, 8/5, 12/5))$ 
\item Reduced characteristic vector 7: $(4/5, (0, 4/5, 8/5, 12/5, 16/5))$ 
\item Reduced characteristic vector 8: $(4/5, (4/5, 8/5, 12/5, 16/5))$ 
\item Reduced characteristic vector 9: $(4/5, (4/5, 8/5, 16/5))$ 
\item Reduced characteristic vector 10: $(4/5, (0, 8/5, 12/5))$ 
\end{itemize}
The maps are:
\begin{itemize}
\item RCV $1 \to [2, 3, 4, 5]$
\item RCV $2 \to [2, 6, 7, 7]$
\item RCV $3 \to [7, 7, 7, 7, 8, 4, 5, 2]$
\item RCV $4 \to [6, 7, 8, 4]$
\item RCV $5 \to [5, 9, 10, 5]$
\item RCV $6 \to [7, 7, 7, 7]$
\item RCV $7 \to [7, 7, 7, 7]$
\item RCV $8 \to [8, 4, 5, 2]$
\item RCV $9 \to [9, 10, 5, 9]$
\item RCV $10 \to [10, 5, 9, 10]$
\end{itemize}

See Figure \ref{fig:ArXPic1} for the transition diagram.
\begin{figure}[tbp]
\includegraphics[scale=0.5]{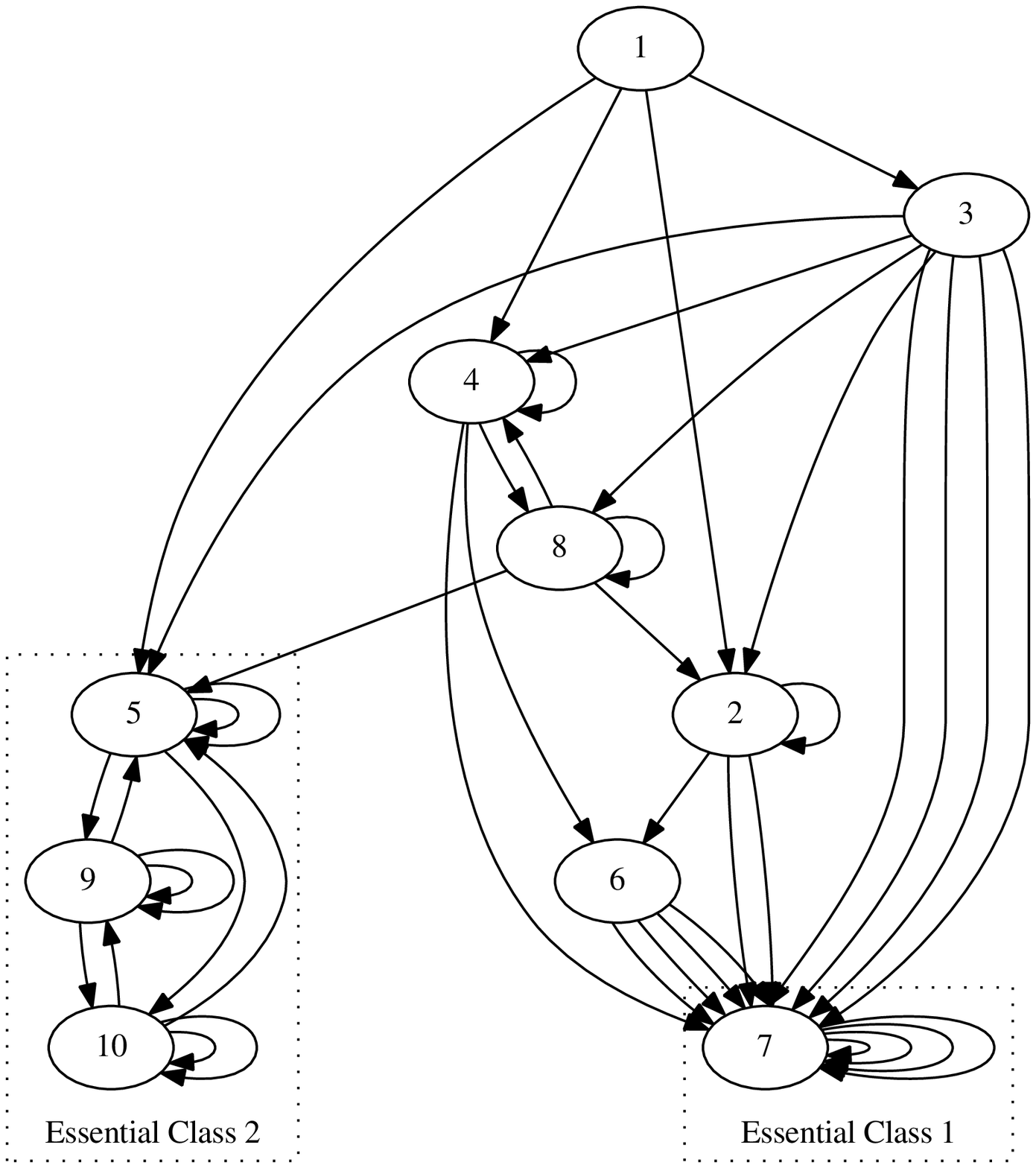}
\caption{$4 x-1$ with $d_i \in [0, 3/5, 6/5, 9/5, 3]$, Full set + Essential class}
\label{fig:ArXPic1}
\end{figure}
This has (normalized) transition matrices:
\begin{align*}
T(1,2) & =  \left[ \begin {array}{cccc} p_{{1}}&0&0&0\\ \noalign{\medskip}0&0&p_{{2}}&0\\ \noalign{\medskip}0&p_{{4}}&0&p_{{3}}\\ \noalign{\medskip}p_{{5}}&0&0&0\end {array} \right] & 
 T(1,3) & =  \left[ \begin {array}{cccc} 0&p_{{1}}&0&0\\ \noalign{\medskip}p_{{3}}&0&0&p_{{2}}\\ \noalign{\medskip}0&0&p_{{4}}&0\\ \noalign{\medskip}0&p_{{5}}&0&0\end {array} \right] \\ 
T(1,4) & =  \left[ \begin {array}{cccc} p_{{2}}&0&p_{{1}}&0\\ \noalign{\medskip}0&p_{{3}}&0&0\\ \noalign{\medskip}0&0&0&p_{{4}}\\ \noalign{\medskip}0&0&p_{{5}}&0\end {array} \right] & 
 T(1,5) & =  \left[ \begin {array}{cccc} 0&p_{{2}}&0&p_{{1}}\\ \noalign{\medskip}p_{{4}}&0&p_{{3}}&0\\ \noalign{\medskip}0&0&0&0\\ \noalign{\medskip}0&0&0&p_{{5}}\end {array} \right] \\ 
T(2,2) & =  \left[ \begin {array}{cccc} p_{{1}}&0&0&0\\ \noalign{\medskip}0&p_{{2}}&0&p_{{1}}\\ \noalign{\medskip}0&0&p_{{3}}&0\\ \noalign{\medskip}0&p_{{5}}&0&0\end {array} \right] & 
 T(2,6) & =  \left[ \begin {array}{cccc} 0&p_{{1}}&0&0\\ \noalign{\medskip}0&0&p_{{2}}&0\\ \noalign{\medskip}p_{{4}}&0&0&p_{{3}}\\ \noalign{\medskip}0&0&p_{{5}}&0\end {array} \right] \\ 
T(2,7) & =  \left[ \begin {array}{ccccc} 0&0&p_{{1}}&0&0\\ \noalign{\medskip}p_{{3}}&0&0&p_{{2}}&0\\ \noalign{\medskip}0&p_{{4}}&0&0&p_{{3}}\\ \noalign{\medskip}0&0&0&p_{{5}}&0\end {array} \right] & 
 T(2,7) & =  \left[ \begin {array}{ccccc} p_{{2}}&0&0&p_{{1}}&0\\ \noalign{\medskip}0&p_{{3}}&0&0&p_{{2}}\\ \noalign{\medskip}0&0&p_{{4}}&0&0\\ \noalign{\medskip}0&0&0&0&p_{{5}}\end {array} \right] \\ 
T(3,7) & =  \left[ \begin {array}{ccccc} p_{{1}}&0&0&0&0\\ \noalign{\medskip}0&p_{{2}}&0&0&p_{{1}}\\ \noalign{\medskip}0&0&p_{{3}}&0&0\\ \noalign{\medskip}0&0&0&p_{{4}}&0\end {array} \right] & 
 T(3,7) & =  \left[ \begin {array}{ccccc} 0&p_{{1}}&0&0&0\\ \noalign{\medskip}0&0&p_{{2}}&0&0\\ \noalign{\medskip}p_{{4}}&0&0&p_{{3}}&0\\ \noalign{\medskip}0&0&0&0&p_{{4}}\end {array} \right] \\ 
T(3,7) & =  \left[ \begin {array}{ccccc} 0&0&p_{{1}}&0&0\\ \noalign{\medskip}p_{{3}}&0&0&p_{{2}}&0\\ \noalign{\medskip}0&p_{{4}}&0&0&p_{{3}}\\ \noalign{\medskip}0&0&0&0&0\end {array} \right] & 
 T(3,7) & =  \left[ \begin {array}{ccccc} p_{{2}}&0&0&p_{{1}}&0\\ \noalign{\medskip}0&p_{{3}}&0&0&p_{{2}}\\ \noalign{\medskip}0&0&p_{{4}}&0&0\\ \noalign{\medskip}p_{{5}}&0&0&0&0\end {array} \right] \\ 
T(3,8) & =  \left[ \begin {array}{cccc} p_{{2}}&0&0&p_{{1}}\\ \noalign{\medskip}0&p_{{3}}&0&0\\ \noalign{\medskip}0&0&p_{{4}}&0\\ \noalign{\medskip}p_{{5}}&0&0&0\end {array} \right] & 
 T(3,4) & =  \left[ \begin {array}{cccc} 0&p_{{2}}&0&0\\ \noalign{\medskip}p_{{4}}&0&p_{{3}}&0\\ \noalign{\medskip}0&0&0&p_{{4}}\\ \noalign{\medskip}0&p_{{5}}&0&0\end {array} \right] \\ 
T(3,5) & =  \left[ \begin {array}{cccc} p_{{3}}&0&p_{{2}}&0\\ \noalign{\medskip}0&p_{{4}}&0&p_{{3}}\\ \noalign{\medskip}0&0&0&0\\ \noalign{\medskip}0&0&p_{{5}}&0\end {array} \right] & 
 T(3,2) & =  \left[ \begin {array}{cccc} 0&p_{{3}}&0&p_{{2}}\\ \noalign{\medskip}0&0&p_{{4}}&0\\ \noalign{\medskip}p_{{5}}&0&0&0\\ \noalign{\medskip}0&0&0&p_{{5}}\end {array} \right] \\ 
T(4,6) & =  \left[ \begin {array}{cccc} p_{{1}}&0&0&0\\ \noalign{\medskip}0&0&p_{{3}}&0\\ \noalign{\medskip}0&0&0&p_{{4}}\\ \noalign{\medskip}0&p_{{5}}&0&0\end {array} \right] & 
 T(4,7) & =  \left[ \begin {array}{ccccc} 0&p_{{1}}&0&0&0\\ \noalign{\medskip}p_{{4}}&0&0&p_{{3}}&0\\ \noalign{\medskip}0&0&0&0&p_{{4}}\\ \noalign{\medskip}0&0&p_{{5}}&0&0\end {array} \right] \\ 
T(4,8) & =  \left[ \begin {array}{cccc} 0&p_{{1}}&0&0\\ \noalign{\medskip}p_{{4}}&0&0&p_{{3}}\\ \noalign{\medskip}0&0&0&0\\ \noalign{\medskip}0&0&p_{{5}}&0\end {array} \right] & 
 T(4,4) & =  \left[ \begin {array}{cccc} p_{{2}}&0&p_{{1}}&0\\ \noalign{\medskip}0&p_{{4}}&0&0\\ \noalign{\medskip}p_{{5}}&0&0&0\\ \noalign{\medskip}0&0&0&p_{{5}}\end {array} \right] \\ 
T(5,5) & =  \left[ \begin {array}{cccc} p_{{1}}&0&0&0\\ \noalign{\medskip}0&p_{{2}}&0&p_{{1}}\\ \noalign{\medskip}0&0&p_{{4}}&0\\ \noalign{\medskip}0&p_{{5}}&0&0\end {array} \right] & 
 T(5,9) & =  \left[ \begin {array}{ccc} p_{{1}}&0&0\\ \noalign{\medskip}0&p_{{2}}&0\\ \noalign{\medskip}0&0&p_{{4}}\\ \noalign{\medskip}0&p_{{5}}&0\end {array} \right] \\ 
T(5,10) & =  \left[ \begin {array}{ccc} 0&p_{{1}}&0\\ \noalign{\medskip}p_{{3}}&0&p_{{2}}\\ \noalign{\medskip}0&0&0\\ \noalign{\medskip}0&0&p_{{5}}\end {array} \right] & 
 T(5,5) & =  \left[ \begin {array}{cccc} p_{{2}}&0&p_{{1}}&0\\ \noalign{\medskip}0&p_{{3}}&0&p_{{2}}\\ \noalign{\medskip}p_{{5}}&0&0&0\\ \noalign{\medskip}0&0&0&p_{{5}}\end {array} \right] \\ 
T(6,7) & =  \left[ \begin {array}{ccccc} p_{{1}}&0&0&0&0\\ \noalign{\medskip}0&p_{{2}}&0&0&p_{{1}}\\ \noalign{\medskip}0&0&p_{{3}}&0&0\\ \noalign{\medskip}0&0&0&p_{{4}}&0\end {array} \right] & 
 T(6,7) & =  \left[ \begin {array}{ccccc} 0&p_{{1}}&0&0&0\\ \noalign{\medskip}0&0&p_{{2}}&0&0\\ \noalign{\medskip}p_{{4}}&0&0&p_{{3}}&0\\ \noalign{\medskip}0&0&0&0&p_{{4}}\end {array} \right] \\ 
T(6,7) & =  \left[ \begin {array}{ccccc} 0&0&p_{{1}}&0&0\\ \noalign{\medskip}p_{{3}}&0&0&p_{{2}}&0\\ \noalign{\medskip}0&p_{{4}}&0&0&p_{{3}}\\ \noalign{\medskip}0&0&0&0&0\end {array} \right] & 
 T(6,7) & =  \left[ \begin {array}{ccccc} p_{{2}}&0&0&p_{{1}}&0\\ \noalign{\medskip}0&p_{{3}}&0&0&p_{{2}}\\ \noalign{\medskip}0&0&p_{{4}}&0&0\\ \noalign{\medskip}p_{{5}}&0&0&0&0\end {array} \right] \\ 
T(7,7) & =  \left[ \begin {array}{ccccc} p_{{1}}&0&0&0&0\\ \noalign{\medskip}0&p_{{2}}&0&0&p_{{1}}\\ \noalign{\medskip}0&0&p_{{3}}&0&0\\ \noalign{\medskip}0&0&0&p_{{4}}&0\\ \noalign{\medskip}0&p_{{5}}&0&0&0\end {array} \right] & 
 T(7,7) & =  \left[ \begin {array}{ccccc} 0&p_{{1}}&0&0&0\\ \noalign{\medskip}0&0&p_{{2}}&0&0\\ \noalign{\medskip}p_{{4}}&0&0&p_{{3}}&0\\ \noalign{\medskip}0&0&0&0&p_{{4}}\\ \noalign{\medskip}0&0&p_{{5}}&0&0\end {array} \right] \\ 
T(7,7) & =  \left[ \begin {array}{ccccc} 0&0&p_{{1}}&0&0\\ \noalign{\medskip}p_{{3}}&0&0&p_{{2}}&0\\ \noalign{\medskip}0&p_{{4}}&0&0&p_{{3}}\\ \noalign{\medskip}0&0&0&0&0\\ \noalign{\medskip}0&0&0&p_{{5}}&0\end {array} \right] & 
 T(7,7) & =  \left[ \begin {array}{ccccc} p_{{2}}&0&0&p_{{1}}&0\\ \noalign{\medskip}0&p_{{3}}&0&0&p_{{2}}\\ \noalign{\medskip}0&0&p_{{4}}&0&0\\ \noalign{\medskip}p_{{5}}&0&0&0&0\\ \noalign{\medskip}0&0&0&0&p_{{5}}\end {array} \right] \\ 
T(8,8) & =  \left[ \begin {array}{cccc} p_{{2}}&0&0&p_{{1}}\\ \noalign{\medskip}0&p_{{3}}&0&0\\ \noalign{\medskip}0&0&p_{{4}}&0\\ \noalign{\medskip}p_{{5}}&0&0&0\end {array} \right] & 
 T(8,4) & =  \left[ \begin {array}{cccc} 0&p_{{2}}&0&0\\ \noalign{\medskip}p_{{4}}&0&p_{{3}}&0\\ \noalign{\medskip}0&0&0&p_{{4}}\\ \noalign{\medskip}0&p_{{5}}&0&0\end {array} \right] \\ 
T(8,5) & =  \left[ \begin {array}{cccc} p_{{3}}&0&p_{{2}}&0\\ \noalign{\medskip}0&p_{{4}}&0&p_{{3}}\\ \noalign{\medskip}0&0&0&0\\ \noalign{\medskip}0&0&p_{{5}}&0\end {array} \right] & 
 T(8,2) & =  \left[ \begin {array}{cccc} 0&p_{{3}}&0&p_{{2}}\\ \noalign{\medskip}0&0&p_{{4}}&0\\ \noalign{\medskip}p_{{5}}&0&0&0\\ \noalign{\medskip}0&0&0&p_{{5}}\end {array} \right] \\ 
T(9,9) & =  \left[ \begin {array}{ccc} p_{{2}}&0&p_{{1}}\\ \noalign{\medskip}0&p_{{3}}&0\\ \noalign{\medskip}p_{{5}}&0&0\end {array} \right] & 
 T(9,10) & =  \left[ \begin {array}{ccc} 0&p_{{2}}&0\\ \noalign{\medskip}p_{{4}}&0&p_{{3}}\\ \noalign{\medskip}0&p_{{5}}&0\end {array} \right] \\ 
T(9,5) & =  \left[ \begin {array}{cccc} p_{{3}}&0&p_{{2}}&0\\ \noalign{\medskip}0&p_{{4}}&0&p_{{3}}\\ \noalign{\medskip}0&0&p_{{5}}&0\end {array} \right] & 
 T(9,9) & =  \left[ \begin {array}{ccc} p_{{3}}&0&p_{{2}}\\ \noalign{\medskip}0&p_{{4}}&0\\ \noalign{\medskip}0&0&p_{{5}}\end {array} \right] \\ 
T(10,10) & =  \left[ \begin {array}{ccc} p_{{1}}&0&0\\ \noalign{\medskip}0&p_{{3}}&0\\ \noalign{\medskip}0&0&p_{{4}}\end {array} \right] & 
 T(10,5) & =  \left[ \begin {array}{cccc} 0&p_{{1}}&0&0\\ \noalign{\medskip}p_{{4}}&0&p_{{3}}&0\\ \noalign{\medskip}0&0&0&p_{{4}}\end {array} \right] \\ 
T(10,9) & =  \left[ \begin {array}{ccc} 0&p_{{1}}&0\\ \noalign{\medskip}p_{{4}}&0&p_{{3}}\\ \noalign{\medskip}0&0&0\end {array} \right] & 
 T(10,10) & =  \left[ \begin {array}{ccc} p_{{2}}&0&p_{{1}}\\ \noalign{\medskip}0&p_{{4}}&0\\ \noalign{\medskip}p_{{5}}&0&0\end {array} \right] \\ 
\end{align*}

There are two essential classes in this case.
The first is given by the reduced characteristic vector $7$ and the second 
    by the reduced characteristic vectors $5, 9$ and $10$.

There are also two non-essential maximal loop classes, given by the 
    reduced characteristic vectors $4$ and $8$ and by the reduced
    characteristic vector 2.

\section{Details for Example \ref{ex:golden} -- Part 1}
 
Consider $\varrho$, the root of $x^2+x-1$ and the maps $S_i(x) = \varrho x  + d_i$ with $d_{0} = 0$, $d_{1} = 1-\varrho$, and $d_{2} = 2-2 \varrho$.
This measure is of full support on $[0,2]$.
The probabilities are given by $p_{0} = 1/4$, $p_{1} = 1/2$, and $p_{2} = 1/4$.
We first give the information for this measure on $\BbR$.
The reduced transition diagram has 40 reduced characteristic vectors.
The reduced characteristic vectors are:
\begin{itemize}
\item Reduced characteristic vector 1: $(1, (0))$ 
\item Reduced characteristic vector 2: $(1/2 \varrho, (0))$ 
\item Reduced characteristic vector 3: $(1/2 \varrho, (0, 1/2 \varrho))$ 
\item Reduced characteristic vector 4: $(1-\varrho, (0, 1/2 \varrho, \varrho))$ 
\item Reduced characteristic vector 5: $(1/2 \varrho, (1-\varrho, 1-1/2 \varrho))$ 
\item Reduced characteristic vector 6: $(1/2 \varrho, (1-1/2 \varrho))$ 
\item Reduced characteristic vector 7: $(1/2-1/2 \varrho, (0, 1/2 \varrho))$ 
\item Reduced characteristic vector 8: $(-1/2+\varrho, (0, 1/2-1/2 \varrho, 1/2))$ 
\item Reduced characteristic vector 9: $(1/2-1/2 \varrho, (0, -1/2+\varrho, 1/2 \varrho, \varrho))$ 
\item Reduced characteristic vector 10: $(1/2-1/2 \varrho, (0, 1/2-1/2 \varrho, 1/2 \varrho, 1/2, 1/2+1/2 \varrho))$ 
\item Reduced characteristic vector 11: $(-1/2+\varrho, (0, 1/2-1/2 \varrho, 1-\varrho, 1/2, 1-1/2 \varrho))$ 
\item Reduced characteristic vector 12: $(1/2-1/2 \varrho, (0, -1/2+\varrho, 1/2 \varrho, 1/2, \varrho, 1/2+1/2 \varrho))$ 
\item Reduced characteristic vector 13: $(1/2-1/2 \varrho, (0, 1/2-1/2 \varrho, 1/2 \varrho, 1/2, 1-1/2 \varrho, 1/2+1/2 \varrho))$ 
\item Reduced characteristic vector 14: $(-1/2+\varrho, (1/2-1/2 \varrho, 1-\varrho, 1/2, 1-1/2 \varrho, 3/2-\varrho))$ 
\item Reduced characteristic vector 15: $(1/2-1/2 \varrho, (0, 1/2 \varrho, 1/2, \varrho, 1/2+1/2 \varrho))$ 
\item Reduced characteristic vector 16: $(1/2-1/2 \varrho, (1/2-1/2 \varrho, 1/2, 1-1/2 \varrho, 1/2+1/2 \varrho))$ 
\item Reduced characteristic vector 17: $(-1/2+\varrho, (1-\varrho, 1-1/2 \varrho, 3/2-\varrho))$ 
\item Reduced characteristic vector 18: $(1/2-1/2 \varrho, (1/2, 1/2+1/2 \varrho))$ 
\item Reduced characteristic vector 19: $(-1/2+\varrho, (0, 1/2-1/2 \varrho, 1/2 \varrho, 1/2, 1-1/2 \varrho, 1/2+1/2 \varrho))$ 
\item Reduced characteristic vector 20: $(1-3/2 \varrho, (0, -1/2+\varrho, 1/2 \varrho, -1/2+3/2 \varrho, \varrho, 1/2+1/2 \varrho, 3/2 \varrho))$ 
\item Reduced characteristic vector 21: $(-1/2+\varrho, (1-3/2 \varrho, 1/2-1/2 \varrho, 1-\varrho, 1/2, 1-1/2 \varrho, 3/2-\varrho))$ 
\item Reduced characteristic vector 22: $(1/2-1/2 \varrho, (0, 1/2-1/2 \varrho, 1/2 \varrho, 1/2, \varrho, 1/2+1/2 \varrho))$ 
\item Reduced characteristic vector 23: $(-1/2+\varrho, (0, 1/2-1/2 \varrho, 1-\varrho, 1/2, 1-1/2 \varrho, 1/2+1/2 \varrho))$ 
\item Reduced characteristic vector 24: $(1-3/2 \varrho, (0, -1/2+\varrho, 1/2 \varrho, 1/2, \varrho, 1/2+1/2 \varrho, 3/2 \varrho))$ 
\item Reduced characteristic vector 25: $(-1/2+\varrho, (1-3/2 \varrho, 1/2-1/2 \varrho, 1-\varrho, 3/2-3/2 \varrho, 1-1/2 \varrho, 3/2-\varrho))$ 
\item Reduced characteristic vector 26: $(-1/2+\varrho, (0, 1-3/2 \varrho, 1/2-1/2 \varrho, 1-\varrho, 1/2, 1-1/2 \varrho, 3/2-\varrho))$ 
\item Reduced characteristic vector 27: $(1/2-1/2 \varrho, (0, -1/2+\varrho, 1/2-1/2 \varrho, 1/2 \varrho, 1/2, \varrho, 1/2+1/2 \varrho))$ 
\item Reduced characteristic vector 28: $(-1/2+\varrho, (0, 1/2-1/2 \varrho, 1/2 \varrho, 1-\varrho, 1/2, 1-1/2 \varrho, 1/2+1/2 \varrho))$ 
\item Reduced characteristic vector 29: $(1-3/2 \varrho, (0, -1/2+\varrho, 1/2 \varrho, -1/2+3/2 \varrho, 1/2, \varrho, 1/2+1/2 \varrho, 3/2 \varrho))$ 
\item Reduced characteristic vector 30: $(-1/2+\varrho, (1-3/2 \varrho, 1/2-1/2 \varrho, 1-\varrho, 1/2, 3/2-3/2 \varrho, 1-1/2 \varrho, 3/2-\varrho))$ 
\item Reduced characteristic vector 31: $(1/2-1/2 \varrho, (0, 1/2-1/2 \varrho, 1/2 \varrho, 1/2, \varrho, 1-1/2 \varrho, 1/2+1/2 \varrho))$ 
\item Reduced characteristic vector 32: $(-1/2+\varrho, (0, 1/2-1/2 \varrho, 1-\varrho, 1/2, 1-1/2 \varrho, 1/2+1/2 \varrho, 3/2-\varrho))$ 
\item Reduced characteristic vector 33: $(-1/2+\varrho, (0, 1/2-1/2 \varrho, 1/2 \varrho, 1/2, \varrho, 1-1/2 \varrho, 1/2+1/2 \varrho))$ 
\item Reduced characteristic vector 34: $(1-3/2 \varrho, (0, -1/2+\varrho, 1/2 \varrho, -1/2+3/2 \varrho, \varrho, -1/2+2 \varrho, 1/2+1/2 \varrho, 3/2 \varrho))$ 
\item Reduced characteristic vector 35: $(-1/2+\varrho, (0, 1-3/2 \varrho, 1/2-1/2 \varrho, 1-\varrho, 1/2, 1-1/2 \varrho, 1/2+1/2 \varrho, 3/2-\varrho))$ 
\item Reduced characteristic vector 36: $(1-3/2 \varrho, (0, -1/2+\varrho, 1/2-1/2 \varrho, 1/2 \varrho, 1/2, \varrho, 1/2+1/2 \varrho, 3/2 \varrho))$ 
\item Reduced characteristic vector 37: $(-1/2+\varrho, (1-3/2 \varrho, 1/2-1/2 \varrho, -2 \varrho+3/2, 1-\varrho, 3/2-3/2 \varrho, 1-1/2 \varrho, 3/2-\varrho))$ 
\item Reduced characteristic vector 38: $(-1/2+\varrho, (0, 1-3/2 \varrho, 1/2-1/2 \varrho, 1-\varrho, 1/2, 3/2-3/2 \varrho, 1-1/2 \varrho, 3/2-\varrho))$ 
\item Reduced characteristic vector 39: $(1/2-1/2 \varrho, (0, -1/2+\varrho, 1/2-1/2 \varrho, 1/2 \varrho, 1/2, \varrho, 1-1/2 \varrho, 1/2+1/2 \varrho))$ 
\item Reduced characteristic vector 40: $(-1/2+\varrho, (0, 1/2-1/2 \varrho, 1/2 \varrho, 1-\varrho, 1/2, 1-1/2 \varrho, 1/2+1/2 \varrho, 3/2-\varrho))$ 
\end{itemize}
The maps are:
\begin{itemize}
\item RCV $1 \to [2, 3, 4, 5, 6]$
\item RCV $2 \to [2, 7]$
\item RCV $3 \to [8, 9, 10]$
\item RCV $4 \to [11, 12, 13, 14]$
\item RCV $5 \to [15, 16, 17]$
\item RCV $6 \to [18, 6]$
\item RCV $7 \to [8, 9]$
\item RCV $8 \to [10]$
\item RCV $9 \to [11, 12]$
\item RCV $10 \to [19, 20, 21]$
\item RCV $11 \to [22]$
\item RCV $12 \to [23, 24, 25]$
\item RCV $13 \to [19, 20, 21]$
\item RCV $14 \to [22]$
\item RCV $15 \to [23, 24, 25]$
\item RCV $16 \to [13, 14]$
\item RCV $17 \to [15]$
\item RCV $18 \to [16, 17]$
\item RCV $19 \to [19, 20]$
\item RCV $20 \to [26]$
\item RCV $21 \to [27]$
\item RCV $22 \to [28, 29, 30]$
\item RCV $23 \to [31]$
\item RCV $24 \to [32]$
\item RCV $25 \to [24, 25]$
\item RCV $26 \to [27]$
\item RCV $27 \to [28, 29, 30]$
\item RCV $28 \to [33, 34]$
\item RCV $29 \to [35]$
\item RCV $30 \to [36, 37]$
\item RCV $31 \to [28, 29, 30]$
\item RCV $32 \to [31]$
\item RCV $33 \to [28, 29]$
\item RCV $34 \to [38]$
\item RCV $35 \to [39]$
\item RCV $36 \to [40]$
\item RCV $37 \to [29, 30]$
\item RCV $38 \to [36, 37]$
\item RCV $39 \to [28, 29, 30]$
\item RCV $40 \to [33, 34]$
\end{itemize}

See Figure \ref{fig:ArXPic2} for the transition diagram.
\begin{figure}[tbp]
\includegraphics[scale=0.4]{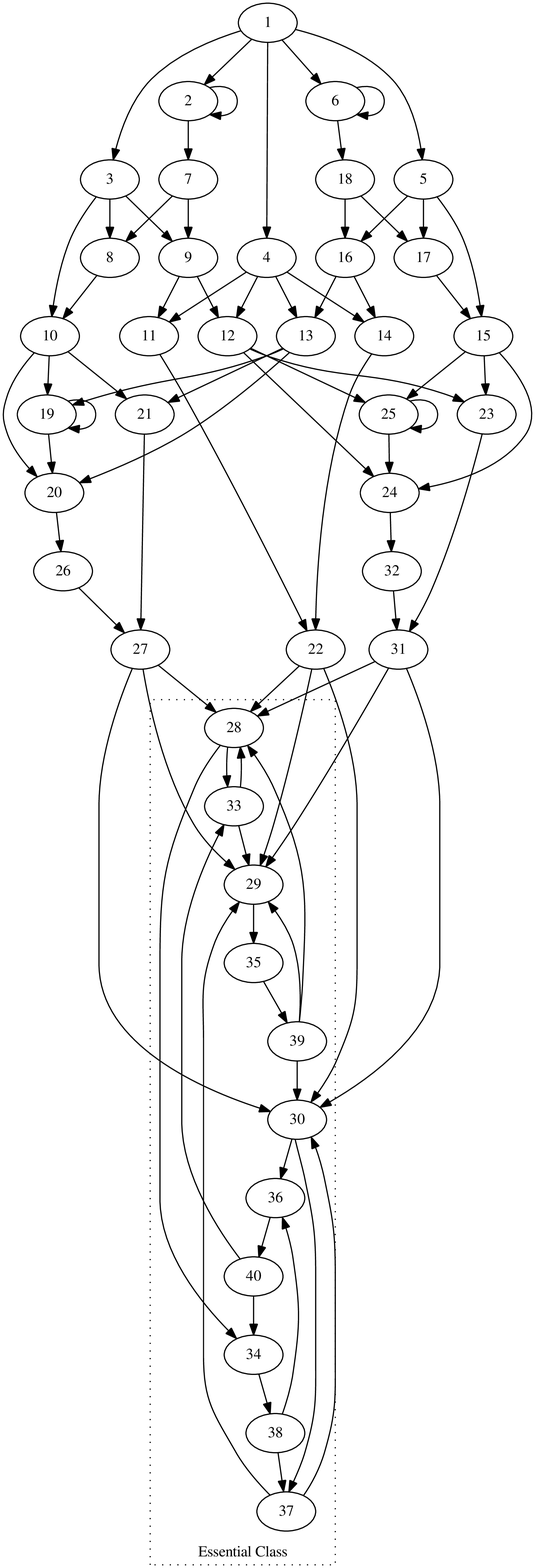}
\caption{$x^2+x-1$ with $d_i \in [0, 1-\varrho, 2-2 \varrho]$, Full set + Essential class}
\label{fig:ArXPic2}
\end{figure}
This has (normalized) transition matrices:
\begin{align*}
T(1,2) & =  \left[ \begin {array}{c} 1\end {array} \right] & 
 T(1,3) & =  \left[ \begin {array}{cc} 2&1\end {array} \right] \\ 
T(1,4) & =  \left[ \begin {array}{ccc} 1&2&1\end {array} \right] & 
 T(1,5) & =  \left[ \begin {array}{cc} 1&2\end {array} \right] \\ 
T(1,6) & =  \left[ \begin {array}{c} 1\end {array} \right] & 
 T(2,2) & =  \left[ \begin {array}{c} 1\end {array} \right] \\ 
T(2,7) & =  \left[ \begin {array}{cc} 2&1\end {array} \right] & 
 T(3,8) & =  \left[ \begin {array}{ccc} 1&0&0\\ \noalign{\medskip}0&2&1\end {array} \right] \\ 
T(3,9) & =  \left[ \begin {array}{cccc} 0&1&0&0\\ \noalign{\medskip}1&0&2&1\end {array} \right] & 
 T(3,10) & =  \left[ \begin {array}{ccccc} 2&0&1&0&0\\ \noalign{\medskip}0&1&0&2&1\end {array} \right] \\ 
T(4,11) & =  \left[ \begin {array}{ccccc} 1&0&0&0&0\\ \noalign{\medskip}0&2&0&1&0\\ \noalign{\medskip}0&0&1&0&2\end {array} \right] & 
 T(4,12) & =  \left[ \begin {array}{cccccc} 0&1&0&0&0&0\\ \noalign{\medskip}1&0&2&0&1&0\\ \noalign{\medskip}0&0&0&1&0&2\end {array} \right] \\ 
T(4,13) & =  \left[ \begin {array}{cccccc} 2&0&1&0&0&0\\ \noalign{\medskip}0&1&0&2&0&1\\ \noalign{\medskip}0&0&0&0&1&0\end {array} \right] & 
 T(4,14) & =  \left[ \begin {array}{ccccc} 2&0&1&0&0\\ \noalign{\medskip}0&1&0&2&0\\ \noalign{\medskip}0&0&0&0&1\end {array} \right] \\ 
T(5,15) & =  \left[ \begin {array}{ccccc} 1&2&0&1&0\\ \noalign{\medskip}0&0&1&0&2\end {array} \right] & 
 T(5,16) & =  \left[ \begin {array}{cccc} 1&2&0&1\\ \noalign{\medskip}0&0&1&0\end {array} \right] \\ 
T(5,17) & =  \left[ \begin {array}{ccc} 1&2&0\\ \noalign{\medskip}0&0&1\end {array} \right] & 
 T(6,18) & =  \left[ \begin {array}{cc} 1&2\end {array} \right] \\ 
T(6,6) & =  \left[ \begin {array}{c} 1\end {array} \right] & 
 T(7,8) & =  \left[ \begin {array}{ccc} 1&0&0\\ \noalign{\medskip}0&2&1\end {array} \right] \\ 
T(7,9) & =  \left[ \begin {array}{cccc} 0&1&0&0\\ \noalign{\medskip}1&0&2&1\end {array} \right] & 
 T(8,10) & =  \left[ \begin {array}{ccccc} 1&0&0&0&0\\ \noalign{\medskip}2&0&1&0&0\\ \noalign{\medskip}0&1&0&2&1\end {array} \right] \\ 
T(9,11) & =  \left[ \begin {array}{ccccc} 1&0&0&0&0\\ \noalign{\medskip}0&1&0&0&0\\ \noalign{\medskip}0&2&0&1&0\\ \noalign{\medskip}0&0&1&0&2\end {array} \right] & 
 T(9,12) & =  \left[ \begin {array}{cccccc} 0&1&0&0&0&0\\ \noalign{\medskip}2&0&1&0&0&0\\ \noalign{\medskip}1&0&2&0&1&0\\ \noalign{\medskip}0&0&0&1&0&2\end {array} \right] \\ 
T(10,19) & =  \left[ \begin {array}{cccccc} 1&0&0&0&0&0\\ \noalign{\medskip}2&0&1&0&0&0\\ \noalign{\medskip}0&2&0&1&0&0\\ \noalign{\medskip}0&1&0&2&0&1\\ \noalign{\medskip}0&0&0&0&1&0\end {array} \right] & 
 T(10,20) & =  \left[ \begin {array}{ccccccc} 0&1&0&0&0&0&0\\ \noalign{\medskip}0&2&0&1&0&0&0\\ \noalign{\medskip}1&0&2&0&1&0&0\\ \noalign{\medskip}0&0&1&0&2&0&1\\ \noalign{\medskip}0&0&0&0&0&1&0\end {array} \right] \\ 
T(10,21) & =  \left[ \begin {array}{cccccc} 0&1&0&0&0&0\\ \noalign{\medskip}0&2&0&1&0&0\\ \noalign{\medskip}1&0&2&0&1&0\\ \noalign{\medskip}0&0&1&0&2&0\\ \noalign{\medskip}0&0&0&0&0&1\end {array} \right] & 
 T(11,22) & =  \left[ \begin {array}{cccccc} 1&0&0&0&0&0\\ \noalign{\medskip}2&0&1&0&0&0\\ \noalign{\medskip}1&0&2&0&1&0\\ \noalign{\medskip}0&1&0&2&0&1\\ \noalign{\medskip}0&0&0&1&0&2\end {array} \right] \\ 
T(12,23) & =  \left[ \begin {array}{cccccc} 1&0&0&0&0&0\\ \noalign{\medskip}0&1&0&0&0&0\\ \noalign{\medskip}0&2&0&1&0&0\\ \noalign{\medskip}0&1&0&2&0&1\\ \noalign{\medskip}0&0&1&0&2&0\\ \noalign{\medskip}0&0&0&0&1&0\end {array} \right] \\ 
T(12,24) & =  \left[ \begin {array}{ccccccc} 0&1&0&0&0&0&0\\ \noalign{\medskip}2&0&1&0&0&0&0\\ \noalign{\medskip}1&0&2&0&1&0&0\\ \noalign{\medskip}0&0&1&0&2&0&1\\ \noalign{\medskip}0&0&0&1&0&2&0\\ \noalign{\medskip}0&0&0&0&0&1&0\end {array} \right] \\ 
T(12,25) & =  \left[ \begin {array}{cccccc} 0&1&0&0&0&0\\ \noalign{\medskip}2&0&1&0&0&0\\ \noalign{\medskip}1&0&2&0&1&0\\ \noalign{\medskip}0&0&1&0&2&0\\ \noalign{\medskip}0&0&0&1&0&2\\ \noalign{\medskip}0&0&0&0&0&1\end {array} \right] \\ 
T(13,19) & =  \left[ \begin {array}{cccccc} 1&0&0&0&0&0\\ \noalign{\medskip}2&0&1&0&0&0\\ \noalign{\medskip}0&2&0&1&0&0\\ \noalign{\medskip}0&1&0&2&0&1\\ \noalign{\medskip}0&0&0&1&0&2\\ \noalign{\medskip}0&0&0&0&1&0\end {array} \right] \\ 
T(13,20) & =  \left[ \begin {array}{ccccccc} 0&1&0&0&0&0&0\\ \noalign{\medskip}0&2&0&1&0&0&0\\ \noalign{\medskip}1&0&2&0&1&0&0\\ \noalign{\medskip}0&0&1&0&2&0&1\\ \noalign{\medskip}0&0&0&0&1&0&2\\ \noalign{\medskip}0&0&0&0&0&1&0\end {array} \right] \\ 
T(13,21) & =  \left[ \begin {array}{cccccc} 0&1&0&0&0&0\\ \noalign{\medskip}0&2&0&1&0&0\\ \noalign{\medskip}1&0&2&0&1&0\\ \noalign{\medskip}0&0&1&0&2&0\\ \noalign{\medskip}0&0&0&0&1&0\\ \noalign{\medskip}0&0&0&0&0&1\end {array} \right] \\ 
T(14,22) & =  \left[ \begin {array}{cccccc} 2&0&1&0&0&0\\ \noalign{\medskip}1&0&2&0&1&0\\ \noalign{\medskip}0&1&0&2&0&1\\ \noalign{\medskip}0&0&0&1&0&2\\ \noalign{\medskip}0&0&0&0&0&1\end {array} \right] & 
 T(15,23) & =  \left[ \begin {array}{cccccc} 1&0&0&0&0&0\\ \noalign{\medskip}0&2&0&1&0&0\\ \noalign{\medskip}0&1&0&2&0&1\\ \noalign{\medskip}0&0&1&0&2&0\\ \noalign{\medskip}0&0&0&0&1&0\end {array} \right] \\ 
T(15,24) & =  \left[ \begin {array}{ccccccc} 0&1&0&0&0&0&0\\ \noalign{\medskip}1&0&2&0&1&0&0\\ \noalign{\medskip}0&0&1&0&2&0&1\\ \noalign{\medskip}0&0&0&1&0&2&0\\ \noalign{\medskip}0&0&0&0&0&1&0\end {array} \right] & 
 T(15,25) & =  \left[ \begin {array}{cccccc} 0&1&0&0&0&0\\ \noalign{\medskip}1&0&2&0&1&0\\ \noalign{\medskip}0&0&1&0&2&0\\ \noalign{\medskip}0&0&0&1&0&2\\ \noalign{\medskip}0&0&0&0&0&1\end {array} \right] \\ 
T(16,13) & =  \left[ \begin {array}{cccccc} 2&0&1&0&0&0\\ \noalign{\medskip}0&1&0&2&0&1\\ \noalign{\medskip}0&0&0&1&0&2\\ \noalign{\medskip}0&0&0&0&1&0\end {array} \right] & 
 T(16,14) & =  \left[ \begin {array}{ccccc} 2&0&1&0&0\\ \noalign{\medskip}0&1&0&2&0\\ \noalign{\medskip}0&0&0&1&0\\ \noalign{\medskip}0&0&0&0&1\end {array} \right] \\ 
T(17,15) & =  \left[ \begin {array}{ccccc} 1&2&0&1&0\\ \noalign{\medskip}0&0&1&0&2\\ \noalign{\medskip}0&0&0&0&1\end {array} \right] & 
 T(18,16) & =  \left[ \begin {array}{cccc} 1&2&0&1\\ \noalign{\medskip}0&0&1&0\end {array} \right] \\ 
T(18,17) & =  \left[ \begin {array}{ccc} 1&2&0\\ \noalign{\medskip}0&0&1\end {array} \right] & 
 T(19,19) & =  \left[ \begin {array}{cccccc} 1&0&0&0&0&0\\ \noalign{\medskip}2&0&1&0&0&0\\ \noalign{\medskip}0&2&0&1&0&0\\ \noalign{\medskip}0&1&0&2&0&1\\ \noalign{\medskip}0&0&0&1&0&2\\ \noalign{\medskip}0&0&0&0&1&0\end {array} \right] \\ 
T(19,20) & =  \left[ \begin {array}{ccccccc} 0&1&0&0&0&0&0\\ \noalign{\medskip}0&2&0&1&0&0&0\\ \noalign{\medskip}1&0&2&0&1&0&0\\ \noalign{\medskip}0&0&1&0&2&0&1\\ \noalign{\medskip}0&0&0&0&1&0&2\\ \noalign{\medskip}0&0&0&0&0&1&0\end {array} \right] \\ 
T(20,26) & =  \left[ \begin {array}{ccccccc} 1&0&0&0&0&0&0\\ \noalign{\medskip}0&0&1&0&0&0&0\\ \noalign{\medskip}0&0&2&0&1&0&0\\ \noalign{\medskip}0&1&0&2&0&1&0\\ \noalign{\medskip}0&0&0&1&0&2&0\\ \noalign{\medskip}0&0&0&0&0&1&0\\ \noalign{\medskip}0&0&0&0&0&0&1\end {array} \right] \\ 
T(21,27) & =  \left[ \begin {array}{ccccccc} 0&1&0&0&0&0&0\\ \noalign{\medskip}2&0&0&1&0&0&0\\ \noalign{\medskip}1&0&0&2&0&1&0\\ \noalign{\medskip}0&0&1&0&2&0&1\\ \noalign{\medskip}0&0&0&0&1&0&2\\ \noalign{\medskip}0&0&0&0&0&0&1\end {array} \right] \\ 
T(22,28) & =  \left[ \begin {array}{ccccccc} 1&0&0&0&0&0&0\\ \noalign{\medskip}2&0&1&0&0&0&0\\ \noalign{\medskip}0&2&0&0&1&0&0\\ \noalign{\medskip}0&1&0&0&2&0&1\\ \noalign{\medskip}0&0&0&1&0&2&0\\ \noalign{\medskip}0&0&0&0&0&1&0\end {array} \right] \\ 
T(22,29) & =  \left[ \begin {array}{cccccccc} 0&1&0&0&0&0&0&0\\ \noalign{\medskip}0&2&0&1&0&0&0&0\\ \noalign{\medskip}1&0&2&0&0&1&0&0\\ \noalign{\medskip}0&0&1&0&0&2&0&1\\ \noalign{\medskip}0&0&0&0&1&0&2&0\\ \noalign{\medskip}0&0&0&0&0&0&1&0\end {array} \right] \\ 
T(22,30) & =  \left[ \begin {array}{ccccccc} 0&1&0&0&0&0&0\\ \noalign{\medskip}0&2&0&1&0&0&0\\ \noalign{\medskip}1&0&2&0&0&1&0\\ \noalign{\medskip}0&0&1&0&0&2&0\\ \noalign{\medskip}0&0&0&0&1&0&2\\ \noalign{\medskip}0&0&0&0&0&0&1\end {array} \right] \\ 
T(23,31) & =  \left[ \begin {array}{ccccccc} 1&0&0&0&0&0&0\\ \noalign{\medskip}2&0&1&0&0&0&0\\ \noalign{\medskip}1&0&2&0&1&0&0\\ \noalign{\medskip}0&1&0&2&0&0&1\\ \noalign{\medskip}0&0&0&1&0&0&2\\ \noalign{\medskip}0&0&0&0&0&1&0\end {array} \right] \\ 
T(24,32) & =  \left[ \begin {array}{ccccccc} 1&0&0&0&0&0&0\\ \noalign{\medskip}0&1&0&0&0&0&0\\ \noalign{\medskip}0&2&0&1&0&0&0\\ \noalign{\medskip}0&1&0&2&0&1&0\\ \noalign{\medskip}0&0&1&0&2&0&0\\ \noalign{\medskip}0&0&0&0&1&0&0\\ \noalign{\medskip}0&0&0&0&0&0&1\end {array} \right] \\ 
T(25,24) & =  \left[ \begin {array}{ccccccc} 0&1&0&0&0&0&0\\ \noalign{\medskip}2&0&1&0&0&0&0\\ \noalign{\medskip}1&0&2&0&1&0&0\\ \noalign{\medskip}0&0&1&0&2&0&1\\ \noalign{\medskip}0&0&0&1&0&2&0\\ \noalign{\medskip}0&0&0&0&0&1&0\end {array} \right] \\ 
T(25,25) & =  \left[ \begin {array}{cccccc} 0&1&0&0&0&0\\ \noalign{\medskip}2&0&1&0&0&0\\ \noalign{\medskip}1&0&2&0&1&0\\ \noalign{\medskip}0&0&1&0&2&0\\ \noalign{\medskip}0&0&0&1&0&2\\ \noalign{\medskip}0&0&0&0&0&1\end {array} \right] \\ 
T(26,27) & =  \left[ \begin {array}{ccccccc} 1&0&0&0&0&0&0\\ \noalign{\medskip}0&1&0&0&0&0&0\\ \noalign{\medskip}2&0&0&1&0&0&0\\ \noalign{\medskip}1&0&0&2&0&1&0\\ \noalign{\medskip}0&0&1&0&2&0&1\\ \noalign{\medskip}0&0&0&0&1&0&2\\ \noalign{\medskip}0&0&0&0&0&0&1\end {array} \right] \\ 
T(27,28) & =  \left[ \begin {array}{ccccccc} 1&0&0&0&0&0&0\\ \noalign{\medskip}0&1&0&0&0&0&0\\ \noalign{\medskip}2&0&1&0&0&0&0\\ \noalign{\medskip}0&2&0&0&1&0&0\\ \noalign{\medskip}0&1&0&0&2&0&1\\ \noalign{\medskip}0&0&0&1&0&2&0\\ \noalign{\medskip}0&0&0&0&0&1&0\end {array} \right] \\ 
T(27,29) & =  \left[ \begin {array}{cccccccc} 0&1&0&0&0&0&0&0\\ \noalign{\medskip}2&0&1&0&0&0&0&0\\ \noalign{\medskip}0&2&0&1&0&0&0&0\\ \noalign{\medskip}1&0&2&0&0&1&0&0\\ \noalign{\medskip}0&0&1&0&0&2&0&1\\ \noalign{\medskip}0&0&0&0&1&0&2&0\\ \noalign{\medskip}0&0&0&0&0&0&1&0\end {array} \right] \\ 
T(27,30) & =  \left[ \begin {array}{ccccccc} 0&1&0&0&0&0&0\\ \noalign{\medskip}2&0&1&0&0&0&0\\ \noalign{\medskip}0&2&0&1&0&0&0\\ \noalign{\medskip}1&0&2&0&0&1&0\\ \noalign{\medskip}0&0&1&0&0&2&0\\ \noalign{\medskip}0&0&0&0&1&0&2\\ \noalign{\medskip}0&0&0&0&0&0&1\end {array} \right] \\ 
T(28,33) & =  \left[ \begin {array}{ccccccc} 1&0&0&0&0&0&0\\ \noalign{\medskip}2&0&1&0&0&0&0\\ \noalign{\medskip}0&2&0&1&0&0&0\\ \noalign{\medskip}1&0&2&0&1&0&0\\ \noalign{\medskip}0&1&0&2&0&0&1\\ \noalign{\medskip}0&0&0&1&0&0&2\\ \noalign{\medskip}0&0&0&0&0&1&0\end {array} \right] \\ 
T(28,34) & =  \left[ \begin {array}{cccccccc} 0&1&0&0&0&0&0&0\\ \noalign{\medskip}0&2&0&1&0&0&0&0\\ \noalign{\medskip}1&0&2&0&1&0&0&0\\ \noalign{\medskip}0&1&0&2&0&1&0&0\\ \noalign{\medskip}0&0&1&0&2&0&0&1\\ \noalign{\medskip}0&0&0&0&1&0&0&2\\ \noalign{\medskip}0&0&0&0&0&0&1&0\end {array} \right] \\ 
T(29,35) & =  \left[ \begin {array}{cccccccc} 1&0&0&0&0&0&0&0\\ \noalign{\medskip}0&0&1&0&0&0&0&0\\ \noalign{\medskip}0&0&2&0&1&0&0&0\\ \noalign{\medskip}0&1&0&2&0&1&0&0\\ \noalign{\medskip}0&0&1&0&2&0&1&0\\ \noalign{\medskip}0&0&0&1&0&2&0&0\\ \noalign{\medskip}0&0&0&0&0&1&0&0\\ \noalign{\medskip}0&0&0&0&0&0&0&1\end {array} \right] \\ 
T(30,36) & =  \left[ \begin {array}{cccccccc} 0&1&0&0&0&0&0&0\\ \noalign{\medskip}2&0&0&1&0&0&0&0\\ \noalign{\medskip}1&0&0&2&0&1&0&0\\ \noalign{\medskip}0&0&1&0&2&0&1&0\\ \noalign{\medskip}0&0&0&1&0&2&0&1\\ \noalign{\medskip}0&0&0&0&1&0&2&0\\ \noalign{\medskip}0&0&0&0&0&0&1&0\end {array} \right] \\ 
T(30,37) & =  \left[ \begin {array}{ccccccc} 0&1&0&0&0&0&0\\ \noalign{\medskip}2&0&0&1&0&0&0\\ \noalign{\medskip}1&0&0&2&0&1&0\\ \noalign{\medskip}0&0&1&0&2&0&1\\ \noalign{\medskip}0&0&0&1&0&2&0\\ \noalign{\medskip}0&0&0&0&1&0&2\\ \noalign{\medskip}0&0&0&0&0&0&1\end {array} \right] \\ 
T(31,28) & =  \left[ \begin {array}{ccccccc} 1&0&0&0&0&0&0\\ \noalign{\medskip}2&0&1&0&0&0&0\\ \noalign{\medskip}0&2&0&0&1&0&0\\ \noalign{\medskip}0&1&0&0&2&0&1\\ \noalign{\medskip}0&0&0&1&0&2&0\\ \noalign{\medskip}0&0&0&0&1&0&2\\ \noalign{\medskip}0&0&0&0&0&1&0\end {array} \right] \\ 
T(31,29) & =  \left[ \begin {array}{cccccccc} 0&1&0&0&0&0&0&0\\ \noalign{\medskip}0&2&0&1&0&0&0&0\\ \noalign{\medskip}1&0&2&0&0&1&0&0\\ \noalign{\medskip}0&0&1&0&0&2&0&1\\ \noalign{\medskip}0&0&0&0&1&0&2&0\\ \noalign{\medskip}0&0&0&0&0&1&0&2\\ \noalign{\medskip}0&0&0&0&0&0&1&0\end {array} \right] \\ 
T(31,30) & =  \left[ \begin {array}{ccccccc} 0&1&0&0&0&0&0\\ \noalign{\medskip}0&2&0&1&0&0&0\\ \noalign{\medskip}1&0&2&0&0&1&0\\ \noalign{\medskip}0&0&1&0&0&2&0\\ \noalign{\medskip}0&0&0&0&1&0&2\\ \noalign{\medskip}0&0&0&0&0&1&0\\ \noalign{\medskip}0&0&0&0&0&0&1\end {array} \right] \\ 
T(32,31) & =  \left[ \begin {array}{ccccccc} 1&0&0&0&0&0&0\\ \noalign{\medskip}2&0&1&0&0&0&0\\ \noalign{\medskip}1&0&2&0&1&0&0\\ \noalign{\medskip}0&1&0&2&0&0&1\\ \noalign{\medskip}0&0&0&1&0&0&2\\ \noalign{\medskip}0&0&0&0&0&1&0\\ \noalign{\medskip}0&0&0&0&0&0&1\end {array} \right] \\ 
T(33,28) & =  \left[ \begin {array}{ccccccc} 1&0&0&0&0&0&0\\ \noalign{\medskip}2&0&1&0&0&0&0\\ \noalign{\medskip}0&2&0&0&1&0&0\\ \noalign{\medskip}0&1&0&0&2&0&1\\ \noalign{\medskip}0&0&0&1&0&2&0\\ \noalign{\medskip}0&0&0&0&1&0&2\\ \noalign{\medskip}0&0&0&0&0&1&0\end {array} \right] \\ 
T(33,29) & =  \left[ \begin {array}{cccccccc} 0&1&0&0&0&0&0&0\\ \noalign{\medskip}0&2&0&1&0&0&0&0\\ \noalign{\medskip}1&0&2&0&0&1&0&0\\ \noalign{\medskip}0&0&1&0&0&2&0&1\\ \noalign{\medskip}0&0&0&0&1&0&2&0\\ \noalign{\medskip}0&0&0&0&0&1&0&2\\ \noalign{\medskip}0&0&0&0&0&0&1&0\end {array} \right] \\ 
T(34,38) & =  \left[ \begin {array}{cccccccc} 1&0&0&0&0&0&0&0\\ \noalign{\medskip}0&0&1&0&0&0&0&0\\ \noalign{\medskip}0&0&2&0&1&0&0&0\\ \noalign{\medskip}0&1&0&2&0&0&1&0\\ \noalign{\medskip}0&0&0&1&0&0&2&0\\ \noalign{\medskip}0&0&0&0&0&1&0&2\\ \noalign{\medskip}0&0&0&0&0&0&1&0\\ \noalign{\medskip}0&0&0&0&0&0&0&1\end {array} \right] \\ 
T(35,39) & =  \left[ \begin {array}{cccccccc} 1&0&0&0&0&0&0&0\\ \noalign{\medskip}0&1&0&0&0&0&0&0\\ \noalign{\medskip}2&0&0&1&0&0&0&0\\ \noalign{\medskip}1&0&0&2&0&1&0&0\\ \noalign{\medskip}0&0&1&0&2&0&0&1\\ \noalign{\medskip}0&0&0&0&1&0&0&2\\ \noalign{\medskip}0&0&0&0&0&0&1&0\\ \noalign{\medskip}0&0&0&0&0&0&0&1\end {array} \right] \\ 
T(36,40) & =  \left[ \begin {array}{cccccccc} 1&0&0&0&0&0&0&0\\ \noalign{\medskip}0&1&0&0&0&0&0&0\\ \noalign{\medskip}2&0&1&0&0&0&0&0\\ \noalign{\medskip}0&2&0&0&1&0&0&0\\ \noalign{\medskip}0&1&0&0&2&0&1&0\\ \noalign{\medskip}0&0&0&1&0&2&0&0\\ \noalign{\medskip}0&0&0&0&0&1&0&0\\ \noalign{\medskip}0&0&0&0&0&0&0&1\end {array} \right] \\ 
T(37,29) & =  \left[ \begin {array}{cccccccc} 0&1&0&0&0&0&0&0\\ \noalign{\medskip}2&0&1&0&0&0&0&0\\ \noalign{\medskip}0&2&0&1&0&0&0&0\\ \noalign{\medskip}1&0&2&0&0&1&0&0\\ \noalign{\medskip}0&0&1&0&0&2&0&1\\ \noalign{\medskip}0&0&0&0&1&0&2&0\\ \noalign{\medskip}0&0&0&0&0&0&1&0\end {array} \right] \\ 
T(37,30) & =  \left[ \begin {array}{ccccccc} 0&1&0&0&0&0&0\\ \noalign{\medskip}2&0&1&0&0&0&0\\ \noalign{\medskip}0&2&0&1&0&0&0\\ \noalign{\medskip}1&0&2&0&0&1&0\\ \noalign{\medskip}0&0&1&0&0&2&0\\ \noalign{\medskip}0&0&0&0&1&0&2\\ \noalign{\medskip}0&0&0&0&0&0&1\end {array} \right] \\ 
T(38,36) & =  \left[ \begin {array}{cccccccc} 1&0&0&0&0&0&0&0\\ \noalign{\medskip}0&1&0&0&0&0&0&0\\ \noalign{\medskip}2&0&0&1&0&0&0&0\\ \noalign{\medskip}1&0&0&2&0&1&0&0\\ \noalign{\medskip}0&0&1&0&2&0&1&0\\ \noalign{\medskip}0&0&0&1&0&2&0&1\\ \noalign{\medskip}0&0&0&0&1&0&2&0\\ \noalign{\medskip}0&0&0&0&0&0&1&0\end {array} \right] \\ 
T(38,37) & =  \left[ \begin {array}{ccccccc} 1&0&0&0&0&0&0\\ \noalign{\medskip}0&1&0&0&0&0&0\\ \noalign{\medskip}2&0&0&1&0&0&0\\ \noalign{\medskip}1&0&0&2&0&1&0\\ \noalign{\medskip}0&0&1&0&2&0&1\\ \noalign{\medskip}0&0&0&1&0&2&0\\ \noalign{\medskip}0&0&0&0&1&0&2\\ \noalign{\medskip}0&0&0&0&0&0&1\end {array} \right] \\ 
T(39,28) & =  \left[ \begin {array}{ccccccc} 1&0&0&0&0&0&0\\ \noalign{\medskip}0&1&0&0&0&0&0\\ \noalign{\medskip}2&0&1&0&0&0&0\\ \noalign{\medskip}0&2&0&0&1&0&0\\ \noalign{\medskip}0&1&0&0&2&0&1\\ \noalign{\medskip}0&0&0&1&0&2&0\\ \noalign{\medskip}0&0&0&0&1&0&2\\ \noalign{\medskip}0&0&0&0&0&1&0\end {array} \right] \\ 
T(39,29) & =  \left[ \begin {array}{cccccccc} 0&1&0&0&0&0&0&0\\ \noalign{\medskip}2&0&1&0&0&0&0&0\\ \noalign{\medskip}0&2&0&1&0&0&0&0\\ \noalign{\medskip}1&0&2&0&0&1&0&0\\ \noalign{\medskip}0&0&1&0&0&2&0&1\\ \noalign{\medskip}0&0&0&0&1&0&2&0\\ \noalign{\medskip}0&0&0&0&0&1&0&2\\ \noalign{\medskip}0&0&0&0&0&0&1&0\end {array} \right] \\ 
T(39,30) & =  \left[ \begin {array}{ccccccc} 0&1&0&0&0&0&0\\ \noalign{\medskip}2&0&1&0&0&0&0\\ \noalign{\medskip}0&2&0&1&0&0&0\\ \noalign{\medskip}1&0&2&0&0&1&0\\ \noalign{\medskip}0&0&1&0&0&2&0\\ \noalign{\medskip}0&0&0&0&1&0&2\\ \noalign{\medskip}0&0&0&0&0&1&0\\ \noalign{\medskip}0&0&0&0&0&0&1\end {array} \right] \\ 
T(40,33) & =  \left[ \begin {array}{ccccccc} 1&0&0&0&0&0&0\\ \noalign{\medskip}2&0&1&0&0&0&0\\ \noalign{\medskip}0&2&0&1&0&0&0\\ \noalign{\medskip}1&0&2&0&1&0&0\\ \noalign{\medskip}0&1&0&2&0&0&1\\ \noalign{\medskip}0&0&0&1&0&0&2\\ \noalign{\medskip}0&0&0&0&0&1&0\\ \noalign{\medskip}0&0&0&0&0&0&1\end {array} \right] \\ 
T(40,34) & =  \left[ \begin {array}{cccccccc} 0&1&0&0&0&0&0&0\\ \noalign{\medskip}0&2&0&1&0&0&0&0\\ \noalign{\medskip}1&0&2&0&1&0&0&0\\ \noalign{\medskip}0&1&0&2&0&1&0&0\\ \noalign{\medskip}0&0&1&0&2&0&0&1\\ \noalign{\medskip}0&0&0&0&1&0&0&2\\ \noalign{\medskip}0&0&0&0&0&0&1&0\\ \noalign{\medskip}0&0&0&0&0&0&0&1\end {array} \right] \\ 
\end{align*}

The essential class is: [28, 29, 30, 33, 34, 35, 36, 37, 38, 39, 40].
The essential class is of positive type.
An example is the path [28, 34, 38, 36, 40, 34, 38, 37].
The essential class is not a simple loop.
This spectral range will include the interval $[2.469158042, 2.481194304]$.
The minimum comes from the loop $[29, 35, 39, 29]$.
The maximim comes from the loop $[28, 33, 28]$.
These points will include points of local dimension [.992399434, 1.002504754].
The Spectral Range is contained in the range $[2.038390910, 2.701372314]$.
The minimum comes from the column sub-norm on the subset ${{3, 4}}$ of length 20. 
The maximum comes from the total column sup-norm of length 10. 
These points will have local dimension contained in [.815720713, 1.400908289].

There are 4 additional maximal loops.

Maximal Loop Class: [25].
The maximal loop class is a simple loop.
It's spectral radius is an isolated points of 2.481194304.
These points have local dimension .992399434.

Maximal Loop Class: [19].
The maximal loop class is a simple loop.
It's spectral radius is an isolated points of 2.481194304.
These points have local dimension .992399434.

Maximal Loop Class: [6].
The maximal loop class is a simple loop.
It's spectral radius is an isolated points of 1.
These points have local dimension 2.880840181.

Maximal Loop Class: [2].
The maximal loop class is a simple loop.
It's spectral radius is an isolated points of 1.
These points have local dimension 2.880840181.

The set of local dimensions will contain
\[[.9924,1.003]\cup \{ 2.881\}\]This has 2 components. 

The set of local dimensions is contained in
 \[[.8157,1.401]\cup \{ 2.881\}\]This has 2 components.

\section{Details for Example \ref{ex:golden} -- Part 2}
 
Consider $\varrho$, the root of $x^2+x-1$ and the maps $S_i(x) = \varrho x  + d_i$ with $d_{0} = 0$, $d_{1} = 1-\varrho$, and $d_{2} = 2-2 \varrho$.
This measure is of full support on $[0,2]$.
The probabilities are given by $p_{0} = 1/4$, $p_{1} = 1/2$, and $p_{2} = 1/4$.
We now give detailed information for the quotient measure.
The reduced transition diagram has 38 reduced characteristic vectors.
The reduced characteristic vectors are:
\begin{itemize}
\item Reduced characteristic vector 1: $(1, (0, 1))$ 
\item Reduced characteristic vector 2: $(1-\varrho, (0, 1-\varrho, 1, \varrho+1))$ 
\item Reduced characteristic vector 3: $(-1+2 \varrho, (1-\varrho, 2-2 \varrho, 2-\varrho))$ 
\item Reduced characteristic vector 4: $(1-\varrho, (0, \varrho, 1, \varrho+1))$ 
\item Reduced characteristic vector 5: $(-1+2 \varrho, (1-\varrho, 1, 2-\varrho))$ 
\item Reduced characteristic vector 6: $(1-\varrho, (0, \varrho, 2 \varrho, \varrho+1))$ 
\item Reduced characteristic vector 7: $(1-\varrho, (0, 1-\varrho, \varrho, 1, 2-\varrho, \varrho+1))$ 
\item Reduced characteristic vector 8: $(-1+2 \varrho, (1-\varrho, 2-2 \varrho, 1, 2-\varrho, 3-2 \varrho))$ 
\item Reduced characteristic vector 9: $(1-\varrho, (0, \varrho, 1, 2 \varrho, \varrho+1))$ 
\item Reduced characteristic vector 10: $(-1+2 \varrho, (0, 1-\varrho, 1, 2-\varrho, \varrho+1))$ 
\item Reduced characteristic vector 11: $(-3 \varrho+2, (0, -1+2 \varrho, \varrho, 2 \varrho, \varrho+1, 3 \varrho))$ 
\item Reduced characteristic vector 12: $(-1+2 \varrho, (-3 \varrho+2, 1-\varrho, 2-2 \varrho, 2-\varrho, 3-2 \varrho))$ 
\item Reduced characteristic vector 13: $(1-\varrho, (0, 1-\varrho, \varrho, 1, \varrho+1))$ 
\item Reduced characteristic vector 14: $(-1+2 \varrho, (0, 1-\varrho, 2-2 \varrho, 1, 2-\varrho))$ 
\item Reduced characteristic vector 15: $(1-\varrho, (0, -1+2 \varrho, \varrho, 1, 2 \varrho, \varrho+1))$ 
\item Reduced characteristic vector 16: $(-1+2 \varrho, (0, 1-\varrho, \varrho, 1, 2-\varrho, \varrho+1))$ 
\item Reduced characteristic vector 17: $(-3 \varrho+2, (0, -1+2 \varrho, \varrho, -1+3 \varrho, 2 \varrho, \varrho+1, 3 \varrho))$ 
\item Reduced characteristic vector 18: $(-1+2 \varrho, (-3 \varrho+2, 1-\varrho, 2-2 \varrho, 1, 2-\varrho, 3-2 \varrho))$ 
\item Reduced characteristic vector 19: $(1-\varrho, (0, 1-\varrho, \varrho, 1, 2 \varrho, \varrho+1))$ 
\item Reduced characteristic vector 20: $(-1+2 \varrho, (0, 1-\varrho, 2-2 \varrho, 1, 2-\varrho, \varrho+1))$ 
\item Reduced characteristic vector 21: $(-3 \varrho+2, (0, -1+2 \varrho, \varrho, 1, 2 \varrho, \varrho+1, 3 \varrho))$ 
\item Reduced characteristic vector 22: $(-1+2 \varrho, (-3 \varrho+2, 1-\varrho, 2-2 \varrho, 3-3 \varrho, 2-\varrho, 3-2 \varrho))$ 
\item Reduced characteristic vector 23: $(-1+2 \varrho, (0, 1-\varrho, 2-2 \varrho, 1, 2-\varrho, 3-2 \varrho))$ 
\item Reduced characteristic vector 24: $(-1+2 \varrho, (0, -3 \varrho+2, 1-\varrho, 2-2 \varrho, 1, 2-\varrho, 3-2 \varrho))$ 
\item Reduced characteristic vector 25: $(1-\varrho, (0, -1+2 \varrho, 1-\varrho, \varrho, 1, 2 \varrho, \varrho+1))$ 
\item Reduced characteristic vector 26: $(-1+2 \varrho, (0, 1-\varrho, \varrho, 2-2 \varrho, 1, 2-\varrho, \varrho+1))$ 
\item Reduced characteristic vector 27: $(-3 \varrho+2, (0, -1+2 \varrho, \varrho, -1+3 \varrho, 1, 2 \varrho, \varrho+1, 3 \varrho))$ 
\item Reduced characteristic vector 28: $(-1+2 \varrho, (-3 \varrho+2, 1-\varrho, 2-2 \varrho, 1, 3-3 \varrho, 2-\varrho, 3-2 \varrho))$ 
\item Reduced characteristic vector 29: $(1-\varrho, (0, 1-\varrho, \varrho, 1, 2 \varrho, 2-\varrho, \varrho+1))$ 
\item Reduced characteristic vector 30: $(-1+2 \varrho, (0, 1-\varrho, 2-2 \varrho, 1, 2-\varrho, \varrho+1, 3-2 \varrho))$ 
\item Reduced characteristic vector 31: $(-1+2 \varrho, (0, 1-\varrho, \varrho, 1, 2 \varrho, 2-\varrho, \varrho+1))$ 
\item Reduced characteristic vector 32: $(-3 \varrho+2, (0, -1+2 \varrho, \varrho, -1+3 \varrho, 2 \varrho, -1+4 \varrho, \varrho+1, 3 \varrho))$ 
\item Reduced characteristic vector 33: $(-1+2 \varrho, (0, -3 \varrho+2, 1-\varrho, 2-2 \varrho, 1, 2-\varrho, \varrho+1, 3-2 \varrho))$ 
\item Reduced characteristic vector 34: $(-3 \varrho+2, (0, -1+2 \varrho, 1-\varrho, \varrho, 1, 2 \varrho, \varrho+1, 3 \varrho))$ 
\item Reduced characteristic vector 35: $(-1+2 \varrho, (-3 \varrho+2, 1-\varrho, 3-4 \varrho, 2-2 \varrho, 3-3 \varrho, 2-\varrho, 3-2 \varrho))$ 
\item Reduced characteristic vector 36: $(-1+2 \varrho, (0, -3 \varrho+2, 1-\varrho, 2-2 \varrho, 1, 3-3 \varrho, 2-\varrho, 3-2 \varrho))$ 
\item Reduced characteristic vector 37: $(1-\varrho, (0, -1+2 \varrho, 1-\varrho, \varrho, 1, 2 \varrho, 2-\varrho, \varrho+1))$ 
\item Reduced characteristic vector 38: $(-1+2 \varrho, (0, 1-\varrho, \varrho, 2-2 \varrho, 1, 2-\varrho, \varrho+1, 3-2 \varrho))$ 
\end{itemize}
The maps are:
\begin{itemize}
\item RCV $1 \to [2, 3, 4, 5, 6]$
\item RCV $2 \to [7, 8]$
\item RCV $3 \to [9]$
\item RCV $4 \to [10, 11, 12]$
\item RCV $5 \to [13]$
\item RCV $6 \to [14, 15]$
\item RCV $7 \to [16, 17, 18]$
\item RCV $8 \to [19]$
\item RCV $9 \to [20, 21, 22]$
\item RCV $10 \to [7]$
\item RCV $11 \to [23]$
\item RCV $12 \to [15]$
\item RCV $13 \to [16, 17, 18]$
\item RCV $14 \to [19]$
\item RCV $15 \to [20, 21, 22]$
\item RCV $16 \to [16, 17]$
\item RCV $17 \to [24]$
\item RCV $18 \to [25]$
\item RCV $19 \to [26, 27, 28]$
\item RCV $20 \to [29]$
\item RCV $21 \to [30]$
\item RCV $22 \to [21, 22]$
\item RCV $23 \to [19]$
\item RCV $24 \to [25]$
\item RCV $25 \to [26, 27, 28]$
\item RCV $26 \to [31, 32]$
\item RCV $27 \to [33]$
\item RCV $28 \to [34, 35]$
\item RCV $29 \to [26, 27, 28]$
\item RCV $30 \to [29]$
\item RCV $31 \to [26, 27]$
\item RCV $32 \to [36]$
\item RCV $33 \to [37]$
\item RCV $34 \to [38]$
\item RCV $35 \to [27, 28]$
\item RCV $36 \to [34, 35]$
\item RCV $37 \to [26, 27, 28]$
\item RCV $38 \to [31, 32]$
\end{itemize}

See Figure \ref{fig:ArXPic3} for the transition diagram.
\begin{figure}[tbp]
\includegraphics[scale=0.4]{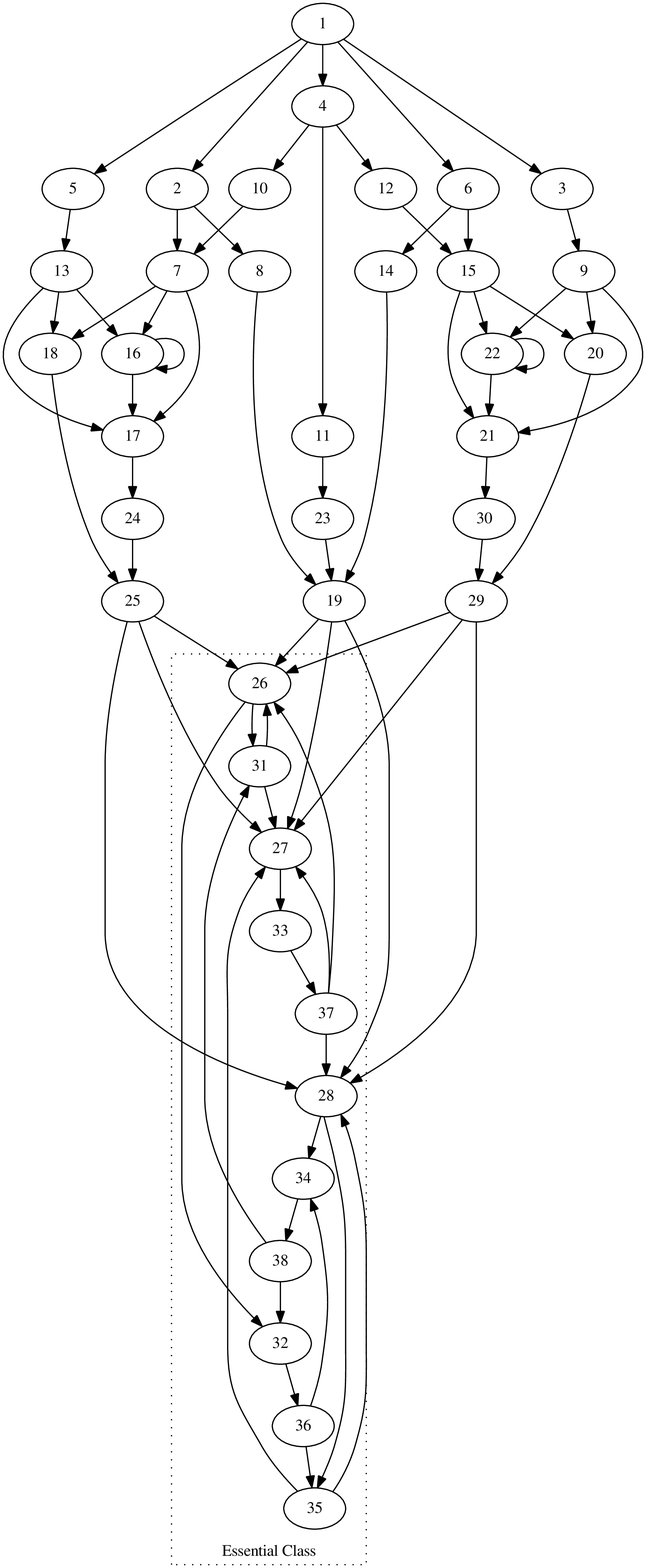}
\caption{$x^2+x-1$ with $d_i \in [0, 1-\varrho, 2-2 \varrho]$, Full set + Essential class}
\label{fig:ArXPic3}
\end{figure}
This has (normalized) transition matrices:
\begin{align*}
T(1,2) & =  \left[ \begin {array}{cccc} 1&0&0&0\\ \noalign{\medskip}0&1&2&1\end {array} \right] & 
 T(1,3) & =  \left[ \begin {array}{ccc} 1&0&0\\ \noalign{\medskip}0&1&2\end {array} \right] \\ 
T(1,4) & =  \left[ \begin {array}{cccc} 2&1&0&0\\ \noalign{\medskip}0&0&1&2\end {array} \right] & 
 T(1,5) & =  \left[ \begin {array}{ccc} 2&1&0\\ \noalign{\medskip}0&0&1\end {array} \right] \\ 
T(1,6) & =  \left[ \begin {array}{cccc} 1&2&1&0\\ \noalign{\medskip}0&0&0&1\end {array} \right] & 
 T(2,7) & =  \left[ \begin {array}{cccccc} 1&0&0&0&0&0\\ \noalign{\medskip}2&0&1&0&0&0\\ \noalign{\medskip}0&1&0&2&0&1\\ \noalign{\medskip}0&0&0&0&1&0\end {array} \right] \\ 
T(2,8) & =  \left[ \begin {array}{ccccc} 1&0&0&0&0\\ \noalign{\medskip}2&0&1&0&0\\ \noalign{\medskip}0&1&0&2&0\\ \noalign{\medskip}0&0&0&0&1\end {array} \right] & 
 T(3,9) & =  \left[ \begin {array}{ccccc} 2&1&0&0&0\\ \noalign{\medskip}1&2&0&1&0\\ \noalign{\medskip}0&0&1&0&2\end {array} \right] \\ 
T(4,10) & =  \left[ \begin {array}{ccccc} 1&0&0&0&0\\ \noalign{\medskip}0&2&1&0&0\\ \noalign{\medskip}0&1&2&0&1\\ \noalign{\medskip}0&0&0&1&0\end {array} \right] & 
 T(4,11) & =  \left[ \begin {array}{cccccc} 0&1&0&0&0&0\\ \noalign{\medskip}1&0&2&1&0&0\\ \noalign{\medskip}0&0&1&2&0&1\\ \noalign{\medskip}0&0&0&0&1&0\end {array} \right] \\ 
T(4,12) & =  \left[ \begin {array}{ccccc} 0&1&0&0&0\\ \noalign{\medskip}1&0&2&1&0\\ \noalign{\medskip}0&0&1&2&0\\ \noalign{\medskip}0&0&0&0&1\end {array} \right] & 
 T(5,13) & =  \left[ \begin {array}{ccccc} 2&0&1&0&0\\ \noalign{\medskip}0&1&0&2&1\\ \noalign{\medskip}0&0&0&1&2\end {array} \right] \\ 
T(6,14) & =  \left[ \begin {array}{ccccc} 1&0&0&0&0\\ \noalign{\medskip}0&2&0&1&0\\ \noalign{\medskip}0&0&1&0&2\\ \noalign{\medskip}0&0&0&0&1\end {array} \right] & 
 T(6,15) & =  \left[ \begin {array}{cccccc} 0&1&0&0&0&0\\ \noalign{\medskip}1&0&2&0&1&0\\ \noalign{\medskip}0&0&0&1&0&2\\ \noalign{\medskip}0&0&0&0&0&1\end {array} \right] \\ 
T(7,16) & =  \left[ \begin {array}{cccccc} 1&0&0&0&0&0\\ \noalign{\medskip}2&0&1&0&0&0\\ \noalign{\medskip}0&2&0&1&0&0\\ \noalign{\medskip}0&1&0&2&0&1\\ \noalign{\medskip}0&0&0&1&0&2\\ \noalign{\medskip}0&0&0&0&1&0\end {array} \right] \\ 
T(7,17) & =  \left[ \begin {array}{ccccccc} 0&1&0&0&0&0&0\\ \noalign{\medskip}0&2&0&1&0&0&0\\ \noalign{\medskip}1&0&2&0&1&0&0\\ \noalign{\medskip}0&0&1&0&2&0&1\\ \noalign{\medskip}0&0&0&0&1&0&2\\ \noalign{\medskip}0&0&0&0&0&1&0\end {array} \right] \\ 
T(7,18) & =  \left[ \begin {array}{cccccc} 0&1&0&0&0&0\\ \noalign{\medskip}0&2&0&1&0&0\\ \noalign{\medskip}1&0&2&0&1&0\\ \noalign{\medskip}0&0&1&0&2&0\\ \noalign{\medskip}0&0&0&0&1&0\\ \noalign{\medskip}0&0&0&0&0&1\end {array} \right] \\ 
T(8,19) & =  \left[ \begin {array}{cccccc} 2&0&1&0&0&0\\ \noalign{\medskip}1&0&2&0&1&0\\ \noalign{\medskip}0&1&0&2&0&1\\ \noalign{\medskip}0&0&0&1&0&2\\ \noalign{\medskip}0&0&0&0&0&1\end {array} \right] \\ 
T(9,20) & =  \left[ \begin {array}{cccccc} 1&0&0&0&0&0\\ \noalign{\medskip}0&2&0&1&0&0\\ \noalign{\medskip}0&1&0&2&0&1\\ \noalign{\medskip}0&0&1&0&2&0\\ \noalign{\medskip}0&0&0&0&1&0\end {array} \right] & 
 T(9,21) & =  \left[ \begin {array}{ccccccc} 0&1&0&0&0&0&0\\ \noalign{\medskip}1&0&2&0&1&0&0\\ \noalign{\medskip}0&0&1&0&2&0&1\\ \noalign{\medskip}0&0&0&1&0&2&0\\ \noalign{\medskip}0&0&0&0&0&1&0\end {array} \right] \\ 
T(9,22) & =  \left[ \begin {array}{cccccc} 0&1&0&0&0&0\\ \noalign{\medskip}1&0&2&0&1&0\\ \noalign{\medskip}0&0&1&0&2&0\\ \noalign{\medskip}0&0&0&1&0&2\\ \noalign{\medskip}0&0&0&0&0&1\end {array} \right] & 
 T(10,7) & =  \left[ \begin {array}{cccccc} 1&0&0&0&0&0\\ \noalign{\medskip}2&0&1&0&0&0\\ \noalign{\medskip}0&1&0&2&0&1\\ \noalign{\medskip}0&0&0&1&0&2\\ \noalign{\medskip}0&0&0&0&1&0\end {array} \right] \\ 
T(11,23) & =  \left[ \begin {array}{cccccc} 1&0&0&0&0&0\\ \noalign{\medskip}0&1&0&0&0&0\\ \noalign{\medskip}0&2&0&1&0&0\\ \noalign{\medskip}0&0&1&0&2&0\\ \noalign{\medskip}0&0&0&0&1&0\\ \noalign{\medskip}0&0&0&0&0&1\end {array} \right] \\ 
T(12,15) & =  \left[ \begin {array}{cccccc} 0&1&0&0&0&0\\ \noalign{\medskip}2&0&1&0&0&0\\ \noalign{\medskip}1&0&2&0&1&0\\ \noalign{\medskip}0&0&0&1&0&2\\ \noalign{\medskip}0&0&0&0&0&1\end {array} \right] \\ 
T(13,16) & =  \left[ \begin {array}{cccccc} 1&0&0&0&0&0\\ \noalign{\medskip}2&0&1&0&0&0\\ \noalign{\medskip}0&2&0&1&0&0\\ \noalign{\medskip}0&1&0&2&0&1\\ \noalign{\medskip}0&0&0&0&1&0\end {array} \right] & 
 T(13,17) & =  \left[ \begin {array}{ccccccc} 0&1&0&0&0&0&0\\ \noalign{\medskip}0&2&0&1&0&0&0\\ \noalign{\medskip}1&0&2&0&1&0&0\\ \noalign{\medskip}0&0&1&0&2&0&1\\ \noalign{\medskip}0&0&0&0&0&1&0\end {array} \right] \\ 
T(13,18) & =  \left[ \begin {array}{cccccc} 0&1&0&0&0&0\\ \noalign{\medskip}0&2&0&1&0&0\\ \noalign{\medskip}1&0&2&0&1&0\\ \noalign{\medskip}0&0&1&0&2&0\\ \noalign{\medskip}0&0&0&0&0&1\end {array} \right] & 
 T(14,19) & =  \left[ \begin {array}{cccccc} 1&0&0&0&0&0\\ \noalign{\medskip}2&0&1&0&0&0\\ \noalign{\medskip}1&0&2&0&1&0\\ \noalign{\medskip}0&1&0&2&0&1\\ \noalign{\medskip}0&0&0&1&0&2\end {array} \right] \\ 
T(15,20) & =  \left[ \begin {array}{cccccc} 1&0&0&0&0&0\\ \noalign{\medskip}0&1&0&0&0&0\\ \noalign{\medskip}0&2&0&1&0&0\\ \noalign{\medskip}0&1&0&2&0&1\\ \noalign{\medskip}0&0&1&0&2&0\\ \noalign{\medskip}0&0&0&0&1&0\end {array} \right] \\ 
T(15,21) & =  \left[ \begin {array}{ccccccc} 0&1&0&0&0&0&0\\ \noalign{\medskip}2&0&1&0&0&0&0\\ \noalign{\medskip}1&0&2&0&1&0&0\\ \noalign{\medskip}0&0&1&0&2&0&1\\ \noalign{\medskip}0&0&0&1&0&2&0\\ \noalign{\medskip}0&0&0&0&0&1&0\end {array} \right] \\ 
T(15,22) & =  \left[ \begin {array}{cccccc} 0&1&0&0&0&0\\ \noalign{\medskip}2&0&1&0&0&0\\ \noalign{\medskip}1&0&2&0&1&0\\ \noalign{\medskip}0&0&1&0&2&0\\ \noalign{\medskip}0&0&0&1&0&2\\ \noalign{\medskip}0&0&0&0&0&1\end {array} \right] \\ 
T(16,16) & =  \left[ \begin {array}{cccccc} 1&0&0&0&0&0\\ \noalign{\medskip}2&0&1&0&0&0\\ \noalign{\medskip}0&2&0&1&0&0\\ \noalign{\medskip}0&1&0&2&0&1\\ \noalign{\medskip}0&0&0&1&0&2\\ \noalign{\medskip}0&0&0&0&1&0\end {array} \right] \\ 
T(16,17) & =  \left[ \begin {array}{ccccccc} 0&1&0&0&0&0&0\\ \noalign{\medskip}0&2&0&1&0&0&0\\ \noalign{\medskip}1&0&2&0&1&0&0\\ \noalign{\medskip}0&0&1&0&2&0&1\\ \noalign{\medskip}0&0&0&0&1&0&2\\ \noalign{\medskip}0&0&0&0&0&1&0\end {array} \right] \\ 
T(17,24) & =  \left[ \begin {array}{ccccccc} 1&0&0&0&0&0&0\\ \noalign{\medskip}0&0&1&0&0&0&0\\ \noalign{\medskip}0&0&2&0&1&0&0\\ \noalign{\medskip}0&1&0&2&0&1&0\\ \noalign{\medskip}0&0&0&1&0&2&0\\ \noalign{\medskip}0&0&0&0&0&1&0\\ \noalign{\medskip}0&0&0&0&0&0&1\end {array} \right] \\ 
T(18,25) & =  \left[ \begin {array}{ccccccc} 0&1&0&0&0&0&0\\ \noalign{\medskip}2&0&0&1&0&0&0\\ \noalign{\medskip}1&0&0&2&0&1&0\\ \noalign{\medskip}0&0&1&0&2&0&1\\ \noalign{\medskip}0&0&0&0&1&0&2\\ \noalign{\medskip}0&0&0&0&0&0&1\end {array} \right] \\ 
T(19,26) & =  \left[ \begin {array}{ccccccc} 1&0&0&0&0&0&0\\ \noalign{\medskip}2&0&1&0&0&0&0\\ \noalign{\medskip}0&2&0&0&1&0&0\\ \noalign{\medskip}0&1&0&0&2&0&1\\ \noalign{\medskip}0&0&0&1&0&2&0\\ \noalign{\medskip}0&0&0&0&0&1&0\end {array} \right] \\ 
T(19,27) & =  \left[ \begin {array}{cccccccc} 0&1&0&0&0&0&0&0\\ \noalign{\medskip}0&2&0&1&0&0&0&0\\ \noalign{\medskip}1&0&2&0&0&1&0&0\\ \noalign{\medskip}0&0&1&0&0&2&0&1\\ \noalign{\medskip}0&0&0&0&1&0&2&0\\ \noalign{\medskip}0&0&0&0&0&0&1&0\end {array} \right] \\ 
T(19,28) & =  \left[ \begin {array}{ccccccc} 0&1&0&0&0&0&0\\ \noalign{\medskip}0&2&0&1&0&0&0\\ \noalign{\medskip}1&0&2&0&0&1&0\\ \noalign{\medskip}0&0&1&0&0&2&0\\ \noalign{\medskip}0&0&0&0&1&0&2\\ \noalign{\medskip}0&0&0&0&0&0&1\end {array} \right] \\ 
T(20,29) & =  \left[ \begin {array}{ccccccc} 1&0&0&0&0&0&0\\ \noalign{\medskip}2&0&1&0&0&0&0\\ \noalign{\medskip}1&0&2&0&1&0&0\\ \noalign{\medskip}0&1&0&2&0&0&1\\ \noalign{\medskip}0&0&0&1&0&0&2\\ \noalign{\medskip}0&0&0&0&0&1&0\end {array} \right] \\ 
T(21,30) & =  \left[ \begin {array}{ccccccc} 1&0&0&0&0&0&0\\ \noalign{\medskip}0&1&0&0&0&0&0\\ \noalign{\medskip}0&2&0&1&0&0&0\\ \noalign{\medskip}0&1&0&2&0&1&0\\ \noalign{\medskip}0&0&1&0&2&0&0\\ \noalign{\medskip}0&0&0&0&1&0&0\\ \noalign{\medskip}0&0&0&0&0&0&1\end {array} \right] \\ 
T(22,21) & =  \left[ \begin {array}{ccccccc} 0&1&0&0&0&0&0\\ \noalign{\medskip}2&0&1&0&0&0&0\\ \noalign{\medskip}1&0&2&0&1&0&0\\ \noalign{\medskip}0&0&1&0&2&0&1\\ \noalign{\medskip}0&0&0&1&0&2&0\\ \noalign{\medskip}0&0&0&0&0&1&0\end {array} \right] \\ 
T(22,22) & =  \left[ \begin {array}{cccccc} 0&1&0&0&0&0\\ \noalign{\medskip}2&0&1&0&0&0\\ \noalign{\medskip}1&0&2&0&1&0\\ \noalign{\medskip}0&0&1&0&2&0\\ \noalign{\medskip}0&0&0&1&0&2\\ \noalign{\medskip}0&0&0&0&0&1\end {array} \right] \\ 
T(23,19) & =  \left[ \begin {array}{cccccc} 1&0&0&0&0&0\\ \noalign{\medskip}2&0&1&0&0&0\\ \noalign{\medskip}1&0&2&0&1&0\\ \noalign{\medskip}0&1&0&2&0&1\\ \noalign{\medskip}0&0&0&1&0&2\\ \noalign{\medskip}0&0&0&0&0&1\end {array} \right] \\ 
T(24,25) & =  \left[ \begin {array}{ccccccc} 1&0&0&0&0&0&0\\ \noalign{\medskip}0&1&0&0&0&0&0\\ \noalign{\medskip}2&0&0&1&0&0&0\\ \noalign{\medskip}1&0&0&2&0&1&0\\ \noalign{\medskip}0&0&1&0&2&0&1\\ \noalign{\medskip}0&0&0&0&1&0&2\\ \noalign{\medskip}0&0&0&0&0&0&1\end {array} \right] \\ 
T(25,26) & =  \left[ \begin {array}{ccccccc} 1&0&0&0&0&0&0\\ \noalign{\medskip}0&1&0&0&0&0&0\\ \noalign{\medskip}2&0&1&0&0&0&0\\ \noalign{\medskip}0&2&0&0&1&0&0\\ \noalign{\medskip}0&1&0&0&2&0&1\\ \noalign{\medskip}0&0&0&1&0&2&0\\ \noalign{\medskip}0&0&0&0&0&1&0\end {array} \right] \\ 
T(25,27) & =  \left[ \begin {array}{cccccccc} 0&1&0&0&0&0&0&0\\ \noalign{\medskip}2&0&1&0&0&0&0&0\\ \noalign{\medskip}0&2&0&1&0&0&0&0\\ \noalign{\medskip}1&0&2&0&0&1&0&0\\ \noalign{\medskip}0&0&1&0&0&2&0&1\\ \noalign{\medskip}0&0&0&0&1&0&2&0\\ \noalign{\medskip}0&0&0&0&0&0&1&0\end {array} \right] \\ 
T(25,28) & =  \left[ \begin {array}{ccccccc} 0&1&0&0&0&0&0\\ \noalign{\medskip}2&0&1&0&0&0&0\\ \noalign{\medskip}0&2&0&1&0&0&0\\ \noalign{\medskip}1&0&2&0&0&1&0\\ \noalign{\medskip}0&0&1&0&0&2&0\\ \noalign{\medskip}0&0&0&0&1&0&2\\ \noalign{\medskip}0&0&0&0&0&0&1\end {array} \right] \\ 
T(26,31) & =  \left[ \begin {array}{ccccccc} 1&0&0&0&0&0&0\\ \noalign{\medskip}2&0&1&0&0&0&0\\ \noalign{\medskip}0&2&0&1&0&0&0\\ \noalign{\medskip}1&0&2&0&1&0&0\\ \noalign{\medskip}0&1&0&2&0&0&1\\ \noalign{\medskip}0&0&0&1&0&0&2\\ \noalign{\medskip}0&0&0&0&0&1&0\end {array} \right] \\ 
T(26,32) & =  \left[ \begin {array}{cccccccc} 0&1&0&0&0&0&0&0\\ \noalign{\medskip}0&2&0&1&0&0&0&0\\ \noalign{\medskip}1&0&2&0&1&0&0&0\\ \noalign{\medskip}0&1&0&2&0&1&0&0\\ \noalign{\medskip}0&0&1&0&2&0&0&1\\ \noalign{\medskip}0&0&0&0&1&0&0&2\\ \noalign{\medskip}0&0&0&0&0&0&1&0\end {array} \right] \\ 
T(27,33) & =  \left[ \begin {array}{cccccccc} 1&0&0&0&0&0&0&0\\ \noalign{\medskip}0&0&1&0&0&0&0&0\\ \noalign{\medskip}0&0&2&0&1&0&0&0\\ \noalign{\medskip}0&1&0&2&0&1&0&0\\ \noalign{\medskip}0&0&1&0&2&0&1&0\\ \noalign{\medskip}0&0&0&1&0&2&0&0\\ \noalign{\medskip}0&0&0&0&0&1&0&0\\ \noalign{\medskip}0&0&0&0&0&0&0&1\end {array} \right] \\ 
T(28,34) & =  \left[ \begin {array}{cccccccc} 0&1&0&0&0&0&0&0\\ \noalign{\medskip}2&0&0&1&0&0&0&0\\ \noalign{\medskip}1&0&0&2&0&1&0&0\\ \noalign{\medskip}0&0&1&0&2&0&1&0\\ \noalign{\medskip}0&0&0&1&0&2&0&1\\ \noalign{\medskip}0&0&0&0&1&0&2&0\\ \noalign{\medskip}0&0&0&0&0&0&1&0\end {array} \right] \\ 
T(28,35) & =  \left[ \begin {array}{ccccccc} 0&1&0&0&0&0&0\\ \noalign{\medskip}2&0&0&1&0&0&0\\ \noalign{\medskip}1&0&0&2&0&1&0\\ \noalign{\medskip}0&0&1&0&2&0&1\\ \noalign{\medskip}0&0&0&1&0&2&0\\ \noalign{\medskip}0&0&0&0&1&0&2\\ \noalign{\medskip}0&0&0&0&0&0&1\end {array} \right] \\ 
T(29,26) & =  \left[ \begin {array}{ccccccc} 1&0&0&0&0&0&0\\ \noalign{\medskip}2&0&1&0&0&0&0\\ \noalign{\medskip}0&2&0&0&1&0&0\\ \noalign{\medskip}0&1&0&0&2&0&1\\ \noalign{\medskip}0&0&0&1&0&2&0\\ \noalign{\medskip}0&0&0&0&1&0&2\\ \noalign{\medskip}0&0&0&0&0&1&0\end {array} \right] \\ 
T(29,27) & =  \left[ \begin {array}{cccccccc} 0&1&0&0&0&0&0&0\\ \noalign{\medskip}0&2&0&1&0&0&0&0\\ \noalign{\medskip}1&0&2&0&0&1&0&0\\ \noalign{\medskip}0&0&1&0&0&2&0&1\\ \noalign{\medskip}0&0&0&0&1&0&2&0\\ \noalign{\medskip}0&0&0&0&0&1&0&2\\ \noalign{\medskip}0&0&0&0&0&0&1&0\end {array} \right] \\ 
T(29,28) & =  \left[ \begin {array}{ccccccc} 0&1&0&0&0&0&0\\ \noalign{\medskip}0&2&0&1&0&0&0\\ \noalign{\medskip}1&0&2&0&0&1&0\\ \noalign{\medskip}0&0&1&0&0&2&0\\ \noalign{\medskip}0&0&0&0&1&0&2\\ \noalign{\medskip}0&0&0&0&0&1&0\\ \noalign{\medskip}0&0&0&0&0&0&1\end {array} \right] \\ 
T(30,29) & =  \left[ \begin {array}{ccccccc} 1&0&0&0&0&0&0\\ \noalign{\medskip}2&0&1&0&0&0&0\\ \noalign{\medskip}1&0&2&0&1&0&0\\ \noalign{\medskip}0&1&0&2&0&0&1\\ \noalign{\medskip}0&0&0&1&0&0&2\\ \noalign{\medskip}0&0&0&0&0&1&0\\ \noalign{\medskip}0&0&0&0&0&0&1\end {array} \right] \\ 
T(31,26) & =  \left[ \begin {array}{ccccccc} 1&0&0&0&0&0&0\\ \noalign{\medskip}2&0&1&0&0&0&0\\ \noalign{\medskip}0&2&0&0&1&0&0\\ \noalign{\medskip}0&1&0&0&2&0&1\\ \noalign{\medskip}0&0&0&1&0&2&0\\ \noalign{\medskip}0&0&0&0&1&0&2\\ \noalign{\medskip}0&0&0&0&0&1&0\end {array} \right] \\ 
T(31,27) & =  \left[ \begin {array}{cccccccc} 0&1&0&0&0&0&0&0\\ \noalign{\medskip}0&2&0&1&0&0&0&0\\ \noalign{\medskip}1&0&2&0&0&1&0&0\\ \noalign{\medskip}0&0&1&0&0&2&0&1\\ \noalign{\medskip}0&0&0&0&1&0&2&0\\ \noalign{\medskip}0&0&0&0&0&1&0&2\\ \noalign{\medskip}0&0&0&0&0&0&1&0\end {array} \right] \\ 
T(32,36) & =  \left[ \begin {array}{cccccccc} 1&0&0&0&0&0&0&0\\ \noalign{\medskip}0&0&1&0&0&0&0&0\\ \noalign{\medskip}0&0&2&0&1&0&0&0\\ \noalign{\medskip}0&1&0&2&0&0&1&0\\ \noalign{\medskip}0&0&0&1&0&0&2&0\\ \noalign{\medskip}0&0&0&0&0&1&0&2\\ \noalign{\medskip}0&0&0&0&0&0&1&0\\ \noalign{\medskip}0&0&0&0&0&0&0&1\end {array} \right] \\ 
T(33,37) & =  \left[ \begin {array}{cccccccc} 1&0&0&0&0&0&0&0\\ \noalign{\medskip}0&1&0&0&0&0&0&0\\ \noalign{\medskip}2&0&0&1&0&0&0&0\\ \noalign{\medskip}1&0&0&2&0&1&0&0\\ \noalign{\medskip}0&0&1&0&2&0&0&1\\ \noalign{\medskip}0&0&0&0&1&0&0&2\\ \noalign{\medskip}0&0&0&0&0&0&1&0\\ \noalign{\medskip}0&0&0&0&0&0&0&1\end {array} \right] \\ 
T(34,38) & =  \left[ \begin {array}{cccccccc} 1&0&0&0&0&0&0&0\\ \noalign{\medskip}0&1&0&0&0&0&0&0\\ \noalign{\medskip}2&0&1&0&0&0&0&0\\ \noalign{\medskip}0&2&0&0&1&0&0&0\\ \noalign{\medskip}0&1&0&0&2&0&1&0\\ \noalign{\medskip}0&0&0&1&0&2&0&0\\ \noalign{\medskip}0&0&0&0&0&1&0&0\\ \noalign{\medskip}0&0&0&0&0&0&0&1\end {array} \right] \\ 
T(35,27) & =  \left[ \begin {array}{cccccccc} 0&1&0&0&0&0&0&0\\ \noalign{\medskip}2&0&1&0&0&0&0&0\\ \noalign{\medskip}0&2&0&1&0&0&0&0\\ \noalign{\medskip}1&0&2&0&0&1&0&0\\ \noalign{\medskip}0&0&1&0&0&2&0&1\\ \noalign{\medskip}0&0&0&0&1&0&2&0\\ \noalign{\medskip}0&0&0&0&0&0&1&0\end {array} \right] \\ 
T(35,28) & =  \left[ \begin {array}{ccccccc} 0&1&0&0&0&0&0\\ \noalign{\medskip}2&0&1&0&0&0&0\\ \noalign{\medskip}0&2&0&1&0&0&0\\ \noalign{\medskip}1&0&2&0&0&1&0\\ \noalign{\medskip}0&0&1&0&0&2&0\\ \noalign{\medskip}0&0&0&0&1&0&2\\ \noalign{\medskip}0&0&0&0&0&0&1\end {array} \right] \\ 
T(36,34) & =  \left[ \begin {array}{cccccccc} 1&0&0&0&0&0&0&0\\ \noalign{\medskip}0&1&0&0&0&0&0&0\\ \noalign{\medskip}2&0&0&1&0&0&0&0\\ \noalign{\medskip}1&0&0&2&0&1&0&0\\ \noalign{\medskip}0&0&1&0&2&0&1&0\\ \noalign{\medskip}0&0&0&1&0&2&0&1\\ \noalign{\medskip}0&0&0&0&1&0&2&0\\ \noalign{\medskip}0&0&0&0&0&0&1&0\end {array} \right] \\ 
T(36,35) & =  \left[ \begin {array}{ccccccc} 1&0&0&0&0&0&0\\ \noalign{\medskip}0&1&0&0&0&0&0\\ \noalign{\medskip}2&0&0&1&0&0&0\\ \noalign{\medskip}1&0&0&2&0&1&0\\ \noalign{\medskip}0&0&1&0&2&0&1\\ \noalign{\medskip}0&0&0&1&0&2&0\\ \noalign{\medskip}0&0&0&0&1&0&2\\ \noalign{\medskip}0&0&0&0&0&0&1\end {array} \right] \\ 
T(37,26) & =  \left[ \begin {array}{ccccccc} 1&0&0&0&0&0&0\\ \noalign{\medskip}0&1&0&0&0&0&0\\ \noalign{\medskip}2&0&1&0&0&0&0\\ \noalign{\medskip}0&2&0&0&1&0&0\\ \noalign{\medskip}0&1&0&0&2&0&1\\ \noalign{\medskip}0&0&0&1&0&2&0\\ \noalign{\medskip}0&0&0&0&1&0&2\\ \noalign{\medskip}0&0&0&0&0&1&0\end {array} \right] \\ 
T(37,27) & =  \left[ \begin {array}{cccccccc} 0&1&0&0&0&0&0&0\\ \noalign{\medskip}2&0&1&0&0&0&0&0\\ \noalign{\medskip}0&2&0&1&0&0&0&0\\ \noalign{\medskip}1&0&2&0&0&1&0&0\\ \noalign{\medskip}0&0&1&0&0&2&0&1\\ \noalign{\medskip}0&0&0&0&1&0&2&0\\ \noalign{\medskip}0&0&0&0&0&1&0&2\\ \noalign{\medskip}0&0&0&0&0&0&1&0\end {array} \right] \\ 
T(37,28) & =  \left[ \begin {array}{ccccccc} 0&1&0&0&0&0&0\\ \noalign{\medskip}2&0&1&0&0&0&0\\ \noalign{\medskip}0&2&0&1&0&0&0\\ \noalign{\medskip}1&0&2&0&0&1&0\\ \noalign{\medskip}0&0&1&0&0&2&0\\ \noalign{\medskip}0&0&0&0&1&0&2\\ \noalign{\medskip}0&0&0&0&0&1&0\\ \noalign{\medskip}0&0&0&0&0&0&1\end {array} \right] \\ 
T(38,31) & =  \left[ \begin {array}{ccccccc} 1&0&0&0&0&0&0\\ \noalign{\medskip}2&0&1&0&0&0&0\\ \noalign{\medskip}0&2&0&1&0&0&0\\ \noalign{\medskip}1&0&2&0&1&0&0\\ \noalign{\medskip}0&1&0&2&0&0&1\\ \noalign{\medskip}0&0&0&1&0&0&2\\ \noalign{\medskip}0&0&0&0&0&1&0\\ \noalign{\medskip}0&0&0&0&0&0&1\end {array} \right] \\ 
T(38,32) & =  \left[ \begin {array}{cccccccc} 0&1&0&0&0&0&0&0\\ \noalign{\medskip}0&2&0&1&0&0&0&0\\ \noalign{\medskip}1&0&2&0&1&0&0&0\\ \noalign{\medskip}0&1&0&2&0&1&0&0\\ \noalign{\medskip}0&0&1&0&2&0&0&1\\ \noalign{\medskip}0&0&0&0&1&0&0&2\\ \noalign{\medskip}0&0&0&0&0&0&1&0\\ \noalign{\medskip}0&0&0&0&0&0&0&1\end {array} \right] \\ 
\end{align*}

The essential class is: [26, 27, 28, 31, 32, 33, 34, 35, 36, 37, 38].
The essential class is of positive type.
An example is the path [26, 32, 36, 34, 38, 32, 36, 35].
The essential class is not a simple loop.
This spectral range will include the interval $[2.469158042, 2.481194304]$.
The minimum comes from the loop $[27, 33, 37, 27]$.
The maximim comes from the loop $[26, 31, 26]$.
These points will include points of local dimension [.992399434, 1.002504754].
The Spectral Range is contained in the range $[2.038390910, 2.701372314]$.
The minimum comes from the column sub-norm on the subset ${{3, 4}}$ of length 20. 
The maximum comes from the total column sup-norm of length 10. 
These points will have local dimension contained in [.815720713, 1.400908289].

There are 2 additional maximal loops.

Maximal Loop Class: [22].
The maximal loop class is a simple loop.
It's spectral radius is an isolated points of 2.481194304.
These points have local dimension .992399434.

Maximal Loop Class: [16].
The maximal loop class is a simple loop.
It's spectral radius is an isolated points of 2.481194304.
These points have local dimension .992399434.

The set of local dimensions will contain
\[[.9924,1.003]\]This has 1 components. 

The set of local dimensions is contained in
 \[[.8157,1.401]\]This has 1 components.

\section{Details for Example \ref{Ex:isolatedPt}}
 
Consider the measure given by the maps $S_i(x) = x/4  + d_i$ with $d_{0} = 0$, $d_{1} = 1/8$, $d_{2} = 1/4$, $d_{3} = 3/8$, $d_{4} = 1/2$, $d_{5} = 7/8$, $d_{6} = 9/8$, and $d_{7} = 3/2$.
The probabilities are given by $p_{0} = 1/2402$, $p_{1} = 500/1201$, $p_{2} = 500/1201$, $p_{3} = 50/1201$, $p_{4} = 50/1201$, $p_{5} = 50/1201$, $p_{6} = 50/1201$, and $p_{7} = 1/2402$.
We consider the quotient measure.
The reduced transition diagram has 10 reduced characteristic vectors.
The reduced characteristic vectors are:
\begin{itemize}
\item Reduced characteristic vector 1: $(2, (0, 2))$ 
\item Reduced characteristic vector 2: $(1, (0, 1))$ 
\item Reduced characteristic vector 3: $(1, (0, 1, 2))$ 
\item Reduced characteristic vector 4: $(1, (0, 1, 2, 3))$ 
\item Reduced characteristic vector 5: $(1, (1, 2, 3))$ 
\item Reduced characteristic vector 6: $(1, (2, 3))$ 
\item Reduced characteristic vector 7: $(1, (0, 3))$ 
\item Reduced characteristic vector 8: $(1, (0, 2, 3))$ 
\item Reduced characteristic vector 9: $(1, (0, 1, 3))$ 
\item Reduced characteristic vector 10: $(2, (0, 1, 2))$ 
\end{itemize}
The maps are:
\begin{itemize}
\item RCV $1 \to [2, 3, 4, 4, 4, 5, 6, 7]$
\item RCV $2 \to [4, 4, 4, 4]$
\item RCV $3 \to [4, 4, 4, 4]$
\item RCV $4 \to [4, 4, 4, 4]$
\item RCV $5 \to [4, 4, 5, 8]$
\item RCV $6 \to [9, 10, 6]$
\item RCV $7 \to [7, 2, 3, 4]$
\item RCV $8 \to [9, 3, 4, 4]$
\item RCV $9 \to [4, 4, 4, 4]$
\item RCV $10 \to [4, 4, 4, 4, 4, 4, 5, 8]$
\end{itemize}

See Figure \ref{fig:ArXPic4} for the transition diagram.
\begin{figure}[tbp]
\includegraphics[scale=0.5]{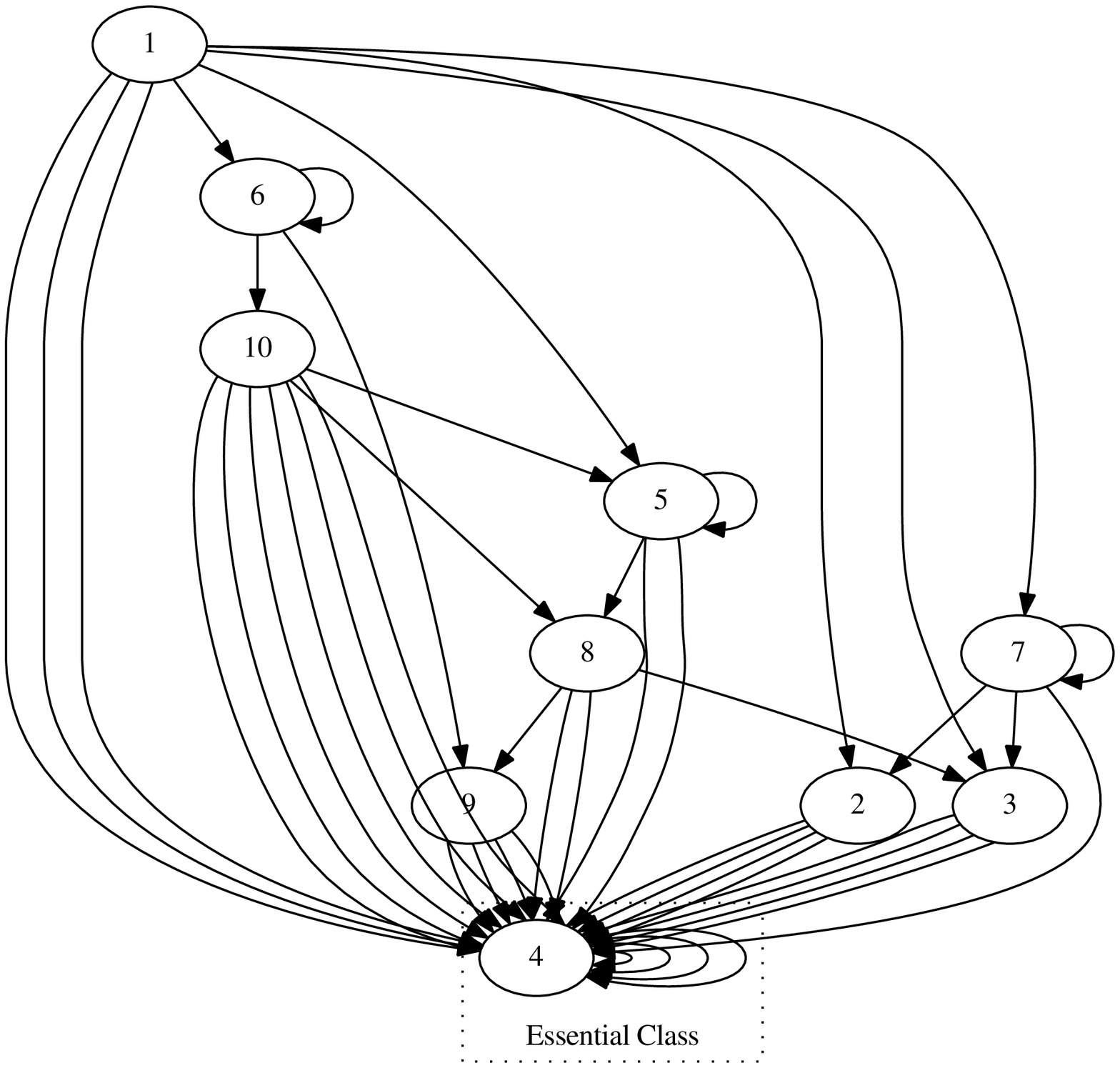}
\caption{$4 x-1$ with $d_i \in [0, 1/8, 1/4, 3/8, 1/2, 7/8, 9/8, 3/2]$, Full set + Essential class}
\label{fig:ArXPic4}
\end{figure}
This has (normalized) transition matrices:
\begin{align*}
T(1,2) & =  \left[ \begin {array}{cc} 1&0\\ \noalign{\medskip}0&100\end {array} \right] & 
 T(1,3) & =  \left[ \begin {array}{ccc} 1000&1&0\\ \noalign{\medskip}100&0&100\end {array} \right] \\ 
T(1,4) & =  \left[ \begin {array}{cccc} 1000&1000&1&0\\ \noalign{\medskip}0&100&0&100\end {array} \right] & 
 T(1,4) & =  \left[ \begin {array}{cccc} 100&1000&1000&1\\ \noalign{\medskip}0&0&100&0\end {array} \right] \\ 
T(1,4) & =  \left[ \begin {array}{cccc} 100&100&1000&1000\\ \noalign{\medskip}1&0&0&100\end {array} \right] & 
 T(1,5) & =  \left[ \begin {array}{ccc} 100&100&1000\\ \noalign{\medskip}1&0&0\end {array} \right] \\ 
T(1,6) & =  \left[ \begin {array}{cc} 100&100\\ \noalign{\medskip}1&0\end {array} \right] & 
 T(1,7) & =  \left[ \begin {array}{cc} 100&100\\ \noalign{\medskip}0&1\end {array} \right] \\ 
T(2,4) & =  \left[ \begin {array}{cccc} 1&0&0&0\\ \noalign{\medskip}100&100&1000&1000\end {array} \right] & 
 T(2,4) & =  \left[ \begin {array}{cccc} 1000&1&0&0\\ \noalign{\medskip}0&100&100&1000\end {array} \right] \\ 
T(2,4) & =  \left[ \begin {array}{cccc} 1000&1000&1&0\\ \noalign{\medskip}0&0&100&100\end {array} \right] & 
 T(2,4) & =  \left[ \begin {array}{cccc} 100&1000&1000&1\\ \noalign{\medskip}100&0&0&100\end {array} \right] \\ 
T(3,4) & =  \left[ \begin {array}{cccc} 1&0&0&0\\ \noalign{\medskip}100&100&1000&1000\\ \noalign{\medskip}0&100&0&0\end {array} \right] & 
 T(3,4) & =  \left[ \begin {array}{cccc} 1000&1&0&0\\ \noalign{\medskip}0&100&100&1000\\ \noalign{\medskip}100&0&100&0\end {array} \right] \\ 
T(3,4) & =  \left[ \begin {array}{cccc} 1000&1000&1&0\\ \noalign{\medskip}0&0&100&100\\ \noalign{\medskip}0&100&0&100\end {array} \right] & 
 T(3,4) & =  \left[ \begin {array}{cccc} 100&1000&1000&1\\ \noalign{\medskip}100&0&0&100\\ \noalign{\medskip}0&0&100&0\end {array} \right] \\ 
T(4,4) & =  \left[ \begin {array}{cccc} 1&0&0&0\\ \noalign{\medskip}100&100&1000&1000\\ \noalign{\medskip}0&100&0&0\\ \noalign{\medskip}1&0&0&100\end {array} \right] & 
 T(4,4) & =  \left[ \begin {array}{cccc} 1000&1&0&0\\ \noalign{\medskip}0&100&100&1000\\ \noalign{\medskip}100&0&100&0\\ \noalign{\medskip}0&1&0&0\end {array} \right] \\ 
T(4,4) & =  \left[ \begin {array}{cccc} 1000&1000&1&0\\ \noalign{\medskip}0&0&100&100\\ \noalign{\medskip}0&100&0&100\\ \noalign{\medskip}0&0&1&0\end {array} \right] & 
 T(4,4) & =  \left[ \begin {array}{cccc} 100&1000&1000&1\\ \noalign{\medskip}100&0&0&100\\ \noalign{\medskip}0&0&100&0\\ \noalign{\medskip}0&0&0&1\end {array} \right] \\ 
T(5,4) & =  \left[ \begin {array}{cccc} 100&100&1000&1000\\ \noalign{\medskip}0&100&0&0\\ \noalign{\medskip}1&0&0&100\end {array} \right] & 
 T(5,4) & =  \left[ \begin {array}{cccc} 0&100&100&1000\\ \noalign{\medskip}100&0&100&0\\ \noalign{\medskip}0&1&0&0\end {array} \right] \\ 
T(5,5) & =  \left[ \begin {array}{ccc} 0&100&100\\ \noalign{\medskip}100&0&100\\ \noalign{\medskip}0&1&0\end {array} \right] & 
 T(5,8) & =  \left[ \begin {array}{ccc} 100&0&100\\ \noalign{\medskip}0&100&0\\ \noalign{\medskip}0&0&1\end {array} \right] \\ 
T(6,9) & =  \left[ \begin {array}{ccc} 0&100&0\\ \noalign{\medskip}1&0&100\end {array} \right] & 
 T(6,10) & =  \left[ \begin {array}{ccc} 100&0&100\\ \noalign{\medskip}0&1&0\end {array} \right] \\ 
T(6,6) & =  \left[ \begin {array}{cc} 100&0\\ \noalign{\medskip}0&1\end {array} \right] & 
 T(7,7) & =  \left[ \begin {array}{cc} 1&0\\ \noalign{\medskip}1&100\end {array} \right] \\ 
T(7,2) & =  \left[ \begin {array}{cc} 1000&1\\ \noalign{\medskip}0&1\end {array} \right] & 
 T(7,3) & =  \left[ \begin {array}{ccc} 1000&1000&1\\ \noalign{\medskip}0&0&1\end {array} \right] \\ 
T(7,4) & =  \left[ \begin {array}{cccc} 100&1000&1000&1\\ \noalign{\medskip}0&0&0&1\end {array} \right] & 
 T(8,9) & =  \left[ \begin {array}{ccc} 1&0&0\\ \noalign{\medskip}0&100&0\\ \noalign{\medskip}1&0&100\end {array} \right] \\ 
T(8,3) & =  \left[ \begin {array}{ccc} 1000&1&0\\ \noalign{\medskip}100&0&100\\ \noalign{\medskip}0&1&0\end {array} \right] & 
 T(8,4) & =  \left[ \begin {array}{cccc} 1000&1000&1&0\\ \noalign{\medskip}0&100&0&100\\ \noalign{\medskip}0&0&1&0\end {array} \right] \\ 
T(8,4) & =  \left[ \begin {array}{cccc} 100&1000&1000&1\\ \noalign{\medskip}0&0&100&0\\ \noalign{\medskip}0&0&0&1\end {array} \right] & 
 T(9,4) & =  \left[ \begin {array}{cccc} 1&0&0&0\\ \noalign{\medskip}100&100&1000&1000\\ \noalign{\medskip}1&0&0&100\end {array} \right] \\ 
T(9,4) & =  \left[ \begin {array}{cccc} 1000&1&0&0\\ \noalign{\medskip}0&100&100&1000\\ \noalign{\medskip}0&1&0&0\end {array} \right] & 
 T(9,4) & =  \left[ \begin {array}{cccc} 1000&1000&1&0\\ \noalign{\medskip}0&0&100&100\\ \noalign{\medskip}0&0&1&0\end {array} \right] \\ 
T(9,4) & =  \left[ \begin {array}{cccc} 100&1000&1000&1\\ \noalign{\medskip}100&0&0&100\\ \noalign{\medskip}0&0&0&1\end {array} \right] & 
 T(10,4) & =  \left[ \begin {array}{cccc} 1&0&0&0\\ \noalign{\medskip}100&100&1000&1000\\ \noalign{\medskip}0&100&0&0\end {array} \right] \\ 
T(10,4) & =  \left[ \begin {array}{cccc} 1000&1&0&0\\ \noalign{\medskip}0&100&100&1000\\ \noalign{\medskip}100&0&100&0\end {array} \right] & 
 T(10,4) & =  \left[ \begin {array}{cccc} 1000&1000&1&0\\ \noalign{\medskip}0&0&100&100\\ \noalign{\medskip}0&100&0&100\end {array} \right] \\ 
T(10,4) & =  \left[ \begin {array}{cccc} 100&1000&1000&1\\ \noalign{\medskip}100&0&0&100\\ \noalign{\medskip}0&0&100&0\end {array} \right] & 
 T(10,4) & =  \left[ \begin {array}{cccc} 100&100&1000&1000\\ \noalign{\medskip}0&100&0&0\\ \noalign{\medskip}1&0&0&100\end {array} \right] \\ 
T(10,4) & =  \left[ \begin {array}{cccc} 0&100&100&1000\\ \noalign{\medskip}100&0&100&0\\ \noalign{\medskip}0&1&0&0\end {array} \right] & 
 T(10,5) & =  \left[ \begin {array}{ccc} 0&100&100\\ \noalign{\medskip}100&0&100\\ \noalign{\medskip}0&1&0\end {array} \right] \\ 
T(10,8) & =  \left[ \begin {array}{ccc} 100&0&100\\ \noalign{\medskip}0&100&0\\ \noalign{\medskip}0&0&1\end {array} \right] & 
 \end{align*}

The essential class is: [4].
The essential class is of positive type.
An example is the path [4, 4, 4, 4].
The essential class is not a simple loop.
This spectral range will include the interval $[175.9930871, 1005.242481]$.
The minimum comes from the loop $[4, 4, 4, 4]$.
The maximim comes from the loop $[4, 4, 4, 4, 4]$.
These points will include points of local dimension [.628346304, 1.885322743].
The Spectral Range is contained in the range $[139.5519097, 1025.016616]$.
The minimum comes from the total row sub-norm of length 5. 
The maximum comes from the total row sup-norm of length 5. 
These points will have local dimension contained in [.614294428, 2.052681190].

There are 3 additional maximal loops.

Maximal Loop Class: [7].
The maximal loop class is a simple loop.
It's spectral radius is an isolated points of 100.00.
These points have local dimension 2.293082124.

Maximal Loop Class: [6].
The maximal loop class is a simple loop.
It's spectral radius is an isolated points of 100.00.
These points have local dimension 2.293082124.

Maximal Loop Class: [5].
The maximal loop class is a simple loop.
It's spectral radius is an isolated points of 100.9901951.
These points have local dimension 2.285974508.

The set of local dimensions will contain
\[[.6283,1.885]\cup \{ 2.286\}\cup \{ 2.293\}\]This has 3 components. 

The set of local dimensions is contained in
 \[[.6143,2.053]\cup \{ 2.286\}\cup \{ 2.293\}\]This has 3 components. 

\section{Details for Remark \ref{remark:conj}}

Below we have provided details to support our conjecture in Remark 
    \ref{remark:conj}.  
Let $\nu$ be the $m$-fold convolution of the Cantor measure with ratio 
    of contraction $1/d$.
We know that:
\begin{equation*}
\{\dim _{loc}\nu _{\pi }(x):x\in \text{supp }\nu _{\pi }\}\subseteq \{\dim
_{loc}\nu (x):x\in \text{supp }\nu ,x\neq 0,d+k\}
\end{equation*}
We wish to show that 
$\sup_{x}\dim _{loc}\nu _{\pi }(x)<\sup_{x\neq 0,m}\dim _{loc}\nu (x)$.
We do this by showing that the upper bound for 
    $\sup_{x\neq 0,d+k}\dim _{loc}\nu (x)$ is bounded below by 
    $b$ (dependent on $m$ and $d$), 
    $\sup_{x\neq 0,d+k}\dim _{loc}\nu (x)$ is bounded above by
    $a$ (dependent on $m$ and $d$), 
    and then observe that $a < b$ in every case except $m = 2, d =3$.

\begin{table}[tbp]
 \begin{tabular}{llllll}
$m$    & $d$   & Upper bound of $\dim\nu(x)$  & Upper bound of $\dim\nu_{\pi}(x)$  & Prop \ref{prop:shrink} & Depth of Norm \\
       &       & bounded below by  & bounded above by  & holds  & Searched \\
\hline
 2 &  3 &  1.261859507 &  1.261859507 & false &  10\\
 3 &  3 &  1.133544891 &  1.077324384 & true &  1\\
 4 &  3 &  1.058745493 &  1.049820435 & true &  2\\
 5 &  3 &  1.027566600 &  1.025209036 & true &  3\\
 6 &  3 &  1.014334772 &  1.011259593 & true &  2\\
 7 &  3 &  1.006057727 &  1.005988453 & true &  3\\
 8 &  3 &  1.003425327 &  1.002515549 & true &  3\\
 9 &  3 &  1.001330543 &  1.001153367 & true &  4\\
 10 &  3 &  1.000793712 &  1.000649489 & true &  3\\
 3 &  4 &  1.500000000 &  1.235839584 & true &  1\\
 4 &  4 &  1.321490682 &  1.166666667 & true &  1\\
 5 &  4 &  1.207518750 &  1.084691151 & true &  1\\
 6 &  4 &  1.132742274 &  1.075965367 & true &  2\\
 7 &  4 &  1.096322539 &  1.057813726 & true &  2\\
 8 &  4 &  1.060623161 &  1.035146820 & true &  2\\
 9 &  4 &  1.046554702 &  1.031111632 & true &  2\\
 10 &  4 &  1.028706359 &  1.027211789 & true &  2\\
 4 &  5 &  1.722706233 &  1.292029675 & true &  1\\
 5 &  5 &  1.515580565 &  1.188044511 & true &  1\\
 6 &  5 &  1.374997393 &  1.138225274 & true &  1\\
 7 &  5 &  1.273287000 &  1.162457996 & true &  1\\
 8 &  5 &  1.218846960 &  1.147466087 & true &  1\\
 9 &  5 &  1.160225830 &  1.139336787 & true &  1\\
 10 &  5 &  1.134263284 &  1.089576553 & true &  1\\
 5 &  6 &  1.934264036 &  1.317946834 & true &  1\\
 6 &  6 &  1.707969651 &  1.239189644 & true &  1\\
 7 &  6 &  1.547411229 &  1.167218645 & true &  1\\
 8 &  6 &  1.426500033 &  1.176448870 & true &  1\\
 9 &  6 &  1.357136478 &  1.171217752 & true &  1\\
 10 &  6 &  1.281621364 &  1.179936574 & true &  1\\
 6 &  7 &  2.137243123 &  1.366428213 & true &  1\\
 7 &  7 &  1.896901635 &  1.274709522 & true &  1\\
 8 &  7 &  1.720507429 &  1.212928829 & true &  1\\
 9 &  7 &  1.584821135 &  1.193520824 & true &  1\\
 10 &  7 &  1.502709990 &  1.190060439 & true &  1\\
 7 &  8 &  2.333333333 &  1.403398681 & true &  1\\
 8 &  8 &  2.082091702 &  1.322408586 & true &  1\\
 9 &  8 &  1.892690635 &  1.250208526 & true &  1\\
 10 &  8 &  1.745082630 &  1.215234300 & true &  1\\
 8 &  9 &  2.523719013 &  1.449978323 & true &  1\\
 9 &  9 &  2.263677476 &  1.362744592 & true &  1\\
 10 &  9 &  2.063319598 &  1.295581060 & true &  1\\
 9 &  10 &  2.709269961 &  1.490378946 & true &  1\\
 10 &  10 &  2.441914915 &  1.409160390 & true &  1
\end{tabular}
\caption{Table caption here} 
\label{tab:conj}
\end{table}

\end{document}